\documentclass[review]{elsarticle}
\journal{Advances in Mathematics}
\usepackage{fullpage}
\usepackage{amsfonts,amsmath,amscd,amssymb,amsthm}
\usepackage{dsfont,mathtools}
\usepackage{pgf,tikz,tikz-cd}
\usetikzlibrary{arrows,positioning}
\usepackage{xspace}
\usepackage{latexsym}
\usepackage[all]{xy}
\usepackage{pifont}
\newcommand{\ccr}{\text{\ding{64}}}

\newtheorem{theorem}{Theorem}[section]
\newtheorem{lemma}[theorem]{Lemma}
\newtheorem{prop}[theorem]{Proposition}
\newtheorem{con}{Conjecture}
\newtheorem{cor}[theorem]{Corollary}
\newcommand{\Mod}{{\mathrm M}{\mathrm o}{\mathrm d}}

\newcommand{\End}{{\mathrm E}{\mathrm n}{\mathrm d}}

\newcommand{\Rad}{{\mathrm R}{\mathrm a}{\mathrm d}}
\newcommand{\Lie}{{\mathrm L}{\mathrm i}{\mathrm e}}
\newcommand{\Cox}{{\mathrm C}{\mathrm o}{\mathrm x}{\mathrm e}{\mathrm t}{\mathrm e}{\mathrm r}}

\newcommand{\Aut}{{\mathrm A}{\mathrm u}{\mathrm t}}
\newcommand{\tAut}{2\mbox{-}\Aut}
\newcommand{\wAut}{\widetilde{\tAut}}
\newcommand{\sAut}{{\mathcal A}{ut}}
\newcommand{\tsAut}{2\mbox{-}\sAut}

\newcommand{\tVec}{2\mbox{-}{\mathrm V}{\mathrm e}{\mathrm c}{\mathrm t}}
\newcommand{\tVecc}{{\mathit 2}\mbox{-}{\mathcal V}{\mathit e}{\mathit c}{\mathit t}}
\newcommand{\tMd}{{\mathit 2}\mbox{-}{\mathcal M}{\mathit o}{\mathit d}^\bK}
\newcommand{\tV}{\tVec^\bK}
\newcommand{\tVc}{\tVecc^\bK}
\newcommand{\PGL}{{\mathrm P}{\mathrm G}{\mathrm L}}
\newcommand{\GL}{{\mathrm G}{\mathrm L}}

\newcommand{\Ind}{{\mathrm I}{\mathrm n}{\mathrm d}}
\newcommand{\Coi}{{\mathrm C}{\mathrm o}{\mathrm i}{\mathrm n}{\mathrm d}}

\newcommand{\Rep}{{\mathrm R}{\mathrm e}{\mathrm p}}

\newcommand{\tRep}{2\mbox{-}\Rep}

\newcommand{\coker}{{\mathrm c}{\mathrm o}{\mathrm k}{\mathrm e}{\mathrm r}}

\newcommand{\sA}{\mathcal{A}}
\newcommand{\sS}{\mathcal{S}}

\newcommand{\sD}{\mathcal{D}}
\newcommand{\sE}{\mathcal{E}}

\newcommand{\sC}{\mathcal{C}}
\newcommand{\sG}{\mathcal{G}}
\newcommand{\sI}{\mathcal{I}}
\newcommand{\sH}{\mathcal{H}}
\newcommand{\sK}{\mathcal{K}}

\newcommand{\sR}{\mathcal{R}}
\newcommand{\sX}{\mathcal{X}}
\newcommand{\mC}{\mathfrak{C}}

\newcommand{\mX}{\mathfrak{X}}

\newcommand{\bA}{{\mathbb A}}
\newcommand{\bB}{{\mathbb B}}
\newcommand{\bC}{{\mathbb C}}
\newcommand{\bK}{{\mathbb K}}

\newcommand{\bG}{{\mathbb G}}

\newcommand{\bN}{{\mathbb N}}

\newcommand{\bX}{{\mathbb X}}
\newcommand{\bZ}{{\mathbb Z}}

\newcommand{\bT}{{\mathbb T}}

\newcommand{\Tr}{\bT{\mathrm r}}

\newcommand{\BP}{{\mathbb B}^\Phi}
\newcommand{\BA}{{\mathbb B}_\bA}

\newcommand{\va}{\ensuremath{\mathbf{a}}\xspace}
\newcommand{\vb}{\ensuremath{\mathbf{b}}\xspace}

\newcommand{\vf}{\ensuremath{\mathbf{f}}\xspace}
\newcommand{\vg}{\ensuremath{\mathbf{g}}\xspace}
\newcommand{\vh}{\ensuremath{\mathbf{h}}\xspace}
\newcommand{\vi}{\ensuremath{\mathbf{i}}\xspace}

\newcommand{\vs}{\ensuremath{\mathbf{s}}\xspace}
\newcommand{\vt}{\ensuremath{\mathbf{t}}\xspace}

\newcommand{\vx}{\ensuremath{\mathbf{x}}\xspace}
\newcommand{\vy}{\ensuremath{\mathbf{y}}\xspace}

\newcommand{\fg}{{\mathfrak{g}}}

\newenvironment{mylist}{\begin{list}{}{
\setlength{\itemsep}{0mm}
\setlength{\parskip}{0mm}
\setlength{\topsep}{1mm}
\setlength{\parsep}{0mm}
\setlength{\itemsep}{0mm}
\setlength{\labelwidth}{6mm}
\setlength{\labelsep}{3mm}
\setlength{\itemindent}{0mm}
\setlength{\leftmargin}{9mm}
\setlength{\listparindent}{6mm}
}}{\end{list}}

\begin{document}

\begin{frontmatter}

%\selectlanguage{english}

\title{2-Groups, 2-Characters, and Burnside Rings\tnoteref{th}}
\author{Dmitriy Rumynin}
\ead{D.Rumynin@warwick.ac.uk}
\address{Department of Mathematics, University of Warwick, Coventry, CV4 7AL, UK\newline
\hspace*{0.31cm}  Associated member of Laboratory of Algebraic Geometry, National
Research University Higher School of Economics, Russia}
\author{Alex Wendland}
\ead{A.P.Wendland@warwick.ac.uk}
\address{Department of Mathematics, University of Warwick, Coventry, CV4 7AL, UK}
\tnotetext[th]{The research was partially supported by the Russian Academic Excellence Project `5--100' and by Leverhulme Foundation.
%The research was partially supported by Leverhulme Foundation.
  We would like to thank Anna Pusk\'{a}s for help with latex packages.
  We would like to thank Derek Holt
  and Bruce Westbury for interesting discussions and valuable suggestions.
  We are indebted to Joao Faria Martins and the referee for pointing
  out the issues with earlier versions of the paper.}

\begin{keyword}
2-group, crossed module, Burnside ring
\MSC[2010] Primary 18D05 Secondary 19A22
\end{keyword}

\begin{abstract}
We study 2-representations, i.e., actions of
2-groups on 2-vector spaces.
Our main focus is character theory for 2-representations.
To this end we employ the technique of extended Burnside rings.
Our main theorem is that the Ganter-Kapranov 2-character
is a particular mark homomorphism of the Burnside ring.
As an application we give a new proof of Osorno's formula for
the Ganter-Kapranov 2-character of a finite group.
\end{abstract}

\end{frontmatter}

This paper
develops the theory of 2-representations 
and their associated 2-characters. 
A finite group $G$
can act on a 2-vector space, 
in a sense of
Kapranov and Voevodsky \cite{KV94}.
Ganter and Kapranov use 2-traces to associate a 2-character to such a
2-representation \cite{GK08}.
Osorno gives an explicit formula for this 2-character in terms of cohomological data 
\cite{AO10}.

We would like to contemplate this 2-character and its formula
in this paper.
We approach it via 
generalised Burnside rings, building on the work by
Gunnells, Rose and the first author \cite{GRR11}.
Our main result is an explicit expression
for the Ganter-Kapranov character
as a mark homomorphism (Theorem~\ref{Our_Char}).

We work in a slightly bigger generality as it has little extra cost
and could prove useful:
instead of a finite group $G$ we work with
a 2-group $\sG=\widetilde{\sK}$
arising from a crossed module $\sK$ whose fundamental group
$\pi_1 (\sG)$ is finite.
Let us explain how this paper
is organised. 

In Section~\ref{s1} we set out the terminology
of 2-categories, in particular, we define
2-groups, 2-vector spaces, 2-representations and 2-modules.
At this stage one can think of a 2-representation
as a ``semisimple'' 2-module. 
Most of this chapter is well known, yet it is
essential to establish our notation and our terminology.
One important result in this section 
is Theorem~\ref{2vec_eq},
an equivalence of two realizations 
of the 2-category of 2-vector spaces. It is probably known to the experts
but we could not find it in the literature.

In Section~\ref{s2} we roll out our philosophy: a 2-module
$\Theta_A$
of $\sG$
arises from an action of the crossed module $\sK$-action on an algebra $A$
(Proposition~\ref{ext_2}).
There is a subtlety: we need to distinguish
%strict and weak 2-modules as well as
strict and weak 2-actions.
A weak 2-action
is a weak homomorphism of crossed modules from $\sK$ to
the crossed module of 2-automorphisms of $A$.
We finish the section with a Morita theorem for strict 2-modules
(Theorem~\ref{G_morita}):
the 2-modules  $\Theta_A$ and  $\Theta_B$ are equivalent
if and only if
$A$ and $B$ are $\sK$-Morita equivalent.

Section~\ref{s2a} is devoted to 2-representations.
If $\sK = (H\rightarrow G)$ then a 2-representation of $\sG$
yields a 2-cocycle for $G$ (Lemma~\ref{morita_dec_group}).
At this stage we uncover another subtlety:
the cocycle may or may not be realized on a finite dimensional
projective representation of $G$.
In the former case we call the cocycle {\em realizable}.
Realizable cocycles lead to split semisimple algebras, while
non-realizable cocycles lead to 
direct sums of full matrix algebras, some of whom
are infinite dimensional.
We call such algebras {\em semimatrix}.
A semimatrix algebra $A$ with an action of $\sK$
gives a 2-representation $\Theta_A^\circ$ of $\sG$,
equal to $\Theta_A$ if $A$ is semisimple.
We can manage not only
$\Theta_A$
but also 
$\Theta_A^\circ$
with some additional care
(cf. Corollary~\ref{morita_dec} and 
Corollary~\ref{morita_dec_1}). 
The upshot of this chapter is that weak non-unital
2-representations
come from strict $\sK$-algebras, in particular,
are equivalent to strict 2-representations.
This gives a convenient insight into a structure of
2-representations.
We finish the section
with a structure theorem for 2-representations
(Theorem~\ref{irr_rep}).

In Section~\ref{s3} 
we utilise generalised Burnside rings \cite{GRR11}
to describe the Grothendieck group of 2-representations
of $\sG$.
It is curious that the Burnside ring is slightly unusual:
two elements $\vx,\vy\in G$ define the same conjugation if $\vx\vy^{-1}$
is in the centre but may determine different pull-backs of 2-representations. 

In Section~\ref{s4}
we define 
the Ganter-Kapranov 2-character for 2-Groups. 
% $\sK =
% (H\xrightarrow{\partial}G)$ where
% $\pi_1(\sK)$ is a finite group. 
We express these 2-characters in terms of the generalised Burnside
rings,
more precisely, the Ganter-Kapranov  2-character is a certain mark homomorphism
(Theorem~\ref{Our_Char}).
This result is the main theorem of this paper.

Starting from Section~\ref{s5}, we work with a group, 
as a particular example of a 2-group.
%In other words, we study a crossed module with a trivial group $H$.
In Section~\ref{s5} we make some technical preparations, namely
we write an explicit formula for Shapiro isomorphism
on the level of cocycles (Theorem~\ref{prop1}). 
It may be known to experts 
but we could not find it in the literature.
It gives a slightly stronger version of Shapiro isomorphism:
two complexes are not only quasiisomorphic but also homotopically
equivalent.

In Section~\ref{s6} we explicitly calculate 
the Ganter-Kapranov 2-character for the groups.
Our first result (Theorem~\ref{th2})
is a formula that follows immediately from
mark homomorphisms.
The second result is a known formula by Osorno
(Theorem~\ref{Osor_Char}).
We intentionally prove the known formula as well
to demonstrate the power of our method.
We finish the section with two conjectures
suggesting how to generalise the content of this section
to the 2-groups.
%It is important to demonstrate that mark ho
%explain the connection between two formulas for 
%the Ganter-Kapranov 2-character: our new formula and a cohomological
%formula given by Osorno.
%Our formula is more general

In the final Section~\ref{s7} we restate two famous
conjectures on the level of 2-representations:
Lusztig's Conjecture about base sets for the double cells
and the 
McKay Conjecture for the number of $p^\prime$-characters of a
finite group.

\begin{center}
\section{Introduction to 2-representations}
\label{s1}
\end{center}

Let us clarify what we mean by 2-representations
of 2-groups in this paper.
They have been studied by several authors
\cite{BBFW, BL04, BaMa, Elg, GK08, AO10}.
In general, we follow the terminology of
Benabou, who distinguishes 2-categories and bicategories
\cite{Ben}.
Both structures consist of a class $\sC_0$ of 0-objects,
a category $\sC_1 (x,y)$ for each pair of 0-objects,
a unit 1-object (or 1-morphism - these are synonyms) 
$\vi_x \in \sC_1 (x,x)$ 
and composition bifunctors
%$$
%\diamond = \diamondsuit_{x,y,z}:\sC_1 (x,y)\times \sC_1 (y,z)\rightarrow \sC_1
%(x,z).
%$$
$$
\diamond = \diamondsuit_{x,y,z}:\sC_1 (x,y)\times \sC_1 (y,z)\rightarrow \sC_1
(x,z).
$$
In particular, $\sC_2 (u,v)$ is the set of 2-morphism between 
1-objects (a.k.a, 1-morphisms) $u$ and $v$ of $\sC_1 (x,y)$. 
We use two symbols for compositions.
The circle $\circ$ stands for the usual compositions of morphisms
or 2-morphism that we write right-to-left (or bottom-to-top in
globular notation).
The diamond $\diamond$ stands for the composition bifunctor in a bicategory
that we write left-to-right.
In a 2-category the compositions of 1-morphisms
are associative and unital.
Let $\sI_{x,y}$ be the identity endofunctor on $\sC_1 (x,y)$.
In a bicategory associativity is a family of natural isomorphisms of
trifunctors
$$
\mbox{Ass}_{w,x,y,z}: 
\diamondsuit_{w,x,z}\circ(\sI_{w,x}\times \diamondsuit_{x,y,z})
\Rightarrow
\diamondsuit_{w,y,z}\circ(\diamondsuit_{w,x,y}\times\sI_{y,z})
,$$
$$ 
\diamondsuit_{w,x,z}\circ(\sI_{w,x}  \times \diamondsuit_{x,y,z})
, \ 
\diamondsuit_{w,y,z}\circ(\diamondsuit_{w,x,y} \times \sI_{y,z}) 
:
\sC_1 (w,x)\times \sC_1 (x,y)\times \sC_1
(y,z)\rightarrow \sC_1 (w,z)
$$
such that 
the pentagon diagrams are commutative. % for each five 0-objects from $\sC_0$.
%$$
%\xymatrix{
%(\rho(g) \circ \rho(h)) \circ \rho(k) \ar@{=>}[rr]^{\mbox{ass}}  \ar@{=>}[d]^{\phi_{g,h} \circ% I_{\rho(k)}} & & \rho(g) \circ (\rho(h) \circ \rho(k)) \ar@{=>}[rr]^{I_{\rho(g)} \circ
%  \phi_{h,k}}  & & \rho(g) \circ\rho(hk) \ar@{=>}[d]^{\phi_{g,hk}}\\
%\rho(gh) \circ \rho(k) \ar@{=>}[rrrr]^{\phi_{gh,k}} & & & & \rho(ghk)\\
%}
%$$
Similarly, 
unitality in a bicategory is two families of natural isomorphisms of
functors
$$
\mbox{RUn}_{x,y}: \diamondsuit_{x,y,y} (\ \ , \vi_y) \Rightarrow
\sI_{x,y}, \ \ 
\mbox{LUn}_{x,y}: \diamondsuit_{x,x,y} (\vi_x , \ \ ) \Rightarrow
\sI_{x,y}
$$
such that 
the triangle diagrams are commutative. % for each three 0-objects from $\sC_0$.
%$$
%\xymatrix{
%\rho(1) \circ \rho(g) \ar@{=>}[rr]^{\phi_{1}\circ I_{\rho(g)}} \ar@{=>}[d]^{\phi_{1,g}} & & I_{[n]} \circ \rho(g) \ar@{=>}[dll]^{\cong}\\
%\rho(g)\\
%}
%\ \ \ \ \ \ \ \ \ 
%\xymatrix{
%\rho(g) \circ \rho(1) \ar@{=>}[rr]^{I_{\rho(g)}\circ\phi_{1}}
%\ar@{=>}[d]^{\phi_{g,1}} & & \rho(g)\circ I_{[n]} \ar@{=>}[dll]^{\cong}\\
%\rho(g)\\
%}
%$$
A bicategory is {\em small} if 
it consists of sets on each level: $\sC_0$ is a set
and all categories 
$\sC_1 (x,y)$ are small.
We will have both 2-categories and bicategories in this paper.
%We will use globular notation to work with them.

{\em A (weak) 2-functor} (between bicategories)
$F : \sC \rightarrow \sD$
is a datum
$$
F^0 : \sC_0 \rightarrow \sD_0, \ \ 
F^1_{x,y} : \sC_1 (x,y) \rightarrow \sD_1 (F^0(x),F^0(y)), \ \ 
F^2_{x} :\vi_{F^0(x)} \Rightarrow F^1_{x,x} (\vi_x), 
$$
$$
F^2_{x,y,z} : 
\diamondsuit_{F^0(x),F^0(y),F^0(z)} \circ (F^1_{x,y}\times F^1_{y,z} )
\Rightarrow 
F^1_{x,z}\circ \diamondsuit_{x,y,z}
$$
where $F^0$ is a function, 
$F^1_{x,y}$ is a family of functors,
$F^2_{x}$ is a family of 2-isomorphisms (i.e., 2-morphisms
that are isomorphisms),
and $F^2_{x,y,z}$ is a family
of compatibility conditions, 
natural isomorphisms of 
bifunctors from $\sC_1 (x,y) \times \sC_1 (y,z)$
to 
$\sD_1 (F^0(x),F^0(z))$.
The requirement is that the hexagon and the square diagrams are
commutative (cf. a definition of monoidal functor).
The hexagon diagram ensures that the two possible natural
transformations 
$$
nt_1, nt_2: \diamondsuit_{F^0(w),F^0(y),F^0(z)} \circ
(\diamondsuit_{F^0(w),F^0(x),F^0(y)} \times \mbox{Id}) \circ
(F^1_{w,x}\times F^1_{x,y}\times F^1_{y,z} )
\Rightarrow 
F^1_{w,z}\circ \diamondsuit_{w,x,z} \circ (\mbox{Id} \times \diamondsuit_{x,y,z}) 
$$
of trifunctors
from $\sC_1 (w,x)\times \sC_1 (x,y) \times \sC_1 (y,z)$
to
$\sD_1 (F^0(w),F^0(z))$ 
are equal.
Similarly, the two square diagrams equate possible natural
transformations 
$$
\diamondsuit_{F^0(x),F^0(y),F^0(y)} \circ
(F^1_{x,y}\times \vi_{F^0(y)} )
\Rightarrow
F^1_{x,y} \circ 
\diamondsuit_{x,y,y} \circ
(\mbox{Id} \times \vi_{y} ), \ \
\diamondsuit_{F^0(x),F^0(x),F^0(y)} \circ
(\vi_{F^0(x)} \times F^1_{x,y})
\Rightarrow
F^1_{x,y} \circ 
\diamondsuit_{x,x,y} \circ
(\vi_{x}\times \mbox{Id})
$$
of functors
from $\sC_1 (x,y)$
to
$\sD_1 (F^0(x),F^0(y))$.  
We say that $F$ is {\em unital} if each $F^2_{x}$ is the identity.
We say that $F$ is {\em strict} if all $F^2_{x}$ and
$F^2_{x,y,z}$ are the identities.

{\em A
natural 2-transformation} 
$\psi : F \Rightarrow G$
between 2-functors
$F,G : \sC \rightarrow \sD$
is a datum
$$
\psi^1_{x} \in \sD_1 (F^0(x),G^0(x)), \ \ 
\psi^2_{x,y}:  
F^1_{x,y}  \diamond \psi^1_{y} 
\Rightarrow  
\psi^1_{x} \diamond G^1_{x,y}
$$
where $\psi^1$ is a family of 1-morphisms in $\sD$
and $\psi^2$ is a family of
natural transformations of functors \linebreak 
$\sC_1 (x,y) \rightarrow \sD_1 (F^0(x),G^0(y))$
satisfying two coherence conditions
that we will spell out following Barrett and Mackaay \cite{BaMa}.
The first condition is that 
$$
\psi_{\va,\vb}=\psi_{\va}\circ\psi_{\vb}
$$
where these three 2-morphisms are defined
for any pair of 1-objects 
$\va\in \sC_1 (x,y)$, $\vb\in \sC_1 (y,z)$
by
$$
\psi_{\va,\vb}:
F^1_{x,y}(\va) \diamond F^1_{y,z}(\vb)  \diamond \psi^1_{z}
\xrightarrow{F^2_{x,y,z}(\va,\vb)}
F^1_{x,z}(\va\diamond\vb) \diamond \psi^1_{z}  
\xrightarrow{\psi^2_{x,z}(\va\diamond\vb)}
\psi^1_{x}\diamond G^1_{x,z}(\va\diamond\vb) 
\xrightarrow{G^2_{x,y,z}(\va,\vb)^{-1}}
\psi^1_{x}\diamond G^1_{x,y}(\va)\diamond G^1_{y,z}(\vb),
$$
$$
\psi_{\vb}:
F^1_{x,y}(\va) \diamond F^1_{y,z}(\vb)  \diamond \psi^1_{z}
\xrightarrow{\mbox{Id}_{F^1_{x,y}(\va)}\diamond\psi^2_{y,z}(\vb)}
F^1_{x,y}(\va) \diamond \psi^1_{y}\diamond G^1_{y,z}(\vb),
$$
$$
\psi_{\va}:
F^1_{x,y}(\va) \diamond \psi^1_{y}\diamond G^1_{y,z}(\vb)
\xrightarrow{\psi^2_{x,y}(\va)\diamond\mbox{Id}_{G^1_{y,z}(\vb)}}
\psi^1_{x}\diamond G^1_{x,y}(\va)\diamond G^1_{y,z}(\vb).
$$
The second condition is that the 2-morphism
$$
\psi^1_{x}
\xrightarrow{LUn^{-1}}
\vi_{F^0(x)}
\diamond 
\psi^1_{x} 
\xrightarrow{F^{1\; -1}_{x}}
F^1_{x,x}  (\vi_x)
\diamond 
\psi^1_{x} 
\xrightarrow{\psi^2_{x,x}(\vi_x)}
\psi^1_{x}
\diamond
G^1_{x,x} (\vi_x)
\xrightarrow{G^{1}_{x}}
\psi^1_{x}
\diamond
\vi_{G^0(x)}
\xrightarrow{RUn}
\psi^1_{x}
%\psi^1_{x}
%\xrightarrow{LUn_{F^0(x),G^0(x)}^{-1}}
%\psi^1_{x} \diamond \vi_{F^0(x)}
%\xrightarrow{F^{1\; -1}_{x}}
%\psi^1_{x} \diamond F^1_{x,x}  (\vi_x) 
%\xrightarrow{\psi^2_{x,x}(\vi_x)}
%G^1_{x,x} (\vi_x)
%\diamond \psi^1_{x}
%\xrightarrow{G^{1}_{x}}
%\vi_{G^0(x)}\diamond \psi^1_{x}
%\xrightarrow{RUn_{F^0(x),G^0(x)}}
%\psi^1_{x}
$$
is equal to $\mbox{Id}_{\psi^1_{x}}$ 
for each $x\in \sC_0$. 
If all $\psi^1_{x}$ 
are 1-equivalences %(i.e., quasiinvertible 1-morphisms) 
and all
$\psi^2_{x,y}$ 
are natural isomorphisms,
we say that $\psi$ is
{\em a natural 2-isomorphism}.
Recall that
a 1-morphism $\va\in \sC_1(x,y)$
is a 1-equivalence if it is quasiinvertible,
i.e., there exists $\va^{-1}\in \sC_1(y,x)$ such that 
$\va\diamond\va^{-1}$ is isomorphic to $\vi_x$
and 
$\va^{-1}\diamond\va$ is isomorphic to $\vi_y$.
It is a 1-isomorphism if it is invertible, i.e., 
there exists $\va^{-1}\in \sC_1(y,x)$ such that 
$\va\diamond\va^{-1}=\vi_x$
and 
$\va^{-1}\diamond\va=\vi_y$.

Two bicategories $\sC$ and $\sD$ are {\em equivalent}
if there exist {\em a 2-equivalence}, 
i.e., a 2-functor $F: \sC\rightarrow\sD$
that admits a quasiinverse  2-functor $G: \sD\rightarrow\sC$
with natural 2-isomorphisms
$FG \xrightarrow{\cong} \sI d_\sD$ and
$GF \xrightarrow{\cong} \sI d_\sC$.

{\em A 2-group} $\sG$
is a bicategory 
such that  $\sG_0$
is a one-element set, 
each 2-morphism is a 2-isomorphism, and
each 1-morphism is a 1-equivalence. 
{\em A strict 2-group} $\sG$
is a 2-category 
such that  $\sG_0$
is a one-element set,
each 2-morphism is a 2-isomorphism, and
each 1-morphism is a 1-isomorphism.

{\em A homomorphism} of 2-groups
is a 2-functor $F: \sG \rightarrow \sH$.
It could be strict or unital, if such is the 2-functor.
The 2-groups $\sG$ and $\sH$ are {\em equivalent} 
if there is a 2-equivalence $F: \sG \rightarrow \sH$.

Strict small 2-groups are equivalent to 2-groups
that arise from crossed modules \cite{BBFW}. 
Let $\sK = (H\xrightarrow{\partial}G)$ be a crossed module. 
We assume that $G$ acts on $H$ on the left:
$$
\vg : h\mapsto \,^\vg h. % = ghg^{-1}.
$$
Its
fundamental groups are $\pi_2 (\sK) = \ker (\partial)$ and
$\pi_1 (\sK) = \coker (\partial)$. 
The crossed module $\sK$ determines a small strict 2-group $\widetilde{\sK}$:
\begin{mylist}
\item{1.}
$\widetilde{\sK}_0$ is a one point set
$\{\star\}$.
\item{2.}
The objects of the category $\widetilde{\sK}_1 (\star, \star)$ is the group $G$.
\item{3.}
For a pair $\vg_1,\vg_2 \in G$ the 2-morphisms 
in the category $\widetilde{\sK}_1 (\star, \star)$
are 
$$\widetilde{\sK}_2 (\vg_1,\vg_2) :=
\{
\vg_1\stackrel{h}{\Longrightarrow}\vg_2
\} \; = \; \{
%\begin{tikzcd}
%A \arrow[l, bend left=50, ""{name=U, below, draw=red}]
%\arrow[r, bend right=50, ""{name=D, draw=red}]
%& B
%\arrow[Rightarrow, from=U, to=D]
%\end{tikzcd}
%\begin{tikzcd}
%g_2\arrow[bend left]{l}\arrow[bend right]{l}&g_1
%\end{tikzcd}
\xymatrix{
\star
   \ar@/^2pc/[rr]_{\quad}^{\vg_2}="2"
   \ar@/_2pc/[rr]_{\vg_1}="1"
&& \star
%   \ar@/^2pc/[ll]_{\quad}^{\vg_1}="1"
%  \ar@/_2pc/[ll]_{\vg_2}="2"
\ar@{}"1";"2"|(.2){\,}="7"
\ar@{}"1";"2"|(.8){\,}="8"
\ar@{=>}"7" ;"8"^{h}
} 
%\endxy
\,\mid\
h\in H,\ \vg_2 = \partial (h) \vg_1 \}
,$$
while their composition is determined by the product in $H$:
$$
[\vg_2\stackrel{h_2}{\Longrightarrow}\vg_3]
  \circ
      [\vg_1\stackrel{h_1}{\Longrightarrow}\vg_2]
      =
\xymatrix{
\star
   \ar@/^2pc/[rr]_{\quad}^{\vg_3}="3"
    \ar[rr]|{\vg_2}="2"
   \ar@/_2pc/[rr]_{\vg_1}="1"
&& \star
%   \ar@/^2pc/[ll]_{\quad}^{\vg_1}="1"
%   \ar[ll]|{\vg_2}="2"
%   \ar@/_2pc/[ll]_{\vg_3}="3"
\ar@{}"1";"2"|(.2){\,}="7"
\ar@{}"1";"2"|(.8){\,}="8"
\ar@{=>}"7" ;"8"^{h_1}
\ar@{}"2";"3"|(.2){\,}="9"
\ar@{}"2";"3"|(.8){\,}="10"
\ar@{=>}"9" ;"10"^{h_2}
} 
%\endxy
=
\xymatrix{
\star
   \ar@/^2pc/[rr]_{\quad}^{\vg_3}="2"
   \ar@/_2pc/[rr]_{\vg_1}="1"
&& \star
%   \ar@/^2pc/[ll]_{\quad}^{\vg_1}="1"
%   \ar@/_2pc/[ll]_{\vg_3}="2"
\ar@{}"1";"2"|(.2){\,}="7"
\ar@{}"1";"2"|(.8){\,}="8"
\ar@{=>}"7" ;"8"^{h_2h_1}
} .
$$
\item{4.}
The composition bifunctor comes from the action and the multiplication:
$$
[\vf\stackrel{h}{\Longrightarrow}\vg]
  \diamond
      [\vf_1\stackrel{h_1}{\Longrightarrow}\vg_1]
      =
\xymatrix{
\star
   \ar@/^2pc/[rr]_{\quad}^{\vg}="1"
   \ar@/_2pc/[rr]_{\vf}="2"
&& \star
%   \ar@/^2pc/[ll]_{\quad}^{\vf}="1"
%   \ar@/_2pc/[ll]_{\vg}="2"
   \ar@/^2pc/[rr]_{\quad}^{\vg_1}="3"
   \ar@/_2pc/[rr]_{\vf_1}="4"
\ar@{}"1";"2"|(.2){\,}="7"
\ar@{}"1";"2"|(.8){\,}="8"
\ar@{=>}"8" ;"7"^{h}
&& \star
%   \ar@/^2pc/[ll]_{\quad}^{\vf_1}="1"
%   \ar@/_2pc/[ll]_{\vg_1}="2"
\ar@{}"4";"3"|(.2){\,}="7"
\ar@{}"4";"3"|(.8){\,}="6"
\ar@{=>}"7" ;"6"^{h_1}
} 
=
\xymatrix{
\star
   \ar@/^2pc/[rr]_{\quad}^{\vg\,\vg_1}="1"
   \ar@/_2pc/[rr]_{\vf\,\vf_1}="2"
&& \star
%   \ar@/^2pc/[ll]_{\quad}^{\vf\vf_1}="1"
%   \ar@/_2pc/[ll]_{\vg\vg_1}="2"
\ar@{}"1";"2"|(.2){\,}="7"
\ar@{}"1";"2"|(.8){\,}="8"
\ar@{=>}"8" ;"7"^{h\,^\vf h_1}
} . 
$$
\item{5.}
$\vi_{\star} = 1_G$.
\end{mylist}
Let us check that the composition works:
%The axioms of the crossed module make it work:
$\partial(h \,^\vf h_1)\vf\vf_1=
\partial(h)\partial(\,^\vf h_1)\vf\vf_1=
\partial(h)\vf\partial(h_1)\vf^{-1}\vf\vf_1= \vg\vg_1
$.

We need the bicategory of 
Kapranov-Voevodsky finite-dimensional 2-vector spaces
over a field $\bK$ \cite{KV94}.
Let us describe 
$\tV$, a version of 2-vector spaces we find particularly useful:
\begin{mylist}
\item{1.}
The 0-objects are the natural numbers: 
%$\tV_0 = \{ [n]\mid n\in\bN\}$
$\tV_0 = \bN$
(we agree that $0\in \bN$).
\item{2.}
The categories $\tV_1(n,0)$ and $\tV_1(0,m)$ are 
the trivial categories with one object.
\item{3.}
For any two positive numbers $n$, $m$ 
the objects of the
category $\tV_1(n,m)$ are $n \times m$-matrices $(V_{i,j})$
of finite-dimensional
$\bK$-vector spaces. 
\item{4.} The 2-morphisms 
in the category $\tV_1(n,m)$ are
$$
\tV_2 ((V_{i,j}),(W_{i,j}))
=\{
\xymatrix{
n
   \ar@/^2pc/[rr]_{\quad}^{W_{i,j}}="2"
   \ar@/_2pc/[rr]_{V_{i,j}}="1"
&& m
%   \ar@/^2pc/[ll]_{\quad}^{(V_{i,j})}="1"
%   \ar@/_2pc/[ll]_{(W_{i,j})}="2"
\ar@{}"1";"2"|(.2){\,}="7"
\ar@{}"1";"2"|(.8){\,}="8"
\ar@{=>}"7" ;"8"^{\varphi_{i,j}}
} 
\,\mid\,
\varphi_{i,j} : V_{i,j} \rightarrow W_{i,j}
\mbox{ is a linear map}
\}
$$
are $n \times m$-matrices $(\varphi_{i,j})$
of linear maps
$\varphi_{i,j} : V_{i,j} \rightarrow W_{i,j}$.
The composition of 2-morphisms 
is the composition of these linear maps:
$$
(\varphi_{i,j})
\circ
(\psi_{i,j})
=
\xymatrixcolsep{2pc}
\xymatrix{
n
   \ar@/^3pc/[rr]_{\quad}^{W_{i,j}}="1"
   \ar@/_3pc/[rr]_{U_{i,j}}="3"
   \ar[rr]|{V_{i,j}}="2"
&& m
%   \ar@/^3pc/[ll]_{\quad}^{U_{i,j}}="1"
%   \ar[ll]|{V_{i,j}}="2"
%   \ar@/_3pc/[ll]_{W_{i,j}}="3"
\ar@{}"2";"1"|(.2){\,}="7"
\ar@{}"2";"1"|(.75){\,}="8"
\ar@{=>}"7" ;"8"^{\varphi_{i,j}}
\ar@{}"3";"2"|(.25){\,}="9"
\ar@{}"3";"2"|(.8){\,}="10"
\ar@{=>}"9" ;"10"^{\psi_{i,j}}
} 
%\endxy
=
\xymatrixcolsep{3pc}
\xymatrix{
n
   \ar@/^2pc/[rr]_{\quad}^{W_{i,j}}="1"
   \ar@/_2pc/[rr]_{U_{i,j}}="2"
&& m
%   \ar@/^2pc/[ll]_{\quad}^{U_{i,j}}="1"
%   \ar@/_2pc/[ll]_{W_{i,j}}="2"
\ar@{}"2";"1"|(.2){\,}="7"
\ar@{}"2";"1"|(.8){\,}="8"
\ar@{=>}"7" ;"8"^{\varphi_{i,j}\psi_{i,j}}
} .
%\endxy .
$$
\item{5.}
The composition bifunctor
$\diamondsuit_{n,p,m} : \tV_1(n,p) \times \tV_1(p,m) \rightarrow \tV_1(n,m)$
is
$$
(\varphi_{i,j})
\diamond
(\psi_{i,j})
=
\xymatrix{
n
   \ar@/^2pc/[rr]^{W_{i,k}}="2"
   \ar@/_2pc/[rr]_{\quad}_{V_{i,k}}="1"
&& p
%   \ar@/^2pc/[ll]_{\quad}^{V_{i,j}}="1"
%   \ar@/_2pc/[ll]_{W_{i,j}}="2"
   \ar@/^2pc/[rr]_{\quad}^{A_{k,j}}="3"
   \ar@/_2pc/[rr]_{B_{k,j}}="4"
\ar@{}"1";"2"|(.2){\,}="5"
\ar@{}"1";"2"|(.8){\,}="6"
\ar@{=>}"5" ;"6"^{\varphi_{i,k}}
&& m
%   \ar@/^2pc/[ll]_{\quad}^{A_{i,j}}="1"
%   \ar@/_2pc/[ll]_{B_{i,j}}="2"
\ar@{}"4";"3"|(.2){\,}="7"
\ar@{}"4";"3"|(.8){\,}="8"
\ar@{=>}"7" ;"8"^{\psi_{k,j}}
} 
=
\xymatrixcolsep{4pc}
\xymatrix{
n
   \ar@/_2pc/[rr]_{\oplus_k V_{i,k} \otimes A_{k,j} }="1"
   \ar@/^2pc/[rr]_{\quad}^{\oplus_k W_{i,k} \otimes B_{k,j} }="2"
&& m
%   \ar@/^2pc/[ll]_{\quad}^{\oplus_k V_{i,k} \otimes A_{k,j} }="1"
%   \ar@/_2pc/[ll]_{\oplus_k W_{i,k} \otimes B_{k,j} }="2"
\ar@{}"1";"2"|(.2){\,}="7"
\ar@{}"1";"2"|(.8){\,}="8"
\ar@{=>}"7" ;"8"^{\oplus_k \varphi_{i,k} \otimes \psi_{k,j}}
} .
$$
\item{6.}
The associativity constraint is non-trivial: it arises from
the associativity of tensor products of vector spaces.
\item{7.}
$\vi_{n} = (V_{i,j})$ where each $V_{i,i}$ is the field $\bK$
and $V_{i,j}=0$ if $i\neq j$.
\item{8.}
The unitality constraints are non-trivial: they arise from
the isomorphisms $\bK\otimes V \cong V \cong V \otimes \bK$.
%
%$$
%\big( (V_{i,j})_{p\times m}, (W_{i,j})_{m\times n}) \big)
%\mapsto ( \bigoplus_k W_{i,k} \otimes V_{k,j} )_{p\times n}
%, \ \ \  
%\big( (\varphi_{i,j}),  (\psi_{i,j})\big) 
%\mapsto ( \bigoplus_k \psi_{i,k} \otimes \varphi_{k,j} )
%$$
%on objects and on morphisms correspondingly.
\end{mylist}

There are other versions of this bicategory in the literature.
There is a ``more skeletal'' version where one uses standard
vector spaces $\bK^n$ instead of all finite dimensional vector spaces
as the matrix entries. We do not find this version very useful
but a ``bigger'' version $\tVc$ is convenient.
In particular, $\tVc$ is a 2-category.
Let $\sA^n$ be the category of finite dimensional representations
of the semisimple commutative algebra $\bK^n$. It is a semisimple
$\bK$-linear abelian category. The category $\sA^1$ is the category of
finite dimensional vector spaces. 
It has a monoidal structure via the
tensor product of vector spaces.
The category $\sA^n$ is an $\sA^1$-module category:
the theory of module categories 
is developed by
Ostrik \cite{Ost}. 
Without repeating it here
we spell out a few useful facts.
The main feature of a module category $\sC$ is 
an action bifunctor 
$$
\boxtimes : \sC \times \sA^1 \rightarrow\sC.
$$
Let $\sC=A\mbox{-}\Mod$ for any associative $\bK$-algebra $A$,
not necessarily $\bK^n$. 
The action bifunctor comes from the tensor product:
if $M$ is an $A$-module, $V$ is a vector space,
then $M\boxtimes V$ is an $A$-module,
defined as a vector space $M\otimes_\bK V$ with the $A$-action
on the first factor.

A module functor from $\sC_1$ to $\sC_2$ 
is a functor $F:\sC_1 \rightarrow\sC_2$
together with action intertwiners, certain functorial morphisms (i.e., given by a natural
transformation of bifunctors)
$$F_{M,X} : F(M \boxtimes X) \longrightarrow F(M) \boxtimes X \ 
\mbox{ for all } \  
X \in \sA^1 ,\; M \in \sC,
$$
satisfying the standard natural conditions (pentagon and triangle
commutativity) \cite[Def 2.7]{Ost}.
A module functor $F$ is {\em strong} if all $F_{M,X}$ are isomorphisms.
A module functor $F$ is {\em strict} if all $F_{M,X}$ are equalities.

A module natural transformation $\psi: F \Rightarrow G$ 
between two module functors $F,G:\sC_1 \rightarrow\sC_2$
is a natural transformation $\psi$ such that the
following diagram
is commutative for all $M\in\sC_1$, $X\in \sA^1$:
$$
\begin{CD}
F (M\boxtimes X) @>{F_{M,X}}>>     F(M)\boxtimes X \\
@V{{\psi}_{M\boxtimes X}}VV              @VV{\psi_M \boxtimes \,\mbox{Id}_X}V \\
G(M\boxtimes X)  @>{G_{M,X}}>> G(M)\boxtimes X.
\end{CD}
$$
%%%%is commutative for all $M\in\sC_1$, $X\in \sA^1$.
%Notice that the second possible
%\begin{CD}
%F (M) \boxtimes \vi_{\sA^1} @>>>     F(M) \\
%@VV{{\psi}_{M}\boxtimes \,\mbox{Id}_{\vi}}V              @VV{\psi_M}V \\
%G(M) \boxtimes \vi_{\sA^1}  @>>> G(M)
%\end{CD}
%$$
%the horizontal rows in the second
%diagram are the unital constraints of the module categories
%and $\vi_{\sA^1} =\bK$ is the unital object. 

Let us give some examples of module functors and module natural transformations.
If $\,_AP_B$ is a
bimodule, the corresponding functor 
$F: B\mbox{-}\Mod\rightarrow A\mbox{-}\Mod$,
$F(M)\coloneqq P\otimes_B M$ admits a canonical strong module functor
structure:
$$F_{M,X}: F(M\boxtimes X) = P\otimes_B (M \otimes_\bK V)
\longrightarrow
(P\otimes_B M) \otimes_\bK V
=
F(M)\boxtimes X
$$
is the standard associativity.
A homomorphism of bimodules
$\psi:\,_AP_B\rightarrow \,_AQ_B$
yields a module natural transformation
$$
\psi : P\otimes_B \; \Rightarrow Q \otimes_B \ , \ \ 
\psi_M (p\otimes_B m) =
\psi (p) \otimes_B m
\, .
$$
Any additive functor (for instance, a Morita equivalence)
$F: B\mbox{-}\Mod\rightarrow A\mbox{-}\Mod$
admits a non-canonical strong module functor structure.
One chooses a basis in each vector space $X$ that
gives isomorphisms $M\boxtimes X = M \otimes_\bK X \cong \oplus M$
with a direct sum of copies of the module indexed by the basis.
This leads to a strong module structure on $F$:
$$F_{M,X}: F(M\boxtimes X) 
\xrightarrow{\cong}
F(\oplus M)
\xrightarrow{\cong}
\oplus F(M)
\xrightarrow{\cong}
F(M)\boxtimes X.
$$

We are looking at the class of $\sA^1$-module categories
module-equivalent to $\sA^n$. 
A module-equivalence admits
a quasiinverse equivalence
that is a strong module functor.
This includes 
the zero category 
$\sA^0$ that has only zero objects.

\begin{lemma}
\label{inn_hom}
Let $\sC$ be an $\sA^1$-module category module-equivalent to $\sA^n$.
Let $L_1, \ldots L_n$ be non-isomorphic simple objects in $\sC$.
Then for any object $M$ of $\sC$ there is a functorial
isomorphism
\begin{equation}
\tag{$\clubsuit$}
M \cong 
\oplus_{i=1}^n 
L_i \boxtimes \sC (L_i, M) . 
\end{equation}
\end{lemma}
\begin{proof}
Let $\psi : \sC \rightarrow \sA^n$ be a module equivalence,
$\varphi$ its quasiinverse.
Functorial isomorphism means that the identity functor and the functor
in the right hand side of ($\clubsuit$) are naturally isomorphic. In 
$\sA^n$ we have
$$
\psi (M) \cong 
\oplus_{i=1}^n 
\psi (L_i) \boxtimes \sA^n (\psi(L_i), \psi (M))  
\cong 
\oplus_{i=1}^n 
\psi (L_i) \boxtimes \sC (L_i, M). 
$$
The first isomorphism is given by the evaluation map.
Now we apply $\varphi$, $\varphi_{M,X}$, and the quasiinverse data:
$$
M \cong 
\varphi (\psi (M)) \cong 
\oplus_{i=1}^n 
\varphi (\psi (L_i)) \boxtimes\sC (L_i, M)
\cong 
\oplus_{i=1}^n 
L_i\boxtimes\sC (L_i, M).  
$$
\end{proof}

Lemma~\ref{inn_hom}
replaces Ostrik's trick with internal hom-s \cite{Ost}.
Let us describe
$\tVc$ now. It is  a subcategory of the 2-category of
2-modules $\tMd$.

\begin{mylist}
\item{1.}
The 0-objects $\tMd_0$
are $\sA^1$-module categories module-equivalent to $A\mbox{-}\Mod$ 
for some associative $\bK$-algebra $A$.
\item{2.}
%If $\sC$ and $\sD$ are 0-objects, then 
The objects of the
category $\tMd_1(\sC,\sD)$ are 
$\sA^1$-module
%$\bK$-linear 
functors 
$\sC\rightarrow\sD$.
\item{3.} The 2-morphisms 
in the category  $\tMd_1(\sC,\sD)$
are $\sA^1$-module
natural transformations $F\Rightarrow G$.
\item{4.}
The composition bifunctor
$\tMd_1(\sC,\sD) \times \tMd_1(\sD,\sE) \rightarrow \tMd_1(\sC,\sE)$
is the composition of functors: $F\diamond G := G\circ F =GF$.
The action intertwiners are compositions too:
$$
(F\diamond G)_{M,X}=G_{F(M),X}\circ G(F_{M,X}):
GF(M\boxtimes X) \rightarrow
G(F(M)\boxtimes X) \rightarrow
GF(M) \boxtimes X. 
$$
\item{5.}
The associativity constraint is trivial. For instance, on the level of
functorial morphisms
$$
(FG)H_{M,X}=F_{GH(M),X}\circ F(G_{H(M),X})\circ FG(H_{M,X})
= F(GH)_{M,X}.
$$
\end{mylist}

In particular, $\tMd$ is a 2-category.
The 2-category of 2-vector spaces $\tVc$ is a full 2-subcategory
of $\tMd$ whose 0-objects $\tVc_0$
are those $\sA^1$-module categories module-equivalent 
to $\sA^n$ for some $n\in\bN$. Now we are ready for the main theorem
of this section.

\begin{theorem}
\label{2vec_eq}
The bicategories $\tV$ and $\tVc$ are equivalent.
\end{theorem}
\begin{proof} 
Let us construct 
a 2-functor $F : \tV \rightarrow \tVc$.
We choose non-isomorphic simple modules $L_1^{(n)}, \ldots L_n^{(n)}$
in $\sA^n$ and set 
%Let $\sA^n_0$ be a full subcategory of $\sA^n$
%whose objects are $\oplus_j V_j \boxtimes L_j^{(n)}$
%for some (finite dimensional) vector spaces $V_j$.
%The choice of smaller categories $\sA^n_0$ facilitates
%our construction of  
%a 2-functor $F : \tV \rightarrow \tVc$:
$$
F^0 (n) = \sA^n , \ \ 
F^1_{n,m} ((W_{i,j})_{n\times m}) :
M \xrightarrow{\clubsuit}
\bigoplus_i L_i^{(n)} \boxtimes \sA^n (L_i^{(n)} , M) 
\mapsto 
\bigoplus_{i,j} L_j^{(m)} \boxtimes \sA^n (L_i^{(n)} , M) \otimes W_{i,j}
.
$$
This defines functors $F^1_{n,m}$ on the objects. On the morphisms
it acts on the coefficient vector spaces:
$$
F^1_{n,m} ((\varphi_{i,j}):(V_{i,j})\Rightarrow (W_{i,j})) = 
\bigoplus_{i,j} 
\mbox{Id}_{L_j^{(m)} } \boxtimes \mbox{Id}_{\sA^n (L_i^{(n)}  , M)} \otimes  \varphi_{i,j} .
$$
The 2-isomorphism
$F^2_{n} :\vi_{\sA^n} \Rightarrow F^1_{n,n} (\vi_{n})$
is ($\clubsuit$)
combined with action isomorphisms $V\cong V \boxtimes \bK$:
%defined 
%on each object $\oplus_j V_j \boxtimes L_j^{(n)}$:
$$
F^2_{n} (M) %\bigoplus_j V_j \boxtimes L_j^{(n)})
:
M
\xrightarrow{\clubsuit}
\bigoplus_i  L_i^{(n)} \boxtimes \sA^n (L_i^{(n)} , M) 
\xrightarrow{\cong}
\bigoplus_i  L_i^{(n)} \boxtimes \sA^n (L_i^{(n)} , M) \otimes \bK 
%\oplus_{i,j} W_{i,j} \otimes  \sA^n (L_j^{(n)} , M) \boxtimes L_i^{(m)}
%
%\oplus_j V_j \boxtimes L_j^{(n)}
%\rightarrow 
%\oplus_j \bK \otimes V_j \boxtimes L_j^{(n)}
$$
%via action isomorphisms $V_j\cong \bK \otimes V_j$. 
Finally, the compatibilities
$$
F^2_{n,m,p}  : 
\diamondsuit_{\sA^n,\sA^m,\sA^p} \circ (F^1_{n,m}\times F^1_{m.p} )
\Rightarrow 
F^1_{n,p}\circ \diamondsuit_{n,m,p}
$$
boil down to associativities on the level of coefficient vector
spaces:  
\begin{equation}
\tag{$\spadesuit$}
F^2_{n,m,p}  (U,V)(M) \; : \;  \bigoplus_{i,j,k}
L_k^{(p)} \boxtimes   
( 
\Big( \sA^m (L_j^{(m)} ,  
(L_j^{(m)}\boxtimes   \sA^n (L_i^{(n)} , M) \otimes U_{i,j})) 
\Big) \otimes V_{j,k} )
\xrightarrow{\cong}
\end{equation}
$$
\xrightarrow{\cong}
\bigoplus_{i,j,k} 
L_k^{(p)} \boxtimes   
( 
\Big( \sA^n (L_i^{(n)} , M) \otimes U_{i,j}
\Big) \otimes V_{j,k} )
\xrightarrow{\cong}
\bigoplus_{i,j,k} 
L_k^{(p)} \boxtimes   
(\sA^n (L_i^{(n)} , M) \otimes
(U_{i,j} \otimes V_{j,k})).
$$

Let us now construct a quasiinverse 2-functor 
$G : \tVc \rightarrow \tV$.
In categories $\sC$ and $\sD$, $\sA^1$-module equivalent to $\sA^n$ and $\sA^m$
correspondingly, we choose non-isomorphic
simple
objects $C_1, \ldots C_n$ and $D_1, \ldots D_m$. 
%Any object $M\in \sD$
%admits a functorial isomorphism
%$$
%M\cong \oplus_i \sD(D_i,M)\boxtimes D_i$$
Let
$H:\sC\rightarrow\sD$
be an $\sA^1$-module functor.
We use the functorial isomorphism ($\clubsuit$) to write down the
2-functor $G$:
$$
G^0 (\sC) \coloneqq n, \ \ 
G^0 (\sD) \coloneqq m, \ \ 
G^1_{\sC,\sD} (H) \coloneqq (V_{i,j})
\ 
\mbox{ where }
\
V_{i,j}\coloneqq
\sD(D_j, H(C_i)).
$$
On the 2-morphisms we reduce a natural transformation
$\varphi : H \Rightarrow J$ to linear maps on coefficient spaces:
$$
\bigoplus_j \mbox{Id}_{D_j} \boxtimes \varphi_{i,j}  
:
\bigoplus_j  D_j \boxtimes \sD(D_j, H(C_i)) 
\xrightarrow{\clubsuit}
H(C_i)
\xrightarrow{\varphi (C_i)}
J (C_i)
\xrightarrow{\clubsuit}
\bigoplus_j D_j \boxtimes \sD(D_j, J (C_i))
$$
so that
$$
G^1_{\sC,\sD} (\varphi) \coloneqq  (\varphi_{i,j}) .
$$
The 2-morphism
$G^2_{\sC} :\vi_{n} \rightarrow G^1_{\sC,\sC} (\vi_\sC) 
$
is the obvious one. For $i\neq j$, 
$$
G^2_{\sC,i,i} := \mbox{Id}: \bK \rightarrow \sC (C_i,C_i)\cong \bK,
\ \ \ 
G^2_{\sC,i,j} := \mbox{Id}: 0 \rightarrow \sC (C_i,C_j)\cong 0. 
$$
Finally, let us define the compatibilities
$$
G^2_{\sC,\sD,\sE} : 
\diamondsuit_{n,m,p} \circ (G^1_{\sC,\sD}\times G^1_{\sD,\sE} )
\Rightarrow 
G^1_{\sC,\sE}\circ \diamondsuit_{\sC,\sD,\sE}
$$
for each pair of functors
$H: \sC\rightarrow\sD$,
$J:\sD\rightarrow\sE$. Let $E_1,\ldots E_p$ be simple objects in
$\sE$.
The linear maps
$$
\varphi_{i,k}\coloneqq G^2_{\sC,\sD,\sE} (H,J)_{i,k}:
\bigoplus_j
\sD(D_j, H (C_i))
\otimes
\sE(E_k, J (D_j))
\rightarrow
\sE(E_k, JH (C_i))
$$
come from the computation of compositions
 using the same trick as in ($\spadesuit$):
$$
\bigoplus_k \mbox{Id}_{E_k}\boxtimes \varphi_{i,k} 
:
\bigoplus_{j,k}
E_k
\boxtimes 
\big(
\sE(E_k, J (D_j))
\otimes 
\sD(D_j, H (C_i))
\big)
\xrightarrow{\clubsuit}
JH(C_i)
\xrightarrow{\clubsuit}
\bigoplus_k E_k \boxtimes \sE(E_k, JH (C_i)) .
$$
It remains to write down
the natural 2-isomorphisms 
$FG \Rightarrow \tVc$
and
$GF \Rightarrow \tV$.
These 2-isomorphisms are straightforward, so we leave them as an exercise
for an interested reader.
\end{proof}

We define {\em a 2-representation} of $\sK$ as a 2-functor 
$R: \widetilde{\sK} \rightarrow \tV$
and
{\em a 2-module} for $\sK$ as a 2-functor 
$\sR: \widetilde{\sK} \rightarrow \tMd$.
Since 
$\tV$
and 
$\tVc$ are equivalent, we can think of a 2-representation 
as a ``semisimple'' 2-module 
$\sR: \widetilde{\sK} \rightarrow \tVc$.
The latter approach is convenient for construction of
2-representations.
The former approach is convenient for analysis of 2-representations.
Let us summarise what information a 2-representation (the former
version) contains  \cite[Definition 1]{AO10}.
\begin{mylist}
%\label{list_rep}  
\item{1.}
A number $n=R^0 (\star)$. We call this number 
{\em the  degree} of $R$:
\item{2.}
A 1-morphism 
$R^1 (\vg) = R^1_{\star,\star} (\vg) = (V_{i,j}) : n \rightarrow n$
for every $\vg \in G$.
The dimensions of these vector spaces form a matrix 
$\dim (R^1 (\vg))\in \bN^{n\times n}$.
\item{3.}
2-Isomorphisms 
$R^1(\vg,x) = R^1_{\star,\star} (\vg\stackrel{x}{\Longrightarrow} \partial (x)\vg)= (\varphi_{i,j}) :
R^1(\vg) \Rightarrow R^1(\partial (x)\vg)$
for all $x\in H$, 
$\vg\in G$
that are subject to vertical multiplicativity rule
$R^1(\partial(x)\vg,y)R^1(\vg,x)=R^1(\vg,yx)$.
\item{4.}
A 2-isomorphism $R^2_{\star} = (\theta_{i,j}): \vi_{n} \Rightarrow R(1_G)$.  
\item{5.}
2-Isomorphisms 
$R^2 (\vf,\vg) =
R^2_{\star,\star,\star} (\vf,\vg) = (\psi_{i,j}):
R^1(\vf) \diamond R^1(\vg) \Rightarrow R^1(\vf\vg)$
for every pair $\vf,\vg\in G$ 
such that the pentagon
(or rather degenerated due to strictness of $\widetilde{\sK}$ hexagon)
diagram with two possible 2-morphisms 
$(R^1(\vf) \diamond R^1(\vg))\diamond R^1(\vh) \Rightarrow R^1(\vf\vg\vh)$
and the triangle diagrams with two possible 2-morphisms
$R^1(\vf) \diamond R^1(1_G) \Rightarrow R^1(\vf)$
and 
$R^1(1_G) \diamond R^1(\vg) \Rightarrow R^1(\vg)$
are all commutative. As this is a natural transformation of bifunctors, the naturality condition
$$
R^2 (\partial (x) \vf,\partial (y)\vg) \circ \big( R^1 (\vf,x) \diamond R^1 (\vg,y) \big)
=
R^1 (\vf\vg,x\,^{\vf}y) \circ R^2 (\vf,\vg) \; ,
$$
%equality of linear maps $R
ought to be observed. 
\end{mylist}

We call the 2-representation $R$
{\em unital} if $R^2_{\star}$ is an identity
and {\em strict}
if all $R^2(\vf,\vg)$ and $R^2_{\star}$ are identities. 
%{\em normalised}
%if all $R^2(\vf,1)$ and $R^2(1,\vg)$ 
%are identities, and 
We apply the same adjectives to a 2-module
under the similar conditions. 

Item 5 ensures that
$R^1 (\vg)\diamond R^1 (\vf)$
is 2-isomorphic to $R^{1}(\vg\vf)$ which implies
that
$\dim \circ R^1$ is a group homomorphism from $G$
to $\GL_n (\bN)\cong S_n$.
%
%$\dim (R^1 (\vg^{-1}))$ 
%are inverse $\bN$-valued matrices.
%
%
%Thus, each $\dim R^1(\vg)$ is a permutation matrix and
%a part of the data for $R$ is
%a $G$-action on the set $\{1,2,\ldots n\}$.
%
Item 3 further necessitates
$\dim (R^1 (\vg)) = \dim (R^1 (\partial(x)\vg))$.
%In particular, $\vg$ and $\partial (h)\vg$ give the same permutation of
%$\{1,2,\ldots n\}$.
Hence, hidden inside a 2-representation of $\sK$
we find a permutation action of
the fundamental group $\pi_1 (\sK)=G/\partial(H)$
on the finite set $\{1,2,\ldots n\}$.
It is fruitful to think of  a 2-representation of $\sK$
as a permutation action of $\pi_1 (\sK)$ together with some additional
data. The precise nature of this data will be uncovered later
(cf. Section~\ref{s2a}).

A homomorphism of 2-representations $\psi:R \rightarrow R^\prime$ 
or 2-modules $\psi:\sR \rightarrow \sR^\prime$ 
is a natural 2-transformation of 2-functors. 
%is a 2-functor $F:\tV \rightarrow \tV$ such that 
%$R^\prime \cong FR$, i.e. $R^\prime$ and  $FR$ are naturally isomorphic.
An equivalence of 2-representations 
is a natural 2-isomorphism of 2-functors. 
%$F:R \rightarrow R^\prime$ 
%is a 2-equivalence $F:\tV \rightarrow \tV$ such that $R^\prime \cong FR$.
Let $\tRep^n (\sK)$ be the class of
equivalence classes of 
2-representations of
$\widetilde{\sK}$
of degree $n$.
It is clear from the description above that 
$\tRep^n (\sK)$
is actually a set.
%Let $\tSep^n (\sK)$ be the subset of those classes
%containing a strict 2-representation.
%Let 
%$\tRep(\sK)\coloneqq\cup_n \tRep^n (\sK)$
%and
%$\tSep(\sK)\coloneqq\cup_n \tSep^n (\sK)$.

Before we proceed with our studies, we remark that the
2-representations of
$\sK$ can be
considered over any semisimple rigid monoidal category $\sC$.
Instead of $\tV$ one considers square matrices of objects of $\sC$.
Instead of $\tVc$ one considers semisimple $\sC$-module categories
with finitely many simple objects under their Ostrik's internal hom-s 
\cite{Ost}.
The results of this section easily extend to this greater generality.

\begin{center}
\section{Morita theory for a 2-group}
\label{s2}
\end{center}
An associative algebra $A$ admits a crossed module of 2-automorphisms
$\tAut (A) \coloneqq(A^\times \xrightarrow{\partial} \Aut (A))$
where $A^\times$ is the group of units of $A$ and
$\partial (x)$ is the inner automorphism
$y \mapsto xyx^{-1}$.
Let $\sK=(H\xrightarrow{\partial}G)$ be a crossed module. 
By {\em a $\sK$-algebra} we understand 
an associative algebra $A$ with a (left) action of $\sK$,
i.e., a crossed module homomorphism
$\omega_A: \sK\rightarrow \tAut (A)$. 

It is useful to introduce a weak version of this concept:
{\em a weak $\sK$-algebra} is 
an associative algebra $A$ with 
a weak crossed module homomorphism
$\omega_A: \sK\rightarrow \tAut (A)$.
We define {\em a weak crossed module homomorphism} 
$\omega_A: \sK\rightarrow \tAut (A)$
as a triple $(\omega_1,\omega_2,\omega_3)$
where
$\omega_1: G \rightarrow \Aut (A)$
is a group homomorphism,
$\omega_3: G\times G \rightarrow Z(A)^\times$
is a normalised cocycle, i.e., 
$$
\omega_3 (\vg, 1) = \omega_3 (1,\vg) = 1, \
\,^\vf\omega_3 (\vg, \vh)\omega_3 (\vf,\vg\vh) =
\omega_3 (\vf\vg, \vh)\omega_3 (\vf,\vg)
\mbox{ for all } \vf,\vg,\vh\in G ,$$
and 
$\omega_2: H \rightarrow A^\times$
is a unital projective group homomorphism with
the cocycle $\omega_3\circ(\partial\times\partial)$, i.e., 
$$
%\omega_2 (x) = 1, \
\omega_2 (xy) = 
\omega_3 (\partial x, \partial y)\omega_2 (x) \omega_2 (y)
\mbox{ for all } x,y\in H$$
such that they respect the crossed module structure maps:
$$
\omega_1 (\partial x) = \partial (\omega_2 (x)), \ \
\omega_2 (\,^\vf x)=
  \,^{\omega_1(\vf)} \omega_3 (\partial x, \vf^{-1})
  \omega_3(\vf, \partial x\, \vf^{-1})
  \omega_3(\vf, \vf^{-1})^{-1} 
\;^{\omega_1(\vf)} \omega_2 (x) \; .
%\omega_2 (\,^\vf x)= \,^{\omega_1(\vf)} \omega_2 (x)  
$$
%and an additional {\em naturality condition} holds for the cocycle:
%$$
%\omega_3(\vf,\vg) = \,^\vf\omega_3(\vg,\vf)
%\ \mbox{ whenever } \ \vf\vg\in \partial (H).
%$$

%The origin of this naturality condition is mysterious at the moment
%but will become clear later in the paper (cf. Proposition~\ref{ext_2} and Lemma~\ref{Extension}).
Note that
the normalised cocycle condition implies $\omega_2 (1) = 1$. 
Also observe that if $H$ is trivial (so that $\widetilde{\sK}$ is just the group $G$),
a weak homomorphism is just a homomorphism of $G$ with
a normalised cocycle.
%In particular, the naturality condition immediately follows from the normalised cocycle condition:
%$$
%\,^\vf\omega_3(\vf^{-1},\vf) =
%\,^\vf\omega_3(\vf^{-1},\vf)\omega_3 (\vf,\vf^{-1}\vf)=
%\omega_3(\vf\, \vf^{-1},\vf)\omega_3 (\vf,\vf^{-1})=
%\omega_3(\vf,\vf^{-1}).
%$$

Weak crossed module homomorphisms
have been studied by Noohi \cite[Def. 8.4]{Noo}
but our concept is different.
Noohi's weak homomorphism 
is also a triple $(\omega_1,\omega_2,\omega_3)$
of the maps with the same domains and ranges.
The difference is that $\omega_2$ is now required
to be a homomorphism of groups, while 
$\omega_1$
is a projective homomorphism with
the cocycle $\partial \circ \omega_3$.
There is a different respect condition. 
Since both notions yield homomorphisms of 2-groups, 
it is plausible that there may be a common generalisation.
A promising concept is a quadruple 
$(\omega_1,\omega_2,\omega_3,\omega^\prime_3)$
where
$\omega_2$
is a projective homomorphism with
the cocycle $\omega_3\circ(\partial\times\partial)$ and 
$\omega_1$
is a projective homomorphism with
the cocycle $\partial \circ \omega^\prime_3$.
These quadruples need to define homomorphisms of the corresponding
2-groups. We leave it at that hoping that future research
will clarify a connection between these two notions of
weak homomorphism.

Let $A$ be a weak $\sK$-algebra. The group $G$
acts on the category of left $A$-modules 
$A\mbox{-}\Mod$
on the right~\cite{GRR11}.
%We use the notation and
%conventions of Gunnels-Rose-Rumynin
%\cite{GRR11}:
Each automorphism $\vg\in \Aut (A)$ gives an endofunctor 
$[\vg]:A\mbox{-}\Mod\rightarrow A\mbox{-}\Mod$.
For an $A$-module $M$,
the twisted $A$-module $M^{[\vg]}$ is equal to $M$
as a vector space and the new action:
$a \cdot^{[\vg]} m = \vg(a)m$.
On morphisms $\varphi^{[\vg]}=\varphi$.
Notice that $[\vf]\circ[\vg]=[\vg\vf]$,
so that this is a right action.

We would like to extend this 
to an action 
%to the left action
%and to extend it 
of the 2-group $\widetilde{\sK}$. 
With this in mind, we consider the 2-group 
%$\sAut (A\mbox{-}\Mod)$
%of the automorphisms of the category of left $A$-modules.
%One can think of this as a weak homomorphism 
%$\Theta_A :G\rightarrow \sAut
%(A\mbox{-}\Mod)$.
%Strictly speaking,
%the automorphism ``group'' 
%$\sAut (A\mbox{-}\Mod)$
%is not a group for two reasons.
%First, it is a class rather than a set.
%Second, the inverses do not exist, only quasiinverses.
%Besides $\Theta_A$
%preserves the multiplication only up to natural transformations
%$\gamma_{g,h} : [hg]\rightarrow [g]\circ [h]$.
%One way to fix it is to consider autoequivalences only up to natural
%transformations.
%Let
$\tsAut(A\mbox{-}\Mod)$ 
of the automorphisms of the category of left $A$-modules:
\begin{mylist}
\item{1.}
The 0-objects of $\tsAut(A\mbox{-}\Mod)$ is a one point set
$\{\star\}$.
\item{2.}
The 1-objects $\tsAut(A\mbox{-}\Mod)_1(\star, \star)$ 
are strong $\sA_1$-module autoequivalences of $A\mbox{-}\Mod$.
\item{3.}
For a pair of autoequivalences $F_1,F_2$ the 2-morphisms 
$\tsAut(A\mbox{-}\Mod)_2(F_1, F_2)$ 
are $\sA_1$-module natural isomorphisms $F_1 \Rightarrow F_2$.
\item{4.}
The composition bifunctor is the composition of functors 
$F_1\diamond F_2\coloneqq F_2\circ F_1= F_2F_1$.
\end{mylist}
This 2-group is not strict because
functors have quasi-inverses, not inverses, in general.
Another way to think of $\tsAut(A\mbox{-}\Mod)$
is a 2-subcategory of $\tMd$ with one 0-object.
%%%It agrees with the previous section.
%%%Alternatively we can just apply functors on the right.

Before we proceed, let us describe 
the natural transformations between the group twists.
%endofunctors of $A\mbox{-}\Mod$.
\begin{lemma}
  \label{nat_trans}
  Suppose $\vf$ and $\vg$ are automorphisms of a ring $A$.
  Then the map
  $$
  \Upsilon:
  \{x\in A \,\mid\, \forall a\in A \ \  x a=\vg(\vf^{-1}(a))x\}
  \rightarrow
  \mbox{\rm Nat.Trans} ([\vf],[\vg]), \ \
  \Upsilon (x)_M : m \mapsto x \cdot m
  $$
  is a bijection.
\end{lemma}
\begin{proof}
  Observe that $\Upsilon (x)_M \in \hom (M^{[\vf]},M^{[\vg]})$
  for any $A$-module $M$:
  $$
\Upsilon(x)_M ( a \cdot^{[\vf]} m) 
=
x\vf(a)m 
=
\vg(\vf^{-1} (\vf (a))) xm 
=
\vg (a) xm
=
a \cdot^{[\vg]} 
\Upsilon(x)_M (m)
$$
for all $m\in M$, $a\in A$.
Since $F(xm)=xF(m)$ for any homomorphism
  $F$, $\Upsilon (x)$ is a natural transformation.
  Hence, $\Upsilon$
  is a well-defined function.
  % between the aforementioned sets.

To show that this is a bijection we construct the inverse function:
  $$
\Xi:
 \mbox{\rm Nat.Trans} ([\vf],[\vg])
  \rightarrow
   \{x\in A \,\mid\, \forall a\in A \ \  x a=\vg(\vf^{-1}(a))x\}, \ \ 
  \Xi (\varphi) = \varphi_A (1).
  $$  
  The equality $\Xi \Upsilon = \mbox{Id}$ is obvious:
  $\Xi (\Upsilon (x))= \Upsilon(x)_A (1) = x$.
  The opposite equality $\Upsilon \Xi= \mbox{Id}$ follows
  from the fact that $\,_AA$ is a generator of $A\mbox{-}\Mod$
  so that a natural transformation $\varphi$ is uniquely determined
  by $\varphi_A$: using the $A$-module homomorphism
  $F: A\rightarrow M$, $F(a)=am$, we conclude that
  $\varphi_M (m) = \varphi_M (F(1)) = F (\varphi_A (1))
  = \varphi_A (1)m$.
\end{proof}

%The group of automorphisms (up to equivalence) of the category
%$A\mbox{-}\Mod$ 
%is the homotopy group
%$\pi_1(\tsAut(A\mbox{-}\Mod))$
%and we have a group homomorphism
%$G\rightarrow \pi_1 (\tsAut (A\mbox{-}\Mod))$.
%Furthermore, this homomorphism extends to the whole $\widetilde{\sK}$.

Armed with an understanding of natural transformations between the twists
by automorphisms, we can mould $A\mbox{-}\Mod$ into a  2-module
for $\wAut (A)$:
\begin{prop}
\label{ext_1}  
The assignment $\Theta (\vg) = [\vg]$ extends
to a strict 2-module %homomorphism of 2-groups \newline
$\Theta : \wAut (A) \rightarrow \tsAut (A\mbox{-}\Mod)$.
\end{prop}
\begin{proof}
Inevitably, 
$\Theta^0 (\star) = \star$.
We define the functor $\Theta^1_{\star,\star}$ using the map from Lemma~\ref{nat_trans}:
$$
\Theta^1_{\star,\star} (\vg)
\coloneqq [\vg], \ \ 
%We use the natural isomorphisms $\gamma_{g,h} : [hg]\rightarrow [g]\circ [h]$.
\Theta^1_{\star,\star} (\vg_1\stackrel{x}{\Longrightarrow}\vg_2) 
=
\Theta^1_{\star,\star} (\xymatrix{
\star
   \ar@/^2pc/[rr]_{\quad}^{\vg_2}="1"
   \ar@/_2pc/[rr]_{\vg_1}="2"
&& \star
%   \ar@/^2pc/[ll]_{\quad}^{\vg_1}="1"
%   \ar@/_2pc/[ll]_{\vg_2}="2"
\ar@{}"2";"1"|(.2){\,}="7"
\ar@{}"2";"1"|(.8){\,}="8"
\ar@{=>}"7" ;"8"^{x}
} )
\coloneqq \Upsilon (x).
$$
Since each $[\vg]$ is an $\sA^1$-module functor,
$\Theta^1_{\star,\star}$ is well-defined. 
%where $\theta : \sK \rightarrow\tAut (A)$
%is the $\sK$-algebra structure.
%This means that 
%$\Theta^1_{\star,\star} (\vg_1\stackrel{x}{\Longrightarrow}\vg_2)_{M} (m )
%= x\cdot m
%$
%for an $A$-module $M$ and $m\in M$. 
Observe that $\vg_2=\partial (x) \vg_1$
so that
$
xa = xax^{-1}x
=
\vg_2(\vg_1^{-1}(a))x
$
for all $a\in A$
and $\Upsilon (x)$ is
a well-defined  $\sA^1$-module natural transformation.

%and $\vg_2^{-1} = \vg_1^{-1} \partial(x^{-1})$.
Let us verify 
%\newline 
%$\Theta^1_{\star,\star}(\vg_1\stackrel{x}{\Longrightarrow}\vg_2)_{M}$
%is an $A$-module isomorphism
%$M^{[\vg_1]}\rightarrow M^{[\vg_2]}$:
%$$
%x\cdot ( a \cdot^{[\vg_1]} m) 
%=
%x\cdot ( \vg_1(a) \cdot m) 
%=
%(x\vg_1(a)x^{-1}) x \cdot m 
%=
%(\partial (x) \vg_1) (a) \cdot
%x\cdot m
%=
%a \cdot^{[\vg_2]} (x\cdot m) 
%$$
%for all $m\in M$, $a\in A$, $x\in A^\times$.
%The verification
that  $\Theta^1_{\star,\star}$ is functor:
$$
\Theta^1_{\star,\star} (
[\vg_2\stackrel{y}{\Longrightarrow}\vg_3]
\circ
[\vg_1\stackrel{x}{\Longrightarrow}\vg_2]
) =
\Theta^1_{\star,\star} (\vg_1\stackrel{yx}{\Longrightarrow}\vg_3)
=
\Upsilon(yx)=\Upsilon(y)\Upsilon(x) =
\Theta^1_{\star,\star} (\vg_2\stackrel{y}{\Longrightarrow}\vg_3)
\Theta^1_{\star,\star} (\vg_1\stackrel{x}{\Longrightarrow}\vg_2) \; .
$$
The 2-isomorphism $\Theta^2_{\star}$
and the compatibilities 
$\Theta^2_{\star,\star,\star}$ can be chosen to be identities
making $\Theta$ strict.
%%%and unital.
%The square and the {\bf check the square}
%hexagon\footnote{hexagon appears in braiding. Why do I use it here??}
%diagrams are obviously commutative.
\end{proof}

Now we are ready ``to compose'' the canonical homomorphism
of 2-groups 
$\Theta : \wAut (A) \rightarrow \tsAut (A\mbox{-}\Mod)$
with a weak crossed module homomorphism
$\omega_A: \sK\rightarrow \tAut (A)$.
We have to deal repeatedly with expressions of
the form $[\vf]\diamond \Upsilon (x)$ that we transform to
$\Upsilon(\,^\vf x)$. It is instructive
to apply both expression to an element $m\in M$.
The former gives $x\cdot^{[\vf]}m$ while the latter gives
$(\,^\vf x)\cdot m$. Hence, they are equal.

\begin{prop}
\label{ext_2}
The structure of a weak $\sK$-algebra on $A$ gives rise to
a unital 2-module 
%a weak unital normalised homomorphism of 2-groups 
$\Theta_A : \widetilde{\sK}\rightarrow \tsAut (A\mbox{-}\Mod)$
defined as follows:
%``The composition'' is a weak homomorphism of 2-groups
%$\Theta_A : \widetilde{\sK} \rightarrow \tsAut (A\mbox{-}\Mod)$
%defined as follows:
$$
\Theta_A^0(\star) = A\mbox{-}\Mod, \ \ 
\Theta_{A\; \star,\star} ^1(\vg) = [\omega_1(\vg)], \ \
\Theta^1_{A}(\vg, x) =
\Theta^1_{A \; \star,\star} (\vg\stackrel{x}{\Longrightarrow} \partial (x)\vg)= 
\Upsilon (\omega_3(\partial x, \vg) \omega_2 (x))
%:[\omega_1(\vg)] \Rightarrow [\omega_1( (\partial x) \vg)]
,
$$
$$
\Theta^2_{A \; \star} = \mbox{Id}, \ \ \ 
  \Theta^2_{A} (\vf, \vg) =
  \Theta^2_{A \; \star,\star,\star} (\vf, \vg) =
  \Upsilon (\omega_3 (\vf,\vg)).
%  :      [\omega_1(\vf)]\diamond[\omega_1(\vg)]=[\omega_1(\vf\vg)]
%      \Rightarrow [\omega_1(\vf\vg)].
$$
%such that $\Theta^1_{A \; \star,\star} (\vg) = [\vg]$ for all $\vg\in G$.
If $A$ is a $\sK$-algebra, $\Theta_A$ is
a strict  2-module. %homomorphism. 
\end{prop}
\begin{proof}
  Lemma~\ref{nat_trans} guarantees that
  the 2-isomorphisms are well-defined:  
  $\Theta^1_A(\vg,x)$
  acts $\Theta^1_A(\vg) \Rightarrow \Theta_A^1(\partial (x)\vg)$,
  while $\Theta^2_A (\vf,\vg)$ acts
  $\Theta^1_A(\vf) \diamond \Theta^1_A(\vg) \Rightarrow \Theta^1_A(\vf\vg)$.
  The vertical multiplicativity follows from %Proposition~\ref{ext_2} and
  the definition
  of a weak homomorphism of crossed modules:
  $$
  \Theta^1_A(\partial x\, \vg,y)\Theta^1_A(\vg,x)=
  \Upsilon (\omega_3(\partial y, \partial x \, \vg) \omega_2 (y))
  \Upsilon (\omega_3(\partial x,\vg) \omega_2 (x))
  = \Upsilon (\omega_3(\partial y, \partial x \,\vg)
  \,^{\partial y} \omega_3(\partial x, \vg) \omega_2 (y)\omega_2 (x))
  $$
  $$
    = \Upsilon (\omega_3(\partial y\, \partial x,\vg) \omega_3(\partial y, \partial x) \omega_2 (y)\omega_2 (x))
    = \Upsilon (\omega_3(\partial (yx),\vg) \omega_2 (yx))
    =
  \Theta^1_A(\vg,yx) .
  $$
  The equality between the lines is the cocycle condition.
  %The last equality in the first line
  %holds because $\omega_1 (\partial x, \vg)$ is a central elements so it commutes
  %with   
  A similar reasoning proves the naturality condition for $\Theta^2_{A\; \star,\star,\star}$.
  Checking it,  we 
   use the centrality property $\,^{\partial x}\omega_3(\vf,\vg)=\omega_3(\vf,\vg)$.
   To help the reader we underline terms undergoing transformations at the next equality.
\begin{align*}
\Theta^2_A (\partial (x) \vf,\partial (y)\vg)
\circ \big( \Theta^1_A (\vf,x) \diamond \Theta^1_A (\vg,y) \big)
=
&
\Upsilon (\omega_3(\partial x \, \vf,  \partial y \, \vg ))
\Upsilon ( \underline{\,^{\partial x \, \vf}[\omega_3(\partial y, \vg) \omega_2 (y)]})
\cdot
\\
\cdot
\Upsilon (\omega_3(\partial x, \vf) \omega_2 (x))
=
\Upsilon \bigg(
&
\underline{
  \omega_3(\partial x \, \vf,  \partial y \, \vg )
\omega_3(\partial x, \vf) 
}
\,^\vf\omega_3(\partial y, \vg)
\underline{
\,^{\partial x\, \vf}\omega_2 (y)
\omega_2 (x)}
\\
%\Upsilon (
=
\omega_3(\partial x ,  \vf  \partial y \, \vg )
\underline{
  \omega_3(\vf, \partial y\, \vg) 
\,^\vf\omega_3(\partial y, \vg)
}
\omega_2 (x)
\underline{\,^\vf\omega_2 (y)}
%\Upsilon (
=
&
\omega_3(\partial x ,  \vf  \partial y \, \vg )
\underline{\omega_3(\vf, \partial y)} 
\omega_3(\vf\partial y, \vg)
\omega_2 (x) \cdot
\\
\cdot
\underline{
\,^\vf\omega_3(\partial y, \vf^{-1})^{-1}
\omega_3(\vf, \partial y\, \vf^{-1})^{-1}
}
\ 
\underline{\omega_3(\vf, \vf^{-1})}
\omega_2 (\,^\vf y)
%)
&
=
\omega_3(\partial x ,  \vf  \partial y \, \vg )
\underline{
  \omega_3(\vf\partial y, \vf^{-1})^{-1}
\,^\vf\omega_3(\vf^{-1}, \vf)
}
\\
\omega_3(\vf\partial y, \vg)
\omega_2 (x)
\omega_2 (\,^\vf y)
=
%\Upsilon (
\omega_3(\partial x, \vf \partial y\, \vg)
&
\underline{
  \omega_3(\vf \partial y\, \vf^{-1},\vf)
  \omega_3(\vf \partial y\,\vg)
}
\omega_2 (x) \omega_2(\,^{\vf}y)
%)
=
\\
%\Upsilon (
\underline{
\omega_3(\partial x, \vf \partial y\, \vg)
\omega_3(\vf \partial y\, \vf^{-1},\vf\vg)
}
\omega_3(\vf,\vg )
\omega_2 (x) \omega_2(\,^{\vf}y)
%)
=
&
%\Upsilon (
\underline{
\omega_3(\partial x\, \vf \partial y\, \vf^{-1}, \vf\vg)
}
\omega_3(\vf,\vg )
\underline{
\omega_3(\partial x, \vf \partial y\, \vf^{-1})
  \omega_2 (x) \omega_2(\,^{\vf}y)}
\bigg)
\\
=
\Upsilon (\omega_3(\partial (x\,^{\vf}y), \vf\vg) \omega_2 (x\,^{\vf}y))
\Upsilon (\omega_3(\vf,\vg ))
&
=
\Theta^1_A (\vf\vg,x\,^{\vf}y) \circ \Theta^2_A (\vf,\vg) \; .
\end{align*}
Let us now verify commutativity
of the pentagon
diagram. 
One of the key
2-morphisms
\linebreak
$(\Theta^1_A(\vf) \diamond \Theta^1_A(\vg))\diamond \Theta^1_A(\vh)
\Rightarrow \Theta^1_A(\vf\vg\vh)$
is
$$
\Upsilon (\omega_3 (\vf\vg,\vh))\Upsilon (\omega_3 (\vf,\vg)) =
\Upsilon (\omega_3 (\vf\vg,\vh)\omega_3 (\vf,\vg)).
$$
The second key 2-morphism is
$$
\Upsilon (\omega_3 (\vf,\vg\vh))\circ \Big([\vf] \diamond \Upsilon (\omega_3 (\vg,\vh)) \Big)=
\Upsilon (\omega_3 (\vf,\vg\vh)\,^\vf\omega_3 (\vg,\vh)).
$$
They are equal by the cocycle property of $\omega_3$.
Commutativity of the triangle diagrams is obvious because 
$\Theta^2_{A \; \star} = \mbox{Id}$.
%Normalisation means that $\Theta_A(1,\vg)=\mbox{Id}=\Theta_A(\vf,1)$
%and follows the cocycle $\omega_3$ being normalised.
This finishes the proof of the first statement.

A $\sK$-algebra is a weak $\sK$-algebra with trivial $\omega_3$.
Thus, all $\Theta^2_{A} (\vf, \vg)$ are trivial and the homomorphism
$\Theta_A$ is strict. This proves the second statement.
\end{proof}

The following {\em Extension Lemma} is a partially converse statement to
Proposition~\ref{ext_2}. 
%We restrict to unital 2-representations for clarity.
%This restriction can be removed, for instance, by noticing that
%a non-unital 2-representation is 2-isomorphic to a unital 2-representation. 
%It is crucial for later studies.

\begin{lemma} 
\label{Extension}
Let $\sK=(H\rightarrow G)$ be a crossed module,
$A$ an associative $G$-algebra.
Suppose we have a unital 2-module
$\Omega : \widetilde{\sK}\rightarrow \tsAut (A\mbox{-}\Mod)$
whose restriction to $G$ comes from an action
$\varpi : G\rightarrow \Aut (A)$:
$$
\Omega_{\star,\star}^1 (\vg) = [\varpi (\vg)]
\ \
\mbox{ for all }
\vg \in G.
$$
Then there exists
a weak $\sK$-algebra structure on $A$
%$\omega_A : \widetilde{\sK}\rightarrow \tAut (A)$
such that $\Omega = \Theta_A$ (the latter is defined in Proposition~\ref{ext_2}).
Furthermore, if $\Omega$ is strict, then $A$ is a $\sK$-algebra.
\end{lemma}
\begin{proof}
  The first component of a weak homomorphism
$(\omega_1,\omega_2,\omega_3) : \widetilde{\sK}\rightarrow \tAut (A)$
is already defined: $\omega_1 (\vg) \coloneqq \varpi (\vg)$.
Let us define the other two components using Lemma~\ref{nat_trans}: 
$$
\omega_2 (x) \coloneqq
\Upsilon^{-1} \Big( \Omega^1_{\star,\star} (1_G\stackrel{x}{\Longrightarrow}\partial(x)) \Big)
, \ \ \
\omega_3 (\vf,\vg) \coloneqq
\Upsilon^{-1} \Big( \Omega^2_{\star,\star,\star} (\vf,\vg) \Big) 
%=
%\Omega^1_{\star,\star} (
%\xymatrix{
%\star
%   \ar@/^2pc/[rr]_{\quad}^{\partial(h)}="2"
%   \ar@/_2pc/[rr]_{1_G}="1"
%&& \star
%%   \ar@/^2pc/[ll]_{\quad}^{\vg_1}="1"
%%  \ar@/_2pc/[ll]_{\vg_2}="2"
%\ar@{}"1";"2"|(.2){\,}="7"
%\ar@{}"1";"2"|(.8){\,}="8"
%\ar@{=>}"7" ;"8"^{h}
%} )_A (1_A)
$$
for all $\vf, \vg\in G$, $x\in H$. Let us observe that
$\Omega$ can be computed from the triple $(\omega_1,\omega_2,\omega_3)$
by formulas for $\Theta_A$ in Proposition~\ref{ext_2}.
The only formula that requires checking is the following:
$$
\Omega^1_{\star,\star} (\vg, x) = 
%\Omega^1_{\star,\star} (\vg\stackrel{x}{\Longrightarrow} \partial (x)\vg)=
\Omega^1_{\star,\star} (
      [1\stackrel{x}{\Longrightarrow} \partial (x)]
      \diamond
[\vg\stackrel{1}{\Longrightarrow} \vg]
)=
\Omega^2_{\star,\star,\star} (\partial x, \vg) \circ (
\Omega^1_{\star,\star} (1,x)
\diamond
\Omega^1_{\star,\star} (\vg,1) )
=
\Upsilon (\omega_3(\partial x, \vg) \omega_2 (x)).
$$
We have used the fact that
$\Omega^1_{\star,\star} (\vg,1)$
is identity, that follows from vertical multiplicativity.

It remains to verify that $(\omega_1,\omega_2,\omega_3)$ is a weak
homomorphism of crossed modules.
This boils down to checking the following
five identities:
%six identities:
%for all $x,y\in H$, $\vg\in G$.
\begin{enumerate}
\item{$\omega_1 (\partial x) = \partial (\omega_2 (x))$:}
\  Let $a=\omega_2 (x)$. Pick $b\in A$. Observe that
%  \begin{equation}
%    \label{ab}
$    \Omega^1_{\star,\star} (1_G\stackrel{x}{\Longrightarrow}\partial(x))_A (b)
= ab.
$
%\Omega^1_{\star,\star} (1_G\stackrel{x}{\Longrightarrow}\partial(x))_A (b \cdot 1)
%=
%b \cdot^{[\partial x]}
%\Omega^1_{\star,\star} (1_G\stackrel{x}{\Longrightarrow}\partial(x))_A (1)
%= aba^{-1} \cdot a = ab
%.
%\end{equation}
Hence, 
$$
ab = \Omega^1_{\star,\star} (1_G\stackrel{x}{\Longrightarrow}\partial(x))_A (b\cdot 1) =
b\cdot^{[\partial x]}\Omega^1_{\star,\star} (1_G\stackrel{x}{\Longrightarrow}\partial(x))_A (1) =
\,^{\partial x}b \cdot a \; .
$$
We are done since $a$ is inevitably invertible.
\item{$\omega_2 (\,^\vf x)=
%  \omega_3(\vf, \partial x\, \vf^{-1})
%  \,^{\omega_1(\vf)} \omega_3 (\partial x, \vf^{-1})
\ \ldots \ 
\,^{\omega_1(\vf)} \omega_2 (x)$:}
  This follows from the definition of a 2-functor:
\begin{align*}
\omega_2 (\,^\vf x)=
\Omega^1_{\star,\star} (1 \stackrel{\,^\vf x\;}{\Longrightarrow} \vf \partial x \, \vf^{-1} )
=
&
\; 
\Omega^1_{\star,\star} \Big(
\xymatrix{
\star
   \ar@/^2pc/[rr]_{\quad}^{\vf}="1"
   \ar@/_2pc/[rr]_{\vf}="2"
&& \star
   \ar@/^2pc/[rr]_{\quad}^{\partial x\, \vf^{-1}}="3"
   \ar@/_2pc/[rr]_{\vf^{-1}}="4"
\ar@{}"1";"2"|(.2){\,}="7"
\ar@{}"1";"2"|(.8){\,}="8"
\ar@{=>}"8" ;"7"^{1}
&& \star
\ar@{}"4";"3"|(.2){\,}="7"
\ar@{}"4";"3"|(.8){\,}="6"
\ar@{=>}"7" ;"6"^{x}
%\ar@/_2pc/[rr]_{\quad}_{\vf^{-1}}="10"
%\ar@/^2pc/[rr]^{\vf^{-1}}="11"
%&&
%\star
%\ar@{}"11";"10"|(.2){\,}="17"
%\ar@{}"11";"10"|(.8){\,}="16"
%\ar@{=>}"16" ;"17"^{1}
}
\Big)
=
\\
=
%\Omega^1 \Big(
%\xymatrix{
%\star
%   \ar@/^2pc/[rr]_{\quad}^{\vf}="1"
%   \ar@/_2pc/[rr]_{\vf}="2"
%&& \star
%   \ar@/^2pc/[rr]_{\quad}^{\partial x}="3"
%   \ar@/_2pc/[rr]_{1}="4"
%\ar@{}"1";"2"|(.2){\,}="7"
%\ar@{}"1";"2"|(.8){\,}="8"
%\ar@{=>}"8" ;"7"^{1}
%&& \star
%\ar@{}"4";"3"|(.2){\,}="7"
%\ar@{}"4";"3"|(.8){\,}="6"
%\ar@{=>}"7" ;"6"^{x}
%\ar@/_2pc/[rr]_{\quad}_{\vf^{-1}}="10"
%\ar@/^2pc/[rr]^{\vf^{-1}}="11"
%&&
%\star
%\ar@{}"11";"10"|(.2){\,}="17"
%\ar@{}"11";"10"|(.8){\,}="16"
%\ar@{=>}"16" ;"17"^{1}
%}
%\Big)
%=
%\Omega^1
%\Big(
%\xymatrix{
%\star
%   \ar@/^2pc/[rr]_{\quad}^{\vg\,\vg_1}="1"
%   \ar@/_2pc/[rr]_{\vf\,\vf_1}="2"
%&& \star
%%   \ar@/^2pc/[ll]_{\quad}^{\vf\vf_1}="1"
%%   \ar@/_2pc/[ll]_{\vg\vg_1}="2"
%\ar@{}"1";"2"|(.2){\,}="7"
%\ar@{}"1";"2"|(.8){\,}="8"
%\ar@{=>}"8" ;"7"^{h\,^\vf h_1}
%} .
\Omega^2_{\star,\star,\star}(\vf, \partial x\, \vf^{-1})
\circ
(
\Omega^1_{\star,\star}(\vf, 1)
\diamond
\Omega^1_{\star,\star}(\vf^{-1}, x)
)
\circ
&
\Omega^2_{\star,\star,\star}(\vf, \vf^{-1})^{-1}
=
\omega_3(\vf, \partial x\, \vf^{-1})
\cdot
\\
\cdot
\,^\vf (\omega_3 (\partial x, \vf^{-1})\omega_2 (x))
\omega_3(\vf, \vf^{-1})^{-1}
=
&
\omega_3(\vf, \partial x\, \vf^{-1})
\,^\vf \omega_3 (\partial x, \vf^{-1})
\omega_3(\vf, \vf^{-1})^{-1}
\,^\vf\omega_2 (x).
\end{align*}
\item{$\omega_2 (xy) = 
\omega_3 (\partial x, \partial y)\omega_2 (x) \omega_2 (y)$:}
Again, this follows from the definition of a 2-functor:
$$
\omega_2 (xy)
=
\Omega^1
\Big(
\xymatrix{
\star
   \ar@/^2pc/[rr]_{\quad}^{\partial x}="1"
   \ar@/_2pc/[rr]_{1}="2"
&& \star
%   \ar@/^2pc/[ll]_{\quad}^{\vf}="1"
%   \ar@/_2pc/[ll]_{\vg}="2"
   \ar@/^2pc/[rr]_{\quad}^{\partial y}="3"
   \ar@/_2pc/[rr]_{1}="4"
\ar@{}"1";"2"|(.2){\,}="7"
\ar@{}"1";"2"|(.8){\,}="8"
\ar@{=>}"8" ;"7"^{x}
&& \star
%   \ar@/^2pc/[ll]_{\quad}^{\vf_1}="1"
%   \ar@/_2pc/[ll]_{\vg_1}="2"
\ar@{}"4";"3"|(.2){\,}="7"
\ar@{}"4";"3"|(.8){\,}="6"
\ar@{=>}"7" ;"6"^{y}
}
\Big)
=
\Omega^2_{\star,\star,\star} (\partial x, \partial y) \circ
( \Omega^1_{\star,\star} (1,x) \diamond \Omega^1_{\star,\star} (1,y))
$$
$$
=
\omega_3 (\partial x, \partial y) \,^{\partial x}\omega_2 (y) \omega_2 (x).
=
\omega_3 (\partial x, \partial y) \omega_2 (x) \omega_2 (y) \omega_2(x)^{-1} \omega_2 (y)
=
\omega_3 (\partial x, \partial y) \omega_2 (x) \omega_2 (y). 
$$
\item{$\,^\vf\omega_3 (\vg, \vh)\omega_3 (\vf,\vg\vh) =
  \omega_3 (\vf\vg, \vh)\omega_3 (\vf,\vg)$ :}
  This follows from commutativity of the pentagon (collapsed hexagon) diagram.
  The argument is the converse of the argument given in the proof of Proposition~\ref{ext_2}.
\item{$\omega_3 (\vg, 1) = 1 =\omega_3 (1,\vg)$: }
  \ Both identities
  follow from the 2-module being unital.
%  \item{$\omega_3(\vf,\vg) = \,^\vf\omega_3(\vg,\vf)$ :}
\end{enumerate}

The second statement is immediate. If $\Omega^2_{\star,\star,\star}$ are all trivial,
then $\omega_3 \equiv 1$.  Thus, $(\omega_1,\omega_2,1)$ is a homomorphism
of crossed modules.
\end{proof}

It is a moot point whether it is possible to replace
weak $\sK$-algebras with $\sK^{\ccr}$-algebras
for a new ``centrally extended'' crossed module
${\sK}^{\ccr}=(H^{\ccr}\rightarrow G)$
where
$1\rightarrow C\rightarrow{H}^{\ccr} \rightarrow H \rightarrow 1$
is a $G$-equivariant central extension.
It is certainly possible to do it for a single
weak $\sK$-algebra $A$: one can use $C=Z(A)^\times$
to achieve this goal.
However, if one wishes a single ${H}^{\ccr}$
to serve all weak $\sK$-algebras, this becomes subtly
dependent on $\sK$. To achieve this
${H}^{\ccr}$ needs to be a $G$-equivariant stem cover of $H$.
Such thing may exist, say if $H$ is perfect.
However, it does not exist, in general: Derek Holt has shown
us an example of finite $H$ and $G$ where such cover does not exist.

The unitality assumption is not onerous: we can always make
a 2-module unital as explained in the following lemma,
whose proof is routine.
As the referee has pointed out,
the second statement is a special case of the general fact
that any monoidal functor between monoidal categories
is monoidally equivalent to a unital monoidal functor.
\begin{lemma}
\label{norm}
If $R : \widetilde{\sK}\rightarrow \tsAut (A\mbox{-}\Mod)$
is a 2-module with  
$R_{\star,\star} ^1(\vg) = [\varpi (\vg)]$
for some group homomorphism
$\varpi : G \rightarrow \Aut (A)$,
then
$$
R^2(\vf,1) = R^2(1,\vg) = R^2(1,1) =\Upsilon (z)
$$
for some element $z$ of the centre $Z(A)$.
Moreover,
$R$ is equivalent to a unital 2-module 
$\Theta : \widetilde{\sK}\rightarrow \tsAut (A\mbox{-}\Mod)$
defined by
$$
\Theta^1 = R^1, \ \ 
\Theta^2_{\star} = R^2_{\star}, \ \ 
\Theta^2 (\vf, \vg) =
  \Upsilon (z^{-1}) \circ 
R^2 (\vf, \vg) \; .
$$
\end{lemma}

For the rest of the section we concentrate on $\sK$-algebras.
Our goal is Morita Theory that will pinpoint when
two $\sK$-algebras give equivalent 2-modules.

%We have observed in Section~\ref{s1} that 
%$\sA^n$
%is an $\sA^1$-module category. 
%The category 
%$A\mbox{-}\Mod$ is an $\sA^1$-module category
%where $M\boxtimes V = M\otimes_\bK V$
%where $A$ acts on $M$ and $V$ is a vector space.
%Each $[\vg]$ is an $\sA^1$-module functor.
%Thus, $\Theta_A$ is a 2-module for $\sK$.

We consider a Morita equivalence
$\Phi : A\mbox{-}\Mod \rightarrow B\mbox{-}\Mod$
together with its quasiinverse
$\Phi^{-1} : B\mbox{-}\Mod \rightarrow A\mbox{-}\Mod$.
We suppose both are equipped with strong module structures.
This data defines a homomorphism of 2-groups
$$
\widetilde{\Phi} :
\tsAut  (A\mbox{-}\Mod ) \rightarrow 
\tsAut  (B\mbox{-}\Mod ) , \ \  
\widetilde{\Phi}^0 (\star)= \star , \ \ 
\widetilde{\Phi}^1_{\star,\star} (F) = \Phi\circ F\circ\Phi^{-1} ,  
$$
$$
\widetilde{\Phi}^1_{\star,\star} (
\xymatrix{
\star
   \ar@/^2pc/[rr]^{F_2}="2"
   \ar@/_2pc/[rr]_{\quad}_{F_1}="1"
&& \star
%   \ar@/^2pc/[ll]_{\quad}^{F_1}="1"
%   \ar@/_2pc/[ll]_{F_2}="2"
\ar@{}"1";"2"|(.2){\,}="7"
\ar@{}"1";"2"|(.8){\,}="8"
\ar@{=>}"7" ;"8"^{\varphi}
} 
) = 
\xymatrix{
\star
   \ar@/^2pc/[rr]^{\Phi^{-1}}="2"
   \ar@/_2pc/[rr]_{\quad}_{\Phi^{-1}}="1"
&& \star
%   \ar@/^2pc/[ll]_{\quad}^{\Phi}="1"
%   \ar@/_2pc/[ll]_{\Phi}="2"
   \ar@/^2pc/[rr]_{\quad}^{F_2}="3"
   \ar@/_2pc/[rr]_{F_1}="4"
\ar@{}"1";"2"|(.2){\,}="7"
\ar@{}"1";"2"|(.8){\,}="8"
\ar@{=>}"7" ;"8"^{Id}
&& \star
%   \ar@/^2pc/[ll]_{\quad}^{F_1}="1"
%   \ar@/_2pc/[ll]_{F_2}="2"
   \ar@/^2pc/[rr]_{\quad}^{\Phi}="5"
   \ar@/_2pc/[rr]_{\Phi}="6"
\ar@{}"4";"3"|(.2){\,}="9"
\ar@{}"4";"3"|(.8){\,}="10"
\ar@{=>}"9" ;"10"^{\varphi}
&& \star
%   \ar@/^2pc/[ll]_{\quad}^{\Phi^{-1}}="1"
%   \ar@/_2pc/[ll]_{\Phi^{-1}}="2"
\ar@{}"5";"6"|(.2){\,}="11"
\ar@{}"5";"6"|(.8){\,}="12"
\ar@{=>}"12" ;"11"^{Id}
} 
.$$
The 2-isomorphism $\widetilde{\Phi}^2_{\star}$ and
the compatibilities $\widetilde{\Phi}^2_{\star,\star,\star}$
utilise the natural isomorphism between $\Phi^{-1} \circ \Phi$
(or $\Phi \circ \Phi^{-1}$) and the corresponding identity functors.
Thus, $\widetilde{\Phi}$ is neither strict, nor unital, in general.
We say that two $\sK$-algebras $A$ and $B$ are 
%{\em strictly $\sK$-Morita equivalent} 
%if there exist quasiinverse Morita equivalencesutilisutilis
%$\Phi$ and $\Phi^{-1}$
%such that
%$\Phi^\prime \circ \Theta_A = \Theta_B: \widetilde{\sK}
%\rightarrow 
%\tsAut (A\mbox{-}\Mod)$.
%Similarly, 
%we say that two $\sK$-algebras $A$ and $B$ are 
{\em %weakly 
$\sK$-Morita equivalent} 
if there exist quasiinverse Morita equivalences
$\Phi$ and $\Phi^{-1}$
with some choice of a strong module structure 
such that the homomorphisms of 2-groups
$\widetilde{\Phi} \circ \Theta_A, \Theta_B:
\widetilde{\sK}\rightarrow\tsAut (B\mbox{-}\Mod)$ 
are equivalent. %naturally 2-isomorphic.

Recall that a Morita context $(A,B, \,_AM_B, \,_BN_A, \alpha, \beta)$ 
is {\em nondegenerate}
if $\alpha$ and $\beta$ are isomorphisms.
We say the Morita context
is {\em $\sK$-equivariant}
if  $A$ and $B$ are $\sK$-algebras and
\begin{mylist}
\item[(1)] both $M$ and $N$ are $G$-modules,
\item[(2)] the bimodule actions
$A\otimes_\bK M \otimes_\bK B \rightarrow M$
and
$B\otimes_\bK N \otimes_\bK A \rightarrow N$
are homomorphisms of $G$-modules,
%$\vg\cdot (a m b) = ( \vg\cdot a)(\vg \cdot m)(\vg\cdot b)$
%for all $a\in A$, $b\in B$, $\vg\in G$, $m\in M$,
%\item[(3)] 
%$\vg\cdot (b na ) = ( \vg\cdot b)(\vg \cdot n)(\vg\cdot a)$ 
%for all $a\in A$, $b\in B$, $\vg\in G$, $n\in N$,
\item[(3)] the bimodule maps
$\alpha : M\otimes_BN \rightarrow A$ and $\beta : N\otimes_AM \rightarrow B$ 
are homomorphisms of $G$-modules,
\item[(4)] $\,^{\partial(h)}m= \omega_A(h) m \omega_B(h)^{-1}$ and 
$\,^{\partial(h)}n= \omega_B(h)n \omega_{A}(h)^{-1}$ 
for all $h\in H$, $n\in N$, $m\in M$.
\end{mylist}
Observe that axiom (2) manifests in the identities
$\,^\vg(a m b) = \,^ \vg a \,^\vg  m \,^\vg b$.
%for all $a\in A$, $b\in B$, $\vg\in G$, $m\in M$,
The following theorem characterises $\sK$-Morita equivalences 
within the context of Morita theory.

\begin{theorem}
\label{G_morita}
The following statements about associative $\sK$-algebras
$A$ and $B$ are equivalent:
\begin{mylist}
\item[(1)] The 2-modules $\Theta_A$ and $\Theta_B$ for
$\sK$ are equivalent. %2-isomorphic.
\item[(2)] $A$ and $B$ are %strictly 
$\sK$-Morita equivalent.
\item[(3)] There exists a nondegenerate $\sK$-equivariant 
Morita context $(A,B, \,_AM_B, \,_BN_A, \alpha, \beta)$.
\end{mylist}
\end{theorem}
\begin{proof}
{\em (1)$\Rightarrow$(2):} Let $\psi : \Theta_A \Rightarrow \Theta_B$
be a natural 2-isomorphism, $\phi : \Theta_B \Rightarrow \Theta_A$
its quasiinverse. Then 
$\psi^1_\star : A\mbox{-}\Mod\rightarrow B\mbox{-}\Mod$
and
$\phi^1_\star : B\mbox{-}\Mod \rightarrow A\mbox{-}\Mod$
are quasiinverse category equivalences.
Let $\Phi\coloneqq\psi_\star$,  $\Phi^{-1}\coloneqq\phi_\star$.
We have three 2-functors and two 2-isomorphisms between them:
$$
\widetilde{\Phi} \circ \Theta_A \cong \Theta_A \cong  \Theta_B \; : \; 
\widetilde{\sK}\rightarrow\tMd .
$$
Hence, $\widetilde{\Phi} \circ \Theta_A$ and 
$\Theta_B$ are equivalent %2-isomorphic
as 2-functors 
$\widetilde{\sK}\rightarrow\tMd$.
When we consider $\tsAut(B\mbox{-}\Mod)$
as a 2-subcategory of $\tMd$,
it contains all invertible 1-objects and 2-objects,
related to the 0-object  $B\mbox{-}\Mod$.
Consequently, $\widetilde{\Phi} \circ \Theta_A$ and 
$\Theta_B$ are equivalent %2-isomorphic
as 2-functors 
$\widetilde{\sK}\rightarrow\tsAut(B\mbox{-}\Mod)$,
that is, 
$A$ and $B$ are
$\sK$-Morita equivalent.

{\em (3)$\Rightarrow$(1):}
A nondegenerate $\sK$-equivariant context 
$(A,B, \,_AM_B, \,_BN_A, \alpha, \beta)$
gives quasiinverse 2-morphisms 
$\phi:\Theta_A \rightarrow\Theta_B$ and $\psi=\phi^{-1}$
of 2-modules. We define
%The first parts of the 2-transformation data are the Morita functors:
$$
\phi^1_\star : 
A\mbox{-}\Mod \rightarrow B\mbox{-}\Mod, \ 
\phi^1_\star (P) = N\otimes_AP, \ \ 
\psi^1_\star : B\mbox{-}\Mod \rightarrow A\mbox{-}\Mod, \ 
\psi^1_\star (T) = M\otimes_B T.
$$
Let us now discuss 
in detail
%explain the second part 
%$\Phi^2_{\star,\star}$
the natural transformation
%in detail (
$$
\phi^2_{\star,\star} : 
\Theta^1_{A\;\star,\star} \diamond \phi^1_\star 
\Rightarrow 
\phi^1_\star  \diamond \Theta^1_{B\;\star,\star}
\ \mbox{ where } \ 
\Theta^1_{A\;\star,\star} \diamond \phi^1_\star , \;
\phi^1_\star  \diamond \Theta^1_{B\;\star,\star}
%\phi^1_\star \diamond \Theta_A ,
%\Theta_B \diamond \phi^1_\star
:
\widetilde{\sK}_1 (\star,\star)\rightarrow \tMd_1
(A\mbox{-}\Mod,B\mbox{-}\Mod)
$$
and
$\psi^2_{\star,\star}$ is defined similarly.
The objects of the category
$\widetilde{\sK}_1 (\star,\star)$
is the group $G$. Thus, the natural transformation 
$\phi^2_{\star,\star}$
is defined 
for all $A$-modules $P$ and $\vg\in G$:
$$
\phi^2_{\star,\star} (\vg)_P: 
N\otimes_A (P^{[\vg]})
\rightarrow
(N\otimes_A P)^{[\vg]}, \ \ \ 
n\otimes p \mapsto \,^{\vg}n \otimes p.
$$
Let us verify that it is
a homomorphism of $B$-modules: 
$$
\phi^2_{\star,\star} (\vg)_P (bn\otimes p) =
\,^{\vg}(bn)\otimes p =
\,^{\vg}b\; \,^{\vg}n\otimes p =
b\cdot^{[\vg]} (\,^{\vg}n\otimes p) = 
b\cdot^{[\vg]}  \phi^2_{\star,\star} (\vg)_P (n\otimes p) 
$$
for all $b\in B$, $n\in N$, $p\in P$. The fact that
$\phi^2_{\star,\star}$
is  a natural transformation is equivalent to commutativity of the
following diagrams
$$
\begin{CD}
N\otimes_A P^{[\vf]} @>{\phi^2_{\star,\star} (\vf)_P}>>     (N\otimes_A P)^{[\vf]} \\
@A{{\mathrm Id}_N\otimes h}AA              @A{h}AA \\
N\otimes_A P^{[\vg]} @>{\phi^2_{\star,\star} (\vg)_P}>> (N\otimes_A P)^{[\vg]}
\end{CD}
%\hspace{1in}
\ \ 
\mbox{ for each 2-morphism }
\ \ 
\xymatrix{
\star
   \ar@/^2pc/[rr]_{\quad}^{\vf}="2"
   \ar@/_2pc/[rr]_{\vg}="1"
&& \star
%   \ar@/^2pc/[ll]_{\quad}^{\vg_1}="1"
%  \ar@/_2pc/[ll]_{\vg_2}="2"
\ar@{}"1";"2"|(.2){\,}="7"
\ar@{}"1";"2"|(.8){\,}="8"
\ar@{=>}"7" ;"8"^{h}
}
\  
. 
$$
An element $n\otimes p\in N\otimes_A P^{[\vg]}$ is mapped via the
bottom-right corner to
$$
n\otimes p
\mapsto
\,^{\vg}n\otimes p
\mapsto
\omega_B(h)\,^{\vg}n\otimes p
%=
%\,^{\partial (h)}(\omega_B(\,^{\vg}h)\,^{\vg}n)\otimes p
%\vf^{-1}(\omega_B(h)) \vf^{-1}(n)\otimes p
%=
%\omega_B(\,^{\vf^{-1}}h) \vf^{-1}(n)\otimes p
$$
and via the
top-left corner to
$$
n\otimes p
\mapsto
n\otimes \omega_A(h)\cdot p
\mapsto
\,^{\vf}n\otimes \omega_A(h)\cdot p
=
\,^{\partial (h)}(\,^{\vg}n) \omega_A(h) \otimes p
%=
%\,^{\vf^{-1}\partial(h)}n\otimes \,^{\vf^{-1}}\omega_A(h)\cdot^{[\vf]}p
%=
%\,^{\vf^{-1}}(  \,^{\partial(h)}n \; \omega_A(h)  )\otimes p
$$
that are equal by axiom (4) of an $\sK$-equivariant Morita context.
There are still two compatibility conditions to verify.
The second compatibility condition follows easily from the fact that 
$\phi^2_{\star,\star} (1_G)_P$ is the identity operator. 
The first compatibility condition is checked for $\va,\vb\in G$:
$$
\phi_{\va,P} :
(N\otimes_A P^{[\va]})^{[\vb]}
\rightarrow
(N\otimes_A P)^{[\va][\vb]}, \ \ 
n\otimes p
\mapsto
\,^{\va}n\otimes p
\ ,
$$
$$
\phi_{\vb,P} :
N\otimes_A P^{[\va][\vb]}
\rightarrow
(N\otimes_A P^{[\va]})^{[\vb]}, \ \ 
n\otimes p
\mapsto
\,^{\vb}n\otimes p
\ ,
$$
$$
\phi_{\va,\vb,P} :
N\otimes_A P^{[\va][\vb]}
\rightarrow
(N\otimes_A P)^{[\va][\vb]}, \ \ 
n\otimes p
\mapsto
\,^{(\va\vb)}n\otimes p.$$
Now the first compatibility boils down to the 
left action axiom 
$\,^{(\va\vb)}n = \,^{\va}(\,^{\vb}n)$.

We can finish the proof at this point because
%$\phi$ is a
%2-isomorphism. However, it is worth pointing out
%that
$\phi$ and $\psi$ are
quasiinverse 2-morphisms.
%, but we have not set up a terminology in
%this paper to discuss this issue.

{\em (2)$\Rightarrow$(3):}
A $\sK$-Morita equivalence is a $G$-Morita equivalence.
In particular, it is given (up to natural transformations)
by a $G$-equivariant Morita context
%$(A,B, \,_A\widetilde{M}_B, \,_B\widetilde{N}_A, \widetilde{\alpha}, \widetilde{\beta})$
$(A,B, \,_AM_B, \,_BN_A, \alpha, \beta)$
\cite[Theorem 2.3]{GRR11}, i.e., a Morita context satisfying
axioms
(1)--(3) of the definition of a $\sK$-equivariant Morita context.

It remains to establish axiom (4). An analogous property
$\,^{\partial (h)}a=\omega_A(h)a \omega_A(h)^{-1}$ holds for $a\in A$,
$h\in H$.
Let us explain briefly how this property gives axiom (4) for $N=\Phi
(A)$,
leaving technical verifications to an inquisitive reader. The
$\sK$-Morita equivalence $\Phi$ gives an equivalence of 2-modules
$\Theta_A$ and $\Theta_B$. Let us now express the property as an
identity on some structure morphisms $A\rightarrow A^{[\partial (h)]}$: 
$$
\partial (h)_A
\; = \;
\Theta^1_{A\, \star,\star} (1_G\stackrel{h}{\Longrightarrow}\partial (h))_A 
\circ 
R({\omega_A({h}^{-1}}))
$$
where $\partial (h)_A:A\rightarrow A^{[\partial (h)]}$ is the group
action
and  
$R(\omega_A({h}^{-1})):A\rightarrow A$ is the $A$-module endomorphism
given by the right multiplication.
The
$\sK$-Morita equivalence $\Phi$ preserves this identity: the equality
$$
\partial (h)_N
\; = \;
\Theta^1_{N\, \star,\star} (1_G\stackrel{h}{\Longrightarrow}\partial (h))_N 
\circ 
R({\omega_A({h}^{-1}})) \in \hom (N,N^{[\partial (h)]})
$$
holds and it is exactly axiom (4).
\end{proof}

\begin{center}
\section{Structure of 2-representations}
\label{s2a}
\end{center}

There is a critical difference between a 2-module and a 2-representation
of $\widetilde{\sK}$. 
An action of $G$ on $A\mbox{-}\Mod$ for an arbitrary associative
algebra $A$ could be hard to pinpoint: thanks to Theorem~\ref{G_morita}
we need to compute the group of invertible $A$-$A$-bimodules, so called
the non-commutative Picard group of $A$.
However, if $A$ is semisimple, these bimodules could be tracked
modulo {\em realizability} of cocycles.
Let $Z=\bK^n$ with some action of $G$.
We call a cocycle $\mu\in Z^2(G,Z^\times)$
{\em realizable} if there exists a projective $Z$-semilinear
representation
of $G$ on a finite dimensional faithful $Z$-module $M$ with cocycle
$\mu$:
$$
\rho: G \rightarrow \GL (\,_{\bK}M), \
\rho(\vg) (zm) = (\vg\cdot z)  \rho(\vg) (m), \
\rho (\vf\vg) = \mu(\vf,\vg)\rho (\vf) \rho (\vg)
\ \mbox{ for all } \ 
\vf,\vg\in G, z\in Z, m\in M. 
$$
If $G$ is finite, all cocycles are realizable
because projective representations are representations
of a finite dimensional twisted skew group algebra $Z_{\mu} G$
(where $z\vg \cdot z^\prime \vh = z \,^\vg{z^\prime} \mu^{-1}(\vg , \vh) \vg\vh$). 
If $G$ is infinite, it is no longer the case.
Suppose $G$ is simple, $\mu$ is non-trivial.
If $\mu$ is realizable, then the action of $G$ on
$A\coloneqq \End_Z M$ is non-trivial, hence faithful.
But $G$ could be of higher cardinality than $\bK$
or non-linear. In both case,
non-trivial cocycles are not realizable.

Let $e_1, \ldots , e_n$ be all primitive idempotents in $Z$,
$G_i\leq G$ -- the stabiliser of $e_i$.
We decompose the cocycle
into components $\mu=(\mu_i) = \sum_i \mu_i e_i$,
the components $\mu_i$ are cocycles for $G_i$, not the full $G$.
Let us collect basic facts about realizability:
\begin{lemma}
  \label{realize}
  \begin{enumerate}
  \item
    $\mu$ is realizable if and only if
  all $\mu_i|_{G_i^2}$
  are realizable.
\item If $\mu\sim\nu$, then
  $\mu$ is realizable if and only if
      $\nu$ is realizable.
\item Realizable cocycles form a subgroup
  $Z^2_{{\mathrm {\tiny r}}{\mathrm {\tiny e}}}  (G,Z^\times)\leq Z^2 (G,Z^\times)$.
  \end{enumerate}
  \end{lemma}
\begin{proof}
%  \begin{enumerate}
  %  \item
  {\bf 1.} The only if statement is obvious.

    In the opposite direction, the subgroup $H=\cap_i G_i$
    is of finite index in $G$ and all cocycles $\mu_i|_{H^2}$
    are realizable via projective representations $V_i$.
    Then $V\coloneqq \oplus_i V_i$ realizes the cocycle $\mu|_{H^2}$,
    while $Z_{\mu}G \otimes_{Z_{\mu}H}V$ realizes $\mu$ itself.

    %\item
    {\bf 2.}
    This is easy to verify for $Z=\bK$. The general statement follows
    from 1.

    %\item
    {\bf 3.}
    Tensor products and contragradient representations
    verify this for $Z=\bK$. The general statement follows
    from 1.
%\end{enumerate}
\end{proof}

It is known (albeit in a different language) that
a 2-representation of a finite group $G$ 
comes from 
a split semisimple $G$-algebra $A$ 
(cf. \cite{Ost} and \cite[Theorem 4.3]{GRR11}).
%For the benefit of the reader
%we choose to give a self-contained 
%elementary proof of this known fact: 
To state a general statement covering infinite groups
we need {\em a semimatrix algebra},
by which we mean a finite direct sum of $\End_\bK V$
for some, not necessarily finite dimensional vector spaces $V$.
%A similar statement is true for an arbitrary group
%$G$ subject to realizability of a key cocycle:
\begin{lemma} (cf. \cite[Th. 5.5]{Elg}) 
  \label{morita_dec_group}
  The following statements hold for
a 2-representation $R$ of a group $G$.
  \begin{enumerate}
    \item
The 2-representation $R$ 
admits an associated cocycle $\mu\in Z^2(G,Z^\times)$
where $Z=\bK^n$, $n=R^0(\star)$.
\item The cohomology class $[\mu]$ is canonically associated to $R$.
  \item If $\mu$ is realizable, then
there exists
a split semisimple 
$G$-algebra $A$ such
that $R$ and $\Theta_A$ are equivalent. %2-isomorphic.
%If $\Theta$ is strict, $A$ can be chosen to be strict.
\item
%If $\mu$ is not realizable, then
In general, there exists
a semimatrix 
$G$-algebra $A$ such
that $R$ is equivalent to a sub-2-representation of
the 2-module $\Theta_A$.
  \end{enumerate}
\end{lemma}
\begin{proof}
Using Theorem~\ref{2vec_eq}, interpret $R$
as a 2-functor $G \rightarrow \tVc$.
Let $L_1,L_2 \ldots L_n$ be all simple non-isomorphic objects
of the category $\mC=R^0 (\star)$.
Without loss of generality, $\mC = Z\mbox{-}\Mod$
%where $Z=\bK^n=\oplus_i \hom(L_i,L_i)$,
while $R$ is given by the algebraic datum
as in Theorem~\ref{2vec_eq}
(we use $\otimes\coloneqq \otimes_\bK$ as well as $\otimes_Z$):
$$
MR^1_{\star,\star} (\vg) =
\bigoplus_{i,j}
L_j \otimes \hom_Z (L_i,M) \otimes V_{i,j}(\vg), \ \ \ \ 
R^2_{\star} = (\gamma_{i}\delta_{i,j}), \ \ 
\gamma_i : \bK \xrightarrow{\cong} V_{i,i} (1).  
$$
An element $\vg\in G$ yields a permutation
$\sigma = \sigma (\vg)\in S_n$ so that
$L_i R^1(\vg) \cong L_{\sigma (i)}.$

{\bf 1.}   Let us choose bases in all one-dimensional
spaces constituting
$R^1 (\vg)$.
The choice for $R^1 (1)$ is canonical, while the rest
are not: 
$$
b_i = b_i (\vg) \in V_{i,\sigma (i)} (\vg), \ \ \ 
b_i (1) = \gamma_i (1) \in V_{i,i} (1).
$$
The collection of 2-isomorphisms $R^2 (\vf, \vg)$
yields
a 2-cocycle $\mu\in Z^2 (G,Z^\times)$
whose components
$
\mu_i (\vg,\vh)
$
(where 
$\vg,\vh\in G$, $i=\sigma (\vg\vh)(j)= \sigma (\vh)(k)$) 
are defined as a composition (notice $\End_Z L_i=\bK$): 
$$
L_i
\xrightarrow{v \mapsto v \otimes b_j (\vg) \otimes b_{k} (\vh)}
(L_jR^1(\vg))R^1(\vh) =
L_i\otimes V_{j,k} (\vg)\otimes V_{k,i} (\vh)
\xrightarrow{R^2 (\vg, \vh)}
L_jR^1(\vg \vh) =
L_i \otimes V_{j,i} (\vg\vh)
\xrightarrow{v \otimes b_j (\vg\vh)\mapsto v}
L_i .
$$
The cocycle
condition follows from the pentagon condition for $R^2_{\star,\star,\star}$. 

{\bf 2.} A different choice of elements
$b_i (\vg)$ yields a cohomologous cocycle.

{\bf 3.} 
Let us realize the cocycle $\mu^{-1}$ on a finite-dimensional
$Z$-module $S$.
If $S_i = e_iS$ is a Peirce decomposition of $S$, then
we get a split semisimple $G$-algebra
$A \coloneqq \End_Z S = \oplus_i \End_\bK S_i$.
Let $L\coloneqq \oplus_i L_i$. 
Let us define an equivalence of 2-representations
$\psi : R \rightarrow \Theta_A$:
$$
\psi^1_{\star} (M)
\coloneqq 
\hom_Z (L,M) \otimes_Z S
=
\bigoplus_{i} 
\hom_Z (L_i,M) \otimes S_{i},
$$
$$
\psi^2_{\star,\star \ M,\vg}: 
\bigoplus_{i,j}
S_j \otimes \hom_Z (L_i,M) \otimes V_{i,j}(\vg)
\rightarrow \bigoplus_{i}
S_i^{[\vg]} \otimes \hom_Z (L_i,M), \
s \otimes \phi \otimes b_i(\vg) \mapsto \,^\vg s \otimes \phi. 
$$
While the quasi-invertibility of $\psi$ is apparent, we need to check
the two coherence conditions from Section~\ref{s1}.
The components of the first coherence condition are
computed for any pair $\vf,\vg\in G$.
It is convenient to think of $V(\vg)=\oplus_{i,j}V_{i,j}(\vg)$
as a $Z-Z-$bimodule
with a generator $b(\vg) = \sum_i b_{i}(\vg)$ for this computation:
$$
\psi_{\vf,\vg}:
(\hom_Z (L,M) \otimes_Z S) \otimes_Z V(\vf) \otimes_Z V(\vg)
\rightarrow
\hom_Z (L,M) \otimes_Z S^{[\vf \vg]}, \ \ 
\phi \otimes s \otimes b(\vf) \otimes b(\vg)
\mapsto \mu(\vf,\vg)^{-1} \phi \otimes \,^{\vf\vg} s, 
$$
$$
\psi_{\vg}:
(\hom_Z (L,M) \otimes_Z S) \otimes_Z V(\vf) \otimes_Z V(\vg)
\rightarrow
(\hom_Z (L,M) \otimes_Z S^{[\vg]}) \otimes_Z V(\vf), \ \ 
\phi \otimes s \otimes b(\vf) \otimes b(\vg)
\mapsto \phi \otimes b(\vf) \otimes \,^{\vg} s, 
$$
$$
\psi_{\vf}:
(\hom_Z (L,M) \otimes_Z S^{[\vg]}) \otimes_Z V(\vf)
\rightarrow
\hom_Z (L,M) \otimes_Z S^{[\vf \vg]}, \ \ 
\phi \otimes s \otimes b(\vf)
\mapsto \phi \otimes \,^{\vf} s.  
$$
The equality
$\psi_{\vf,\vg}=
\psi_{\vf}
\circ
\psi_{\vg}
$
is precisely the cocycle condition for the projective $G$-module $\,_ZS$.
The second coherence condition degenerates since
$\vi_{\sC}$ and $\Theta^2_{A\, \star}$
are identities. 
The coherence condition becomes
$$
\psi^1_{\star} (M)
%\xrightarrow{LUn_{F^0(x),G^0(x)}^{-1}}
%\psi^1_{x} \diamond \vi_{F^0(x)}
\xrightarrow{R^{1\; -1}_{\star}}
\psi^1_{\star} (R^1(1)
\diamond M) 
\xrightarrow{\psi^2_{\star,\star}(1)}
\psi^1_{\star} (M)
\diamond
\Theta^1_{A   \; \star,\star} (1)
\xrightarrow{=}
\psi^1_{\star} (M), 
$$
once applied to an object $M\in \sC$. This
is equal to the identity:
$$
\phi \otimes s
\xrightarrow{R^{1\; -1}_{\star}}
\phi \otimes s \otimes b(1)
\xrightarrow{\psi^2_{\star,\star}(1)}
\,^1 s \otimes \phi
\xrightarrow{=}
\phi \otimes s.
$$

{\bf 4.} If $\mu$ is not necessarily realizable, the argument in Part~3
still works.
The only difference is that $S$ is infinite dimensional.
This means that $S_i$ is infinite dimensional for some $i$
and the corresponding direct summand
$\End_ZS_i$
is an infinite full matrix algebra.
Consequently, the image of $\psi$ in $\Theta_A$
consists of those semisimple $A$-modules
that are direct sums of standard modules.
(Notice that the infinite full matrix algebra
has also non-standard simple modules.) 
\end{proof}

This result extends to 2-groups
using Extension Lemma~\ref{Extension}.

\begin{cor}
  \label{morita_dec}
  Consider a 2-representation $R$ of
  $\widetilde{\sK}$
  such that its restriction to $G$
  has a realizable cocycle $\mu$. 
Then there exists
a split semisimple $\sK$-algebra $A$ such
that $R$ and $\Theta_A$ are equivalent. %2-isomorphic.
Furthermore, 2-representations
$\Theta_A$ and $\Theta_B$ are equivalent %2-isomorphic
if and only if
$A$ and $B$ are 
${\sK}$-Morita equivalent.
\end{cor}
\begin{proof}
  Lemma~\ref{morita_dec_group}
  gives a split semisimple $G$-algebra $A$.
Lemma~\ref{norm} makes the 2-representation unital. 
Lemma~\ref{Extension} extends 
the $G$-algebra structure on $A$ to a $\sK$-algebra structure.
Theorem~\ref{G_morita} establishes the second statement.
\end{proof}

Let us call (weak) 2-representation of $\widetilde{\sK}$, 
described in Corollary~\ref{morita_dec},
{\em realizable}.
To deal with not necessarily realizable weak 2-representations,
we need to contemplate a semimatrix algebra
$A= \oplus_i \End_\bK S_i$
and its full subcategory $A\mbox{-}\Mod^\circ$ of $A\mbox{-}\Mod$
consisting of semisimple representations whose direct
summands are isomorphic to $S_i^{[\vg]}$ for some automorphisms $\vg$.

If $A$ is a semimatrix $\sK$-algebra,
we can describe the structure of the 2-representation $R=\Theta_A^\circ$
of $\widetilde{\sK}$
afforded by $A$ in the same way as
in the proof of 
Theorem~\ref{2vec_eq}.
Notice that $\Theta_A^\circ=\Theta_A$
for a split semisimple algebra $A$.
Let us depict it using the bicategory $\tVec^\bK$ for the benefit of the reader:
%Observe that
%$A=\oplus_{i=1}^n \End_\bK (S_i)$ where $S_i$ are finite dimensional
%$\bK$-vector spaces. Then $S_1, \ldots , S_n$ are 
%all non-isomorphic irreducible $A$-modules. Then
$$
R^0(\star) = n, \ \ 
R^1(\vg)_{i,j} = \hom (S_i, S_j^{[\vg]}), \ \
{R^2_\star}\,_{i,i} : \hom (S_i,S_i)\xrightarrow{\cong} \bK,
$$
$$
R^2 ({\vf,\vg}):
\bigoplus_j
R^1(\vg)_{i,j} \otimes 
R^1(\vf)_{j,k} 
\rightarrow
R^1(\vg\vf)_{i,k} , \ \ 
\sum_j \varphi_{i,j} \otimes \psi_{j,k} 
\mapsto 
\sum_j
%\gamma^{-1}_{g_2,g_1}(S_i) \circ\varphi_{i,j}^{[g_2]} \circ
%\psi_{j,k}
\psi_{j,k} \circ\varphi_{i,j}
, \ \
$$
$$
R^1 ({\vg,h})
= R^1 (
%\xymatrix{
%\star
%   \ar@/^2pc/[rr]_{\quad}^{\partial (h)\vg}="2"
%   \ar@/_2pc/[rr]_{\vg}="1"
%&& \star
%   \ar@/^2pc/[ll]_{\quad}^{\vg_1}="1"
%  \ar@/_2pc/[ll]_{\vg_2}="2"
%\ar@{}"1";"2"|(.2){\,}="7"
%\ar@{}"1";"2"|(.8){\,}="8"
%\ar@{=>}"7" ;"8"^{h}}
\vg \stackrel{h}{\Longrightarrow} \partial h \, \vg )
: \hom (S_i, S_j^{[\vg]}) \rightarrow \hom (S_i, S_j^{[\partial (h)\vg]}), \ \ 
: \varphi \mapsto \omega_A(h) \circ \varphi.
$$

It is useful to get 
a version of Corollary~\ref{morita_dec}
for semimatrix algebras. We
need a version of Lemma~\ref{nat_trans}:
\begin{lemma}
  \label{nat_trans_1}
  Suppose $\vf$ and $\vg$ are automorphisms of a semimatrix algebra $A$,
  $[\vf]$, $[\vg]$ are twists of the category on $A\mbox{-}\Mod^\circ$.  
  Then the map (cf. Lemma~\ref{nat_trans})
  $$
  \Upsilon:
  \{x\in A \,\mid\, \forall a\in A \ \  x a=\vg(\vf^{-1}(a))x\}
  \rightarrow
  \mbox{\rm Nat.Trans} ([\vf],[\vg]), \ \
  \Upsilon (x)_M : m \mapsto x \cdot m
  $$
  is a bijection.
\end{lemma}
\begin{proof} Thanks to Lemma~\ref{nat_trans}, we just need to construct
an inverse map $\Xi^\circ$.
Given a natural transformation $\varphi$, its value
$\varphi_i \coloneqq \varphi_{S_i}$
is a linear map $S_i\rightarrow S_i$.
Hence, $\varphi_i \in \End_\bK S_i \subseteq A$
and the inverse map is 
  $$
\Xi^\circ:
 \mbox{\rm Nat.Trans} ([\vf],[\vg])
  \rightarrow
   \{x\in A \,\mid\, \forall a\in A \ \  x a=\vg(\vf^{-1}(a))x\}, \ \ 
  \Xi^\circ (\varphi) = \sum_i \varphi_i .
  $$  
\end{proof}

Now we can prove a version of
Corollary~\ref{morita_dec}
for non-realizable cocycles.
It lacks uniqueness statement.

\begin{cor}
  \label{morita_dec_1}
For a 2-representation $R$ of
  $\widetilde{\sK}$
there exists
a semimatrix $\sK$-algebra $A$ such
that $R$ and $\Theta_A^\circ$ are equivalent. %2-isomorphic.
In particular, any 2-representation is equivalent to a strict 2-representation.
\end{cor}
\begin{proof}
  Lemma~\ref{morita_dec_group}
  yields a semimatrix $G$-algebra $A$
  so that $R\cong \Theta_A^\circ$
  as 2-representations of $G$.
Lemma~\ref{norm} makes the 2-representation unital. 
Transformations  $R^1 (\vg, x)$
and $R^2 (\vf,\vg)$
become natural transformations between
twists of $A\mbox{-}\Mod^\circ$,
hence, 
Lemma~\ref{nat_trans_1} associates
elements of $A$ to all
$R^1 (\vg, x)$  and $R^2 (\vf,\vg)$.
All the proofs in  Lemma~\ref{Extension}
work for $A\mbox{-}\Mod^\circ$ for a semimatrix algebra $A$,
yielding the required $\sK$-algebra structure on $A$.
\end{proof}

Using Corollary~\ref{morita_dec} and 
Corollary~\ref{morita_dec_1},
we can describe various constructions on
realizable 2-representations
of $\widetilde{\sK}$.
It is instructive for the reader to verify that each of the
constructions depend only on the 2-representation $R$
but not on its particular realization $R=\Theta_A$
or $R=\Theta_A^\circ$.
Tensor products
and directs sums
of arbitrary 2-representations
are defined by Barrett and Mackaay \cite{BaMa}.
Our definitions agree with theirs and
we follow their notation.

{\em A direct sum} of $\Theta_A^\circ$ and $\Theta_B^\circ$ comes from the direct
sum of algebras:
$$
\Theta_A^\circ \boxplus \Theta_B^\circ
\coloneqq
\Theta_{A\oplus B}^\circ.
$$
The $\sK$-structure $\omega_{A\oplus B}$ on the direct sum is the obvious one:
$$
\omega_{A\oplus B} (\vg) = 
(\omega_{A} (\vg), \omega_{B}(\vg)), \ \  
\omega_{A\oplus B} (h) = \omega_{A} (h) \oplus \omega_{B}(h).
$$
The degree of $\Theta_A^\circ \boxplus \Theta_B^\circ$
is the sum of the degrees. A similar description for
{\em the tensor product}: 
%of $\Theta_A$ and $\Theta_B$ comes from the tensor
%product:
$$
\Theta_A \boxtimes \Theta_B
\coloneqq
\Theta_{A\otimes B}
$$
works only for split semisimple algebras. A slight modification
is required in general
$$
\Theta_A^\circ \boxtimes \Theta_B^\circ
\coloneqq
\Theta_{A \widehat{\otimes} B}^\circ, \
A \widehat{\otimes} B \coloneqq
\bigoplus_{i,j} \End_\bK (S_i\otimes T_j)
\ \mbox{ if } \ 
A = \bigoplus_{i} \End_\bK S_i , \ 
B = \bigoplus_{j} \End_\bK T_j \; .
$$
Since
$A \otimes B$ is a subalgebra of
$A \widehat{\otimes} B$, 
the  $\sK$-structures  $\omega_{A\otimes B}$
and $\omega_{A\widehat{\otimes} B}$
are given by the same obvious formula:
$$
\omega_{A\widehat{\otimes} B} (\vg) : a\otimes b \mapsto
\omega_{A} (\vg) (a) \otimes \omega_{B}(\vg)(b), \ \  
\omega_{A\widehat{\otimes}B} (h) = \omega_{A} (h) \otimes \omega_{B}(h).
$$
The degree of $\Theta_A^\circ \boxtimes \Theta_B^\circ$
is the product of the degrees. {\em The contragradient 2-representation}
of $\Theta_A$ can be defined using the opposite multiplication algebra
$A^{op}$:
$$
\Theta_A^\ast
\coloneqq 
\Theta_{A^{op}}, \ \
\omega_{A^{op}} (\vg) = 
\omega_{A} (\vg), \ \ 
\omega_{A^{op}} (h) = 
\omega_{A} (h)^{-1}. 
$$
It does not work for semimatrix algebra because
$A$ has finite columns but $A^{op}$ has finite rows.
The contragradient 2-representation can be defined using the dual
vector spaces:
$$
\Theta_A^{\circ \; \ast}
\coloneqq 
\Theta_{A^{\sharp}}, \
A^\sharp \coloneqq
\bigoplus_{i} \End_\bK S_i^\ast
\ \mbox{ if } \ 
A = \bigoplus_{i} \End_\bK S_i , \ 
\omega_{A^{\sharp}} (\vg) = 
\omega_{A} (\vg)^\ast, \ \ 
\omega_{A^{\sharp}} (h) = 
\omega_{A} (h)^{\ast} \; .
$$
The following immediate
fact will
be useful.
\begin{lemma} 
\label{Lin2}
The set of equivalence classes 
of degree one  2-representations
$\tRep_1 (\sK)$ is an abelian group
under the tensor product $\boxtimes$. 
The inverse is given by the contragradient representation. 
Realizable degree one 2-representations form a subgroup. 
\end{lemma}
%\begin{proof}
%\end{proof}
{\em A sub-2-representation} of $\Theta_A^\circ$ is $\Theta_B^\circ$
where $A$ is split into a direct sum of ideals
$A=B\oplus C$ such that $B$ is $G$-stable
(observe that $C$ is $G$-stable automatically).
The group $G$ acts on $B$ and the group
$H$ is mapped to $B$  by the projection to the first factor:
$\omega_B: H \xrightarrow{\omega_A} A \xrightarrow{\pi_1} B$.
A 2-representation $\Theta_A^\circ$
is {\em irreducible}
if no such non-trivial splitting exists.

Let $\sK^\prime = (H^\prime\rightarrow G^\prime)$ be a crossed submodule
of $\sK= (H\rightarrow G)$, i.e. $G^\prime$ is a subgroup of $G$
and  $H^\prime$ is a $G^\prime$-stable subgroup of $H$ 
such that $\partial (H^\prime) \subseteq G^\prime$.
Then {\em the restriction} of $\Theta_A^\circ$ is just $\Theta_A^\circ$
considered
as a 2-representation of $\sK^\prime$.

We define the adjoint functor {\em induction}
only under two special assumptions:
\begin{center}
$H^\prime =H$ 
\ \ and \ \  
$|G:G^\prime |$ is finite.
\end{center}
Let $A$ be a semimatrix $\sK^\prime$-algebra.
The finite $G$-set $G/G^\prime$ has a sheaf of algebras
$G\times_{G^\prime} A$. Its space of global sections 
$$
\widetilde{A} \coloneqq \Gamma (G/G^\prime , G\times_{G^\prime} A) 
= \oplus_{\vg\in G} (\vg, A)/ \sim
$$
is an algebra. Each $(\vg, A)$ is just the algebra $A$.
The direct sum $\oplus_{\vg\in G} (\vg, A)$ is the usual direct sum of
algebras.
The congruence $\sim$ is defined by
$$
(\vg\vh, a ) \sim (\vg, \,^{\vh} a)
\mbox{ for all }
\vg\in G, \vh\in G^\prime, a \in A.
$$
Hence, 
a left transversal $T$ to $G^\prime$ in $G$
gives a presentation 
$$
\widetilde{A} 
= \oplus_{\vt\in T} (\vt, A)
$$
as a direct sum of ideals, i.e., 
$\widetilde{A}$
is a semimatrix algebra as each $(\vt,A)$ is just $A$.
We define the structure of a $\sK$-algebra on $\widetilde{A}$:
$$
\,^{\vg} (\vt,a) \coloneqq (\vg\vt, a) = (\vt^\prime, \,^{\vg^\prime}a), \ \ 
\omega_{\widetilde{A}} (h) = \sum_{\vt\in T} (\vt, \omega_A (\,^{\vt^{-1}}h))
$$
where  $h\in H$, $\vg\in G$ and $\vg\vt=\vt^\prime \vg^\prime$ for
unique $\vt^\prime \in T$. 
Let us verify the axiom:
$$
\,^{\partial (h)} (\vt,a) = 
(\vt (\vt^{-1} \partial (h) \vt), a) = 
(\vt \partial (\,^{\vt^{-1}}h), a) = 
(\vt, \,^{\partial (\,^{\vt^{-1}}h)} a)
= 
(\vt, \omega_A(\,^{\vt^{-1}}h)\, a\,  \omega_A((\,^{\vt^{-1}}h)^{-1}))
= 
\omega_{\widetilde{A}}(h)(\vt,  a)  \omega_{\widetilde{A}}(h^{-1}).
$$
Finally, we define 
${\Theta_A^\circ}\uparrow_{\widetilde{\sK}^\prime}^{\widetilde{\sK}}
\coloneqq \Theta_{\widetilde{A}}^\circ$.
We finish this section with classification of  2-representations.
\begin{theorem}
\label{irr_rep}
Let 
$\sK = (H \rightarrow G)$
be a crossed module. % such $\pi_1 (\sK)$ is finite.
\begin{mylist}
\item{1.} If $\Theta$ is a 2-representation
of $\widetilde{\sK}$, then there exist irreducible
2-representations $\Theta_1, \ldots \Theta_n$ such %\newline
that
$\Theta \cong \Theta_1\boxplus  \Theta_2\boxplus \ldots \Theta_n$.
\item{2.}
The
2-representations $\Theta_1, \ldots \Theta_n$ in 1.
are unique up to a permutation and equivalence. %2-isomorphisms.
\item{3.}
If $\Theta$ is an irreducible 2-representation
of  $\widetilde{\sK}$,
then there exists a subgroup of finite index $G^\prime$ such that
$\partial (H) \subseteq G^\prime \subseteq G$
and a degree one 2-representation $\Psi$
of $\widetilde{\sK^\prime}$
where 
$\sK^\prime = (H \rightarrow G^\prime )$
is a crossed submodule 
such that
$\Theta \cong \Psi\uparrow_{\widetilde{\sK}^\prime}^{\widetilde{\sK}}$
\item{4.}
The pair $(G^\prime, \Psi)$ in 3.
is unique up to conjugation by 
an element of $G$.
\end{mylist}
\end{theorem}
\begin{proof}
  {\bf 1.}
  Without loss of generality,  $\Theta=\Theta_A^\circ$ where
  $A=\oplus_{i=1}^m A_i$
where $A_i$ are full matrix algebras. 
Let $X_1, \ldots X_n$ be the $G$-orbits on the set $\{1,\ldots m\}$,
$B_j\coloneqq\oplus_{i\in X_j} A_i$. Then $\Theta_j\coloneqq \Theta_{B_j}^\circ$
is an irreducible 2-representation and 
$\Theta \cong \Theta_1\boxplus  \Theta_2\boxplus \ldots \Theta_n$.

{\bf 2.} Suppose 
$\Theta \cong \Theta_1\boxplus  \Theta_2\boxplus \ldots
\Theta_n
\xrightarrow{\psi \ \ \cong} 
\Theta^\prime_1\boxplus  \Theta^\prime_2\boxplus \ldots 
\Theta^\prime_{n^\prime}$. 
Then $\psi^1_{\star}: \{1,\ldots m\} \rightarrow \{1,\ldots
m^\prime\}$ is
an isomorphism of $G$-sets. $G$-orbits
on $\{1,\ldots m\}$ come from 
the direct summands
$\Theta_1, \Theta_2,  \ldots\Theta_n$,
while 
$G$-orbits
on 
$\{1,\ldots m^\prime\}$ 
come from the direct summands
$\Theta^\prime_1,  \Theta^\prime_2, \ldots 
\Theta^\prime_{n^\prime}$.
Whenever $\psi^1_{\star}$ moves the $G$-orbit corresponding to  
$\Theta_i$ to the $G$-orbit corresponding to
$\Theta^\prime_j$, the equivalence $\psi$ restricts to an
isomorphism
$\Theta_i \cong \Theta^\prime_j$.

{\bf 3.} In this case the set $\{1,\ldots m\}$ consists of a single
$G$-orbit.
Let $G^\prime$ be the stabiliser of $1$ in this set.
Then $\Psi\coloneqq \Theta_{A_1}^\circ$ is a degree one 2-representation
of $\widetilde{\sK^\prime}$
and 
$\Theta \cong \Psi\uparrow_{\widetilde{\sK}^\prime}^{\widetilde{\sK}}$

{\bf 4.} In another pair 
$(\widetilde{G}, \widetilde{\Psi})$
the group $\widetilde{G}$ is a stabiliser of some $k\in \{1,\ldots
m\}$.
Let $\vg \in G$ be such that $\,^\vg 1 = k$. Then
 $\widetilde{G} = \vg G^\prime \vg^{-1}$
and the 2-representations $\Psi$ and $\vg^{-1}\Psi\vg$ 
of $\widetilde{\sK^\prime}$
are equivalent. %2-isomorphic.
\end{proof}

\begin{center}
\section{Burnside ring of a 2-group}
\label{s3}
\end{center}

Let us consider a crossed module $\sK=(H\xrightarrow{\partial}G)$ 
such that $\pi_1 (\sK)$ is a finite group.
Let $\sS (\sK)$ be the  category of subgroups of $\pi_1 (\sK)$ 
(cf. \cite{GRR11}).
Objects of $\sS (\sK)$ are subgroups of $\pi_1 (\sK)$. 
The morphisms $\sS(P,Q)$ are conjugations $\gamma_\vx: P \rightarrow Q$, 
$\gamma_\vx (\va) = \vx\va\vx^{-1}$, $\vx\in \pi_1 (\sK)$ whenever
$\vx P\vx^{-1}\subseteq Q$, restricted to $P$.
If $\vy^{-1}\vx\neq 1_G$ is in the centraliser of $P$,
then 
$\gamma_\vx$ and $\gamma_\vy$ are the same conjugations,
but we treat them as different morphisms in $\sS(P,Q)$.
In this respect we are different from the setup,
studied by Gunnells, Rose and Rumynin \cite{GRR11},
where these are the same morphism.
The setups are not drastically different, so we can still use their
results
exercising a certain care. 

A subgroup $P\leq \pi_1 (\sK)$ has an inverse image 
$\bar{P}\coloneqq
(G\rightarrow \pi_1 (\sK))^{-1} (P)$. It gives a restricted crossed
module
 $\sK_P \coloneqq (H\xrightarrow{\partial}\bar{P})$
with  $\pi_1 (\sK_P)\cong P$. 

Let
$\Phi$ be the functor ``2-representations of degree one''.
It is a contravariant functor
from $\sS(\sK)$ to the category of abelian groups.
On objects, $\Phi (P)\coloneqq \tRep_1 (\widetilde{\sK_P})$.
Let us look at a morphism  $\gamma_\vx: P \rightarrow Q$.
By picking a lifting $\dot{\vx}\in G$,
i.e., an element with $\dot{\vx} \partial (H)=\vx$,
we get a conjugation morphism of crossed modules
$$
\gamma_{\dot{\vx}}: \sK_P \rightarrow \sK_Q, \ \ 
P\ni\vg\mapsto \dot{\vx}\vg\dot{\vx}^{-1}, \ \ 
H\ni h\mapsto \,^{\dot{\vx}}h
$$
and the corresponding homomorphism of 2-groups
$$
\gamma_{\dot{\vx}}: \widetilde{\sK_P} \rightarrow \widetilde{\sK_Q}, \ \ 
P\ni\vg\mapsto \dot{\vx}\vg\dot{\vx}^{-1}, \ \ 
\xymatrix{
\star
   \ar@/^2pc/[rr]_{\quad}^{\vg_2}="2"
   \ar@/_2pc/[rr]_{\vg_1}="1"
&& \star
%   \ar@/^2pc/[ll]_{\quad}^{\vg_1}="1"
%  \ar@/_2pc/[ll]_{\vg_2}="2"
\ar@{}"1";"2"|(.2){\,}="7"
\ar@{}"1";"2"|(.8){\,}="8"
\ar@{=>}"7" ;"8"^{h}
} 
\mapsto
\xymatrix{
\star
   \ar@/^2pc/[rr]_{\quad}^{\dot{\vx}\vg_2\dot{\vx}^{-1}}="2"
   \ar@/_2pc/[rr]_{\dot{\vx}\vg_1\dot{\vx}^{-1}}="1"
&& \star
%   \ar@/^2pc/[ll]_{\quad}^{\vg_1}="1"
%  \ar@/_2pc/[ll]_{\vg_2}="2"
\ar@{}"1";"2"|(.2){\,}="7"
\ar@{}"1";"2"|(.8){\,}="8"
\ar@{=>}"7" ;"8"^{\,^{\dot{\vx}}h}
} 
.$$

An element $h\in H$ gives another lifting
$\partial(h)\dot{\vx}\in G$ of $\vx$
and a new 2-morphism
$\gamma_{\partial(h)\dot{\vx}}: \widetilde{\sK_P} \rightarrow \widetilde{\sK_Q}$.
In essence, they differ by an inner 2-isomorphism determined by $h$.
Let us spell it out.
\begin{lemma}
  \label{pull_back}
Let $\Theta$ be a 2-representation 
of $\widetilde{\sK_Q}$.
Then the 2-representations 
$\Theta\circ\gamma_{\dot{\vx}}$
and
$\Theta\circ\gamma_{\partial(h)\dot{\vx}}$
of $\widetilde{\sK_P}$ are equivalent.
\end{lemma}
\begin{proof}
By Corollary~\ref{morita_dec_1}  
the 2-representation $\Theta$
is equivalent to $\Theta_A^\circ$.
The latter comes from
a weak homomorphism of crossed modules $\theta: \sK_Q \rightarrow\tAut
(A)$. Consequently,
%Now we compute that
$$
\theta\circ\gamma_{\partial(h)\dot{\vx}} (\vg) =
\theta (\partial(h)\dot{\vx}\vg\dot{\vx}^{-1}\partial(h)^{-1}) =
\theta (\partial(h))
[\theta\circ\gamma_{\dot{\vx}} (\vg)]
\theta (\partial(h))^{-1}
=
\theta (h)
[\theta\circ\gamma_{\dot{\vx}} (\vg)]
\theta (h)^{-1}
$$
for $\vg\in P$ and
$$
\theta\circ\gamma_{\partial(h)\dot{\vx}} (g) =
\theta (\,^{\partial(h)\dot{\vx}}g) =
\theta (h\,^{\dot{\vx}}gh^{-1}) =
%\theta (\partial(h))
%[\theta\circ\gamma_{\dot{\vx}} (\vg)]
%\theta (\partial(h))^{-1}
%=
\theta (h)
[\theta\circ\gamma_{\dot{\vx}} (g)]
\theta (h)^{-1}
$$
for $g\in H$.
\end{proof}
Lemma~\ref{pull_back}
applies to 2-representations, hence,
$\gamma_{\dot{\vx}}$
and
$\gamma_{\partial(h)\dot{\vx}}$
determine
the same pull-back homomorphism
that allows us to define the functor $\Phi$ on morphisms: 
$$
\Phi (\gamma_{\vx}) \coloneqq
[\gamma_{\dot{\vx}}]
=
[\gamma_{\partial(h)\dot{\vx}}]:
\Phi (Q)=\tRep_1 (\sK_Q)
\rightarrow
\tRep_1 (\sK_P) = \Phi (P)
.
$$
In general, if the conjugations $\gamma_{\vx}$ and 
$\gamma_{\vy}$ are the same, 
these pull-backs can be different:
$\Theta\circ\gamma_{\dot{\vx}}$
and
$\Theta\circ\gamma_{\dot{\vy}}$
are not necessarily equivalent because the actions of 2-objects
of $\widetilde{\sK}_Q$ could be different. This necessitates our
version
of the category $\sS (\sK)$ (cf. \cite{GRR11}).

Despite this slight difference, the functor $\Phi$ still
leads to
the generalised Burnside ring 
$\BA(\sK) \coloneqq \BP_{\bA} (\pi_1 (\sK))$
with coefficients in a commutative ring
$\bA$~\cite{GRR11}.
The $\bA$-basis of $\BA(\sK)$
consists of pairs
$
\langle \Theta, P \rangle
$
where $P$ is a subgroup of $\pi_1 (\sK)$,
$\Theta$ is a degree one 2-representation
of $\widetilde{\sK_P}$.
In each $\pi_1 (\sK)$-conjugacy class of such pairs
we choose one representative because
$$
\langle \Theta,P\rangle  = \langle \Phi(\gamma_{\vx})(\Theta), \vx^{-1}P\vx \rangle 
$$
for all $\vx\in \pi_1 (\sK)$.
We can also write a pair $\langle \Theta,P\rangle$
with an arbitrary 2-representation $\Theta$ of $\widetilde{\sK_P}$ 
but they can be rewritten as linear combinations
of pairs with degree one  2-representations by the formulas
$$
\langle \Theta_1 \boxplus\Theta_2 ,P\rangle = 
\langle \Theta_1  ,P\rangle  
+
\langle \Theta_2 ,P\rangle  
\mbox{ and }
\langle {\Theta}\uparrow_{\widetilde{\sK}_Q}^{\widetilde{\sK}_P},P\rangle = 
\langle \Theta  , Q\rangle  
.
$$ 
The multiplication in $\BA(\sK)$
is $\bA$-bilinear, defined on the basis by the formula
$$
\langle \Theta,P\rangle  \cdot \langle \Omega, Q\rangle  = 
\sum_{P\vx Q\in P\backslash \pi_1 (\sK)/Q} 
\langle 
\Phi(\gamma_1: P\cap \vx Q \vx^{-1} \rightarrow P)(\Theta) \boxtimes 
\Phi(\gamma_{\vx^{-1}} : P\cap \vx Q\vx^{-1} \rightarrow Q)(b)
,
P\cap \vx Q \vx^{-1}\rangle .
$$
Theorem~\ref{irr_rep}
together with known arguments~\cite[Section 1]{GRR11}
emanates the following proposition.
\begin{prop}
\label{Groth}
Let $K_\bA(\sK)$ 
be the Grothendieck ring of 2-representations
of the 2-group $\widetilde{\sK}$ 
with coefficients in a commutative ring $\bA$, 
where the multiplication comes from the
tensor product $\boxtimes$.
The assignment \newline
$[ {\Theta}\uparrow_{\widetilde{\sK}_P}^{\widetilde{\sK}}] \mapsto \langle \Theta  , P\rangle$
extended by $\bA$-linearity
is an $\bA$-algebra isomorphism
$K_\bA(\sK) \rightarrow \BA(\sK)$. 
\end{prop}

Let us introduce 
{\em a mark homomorphism}
\cite[Lemma 1.2]{GRR11}. 
We need to assume that the order $|\pi_1(\sK)|$
is invertible in $\bA$.
Let $\alpha : \Phi (P) \rightarrow \bA^\times$
be a group homomorphism. The corresponding mark is 
an $\bA$-algebra homomorphism
$f_P^\alpha : \BA(\sK) \rightarrow \bA$ given by the formula
$$
f_P^\alpha (\langle \Theta, Q \rangle)
=
\frac{1}{|Q|}\sum_{\vg\in X}
%\alpha(\Res_{K^g,A}(\mu^g)),
\alpha (
\Phi(\gamma_{\vg} : P \rightarrow Q)(\Theta) 
)$$
where $X=\{\vg\in \pi_1 (\sK) \mid \vg P\vg^{-1}\subseteq Q \}$. 
The marks work magnificently for the ring $\BA (\sK)$ if
all the groups $\Phi  (P)$ are finite
\cite[Corollary 1.3]{GRR11}:
\begin{prop}
\label{BPsK}
Suppose all $\Phi  (P)$ are finite.
Let $N$ be the least common multiple of all the orders of
elements in various  
$\Phi  (P)$. 
If $\bA$ 
is a field, 
containing a primitive $N$-th root of unity, 
then the mark homomorphisms
define an isomorphism of $\bA$-algebras 
$$
\BA (\sK) 
\xrightarrow{\cong}
\oplus \bA = \bA^k
$$
where $k$ is the number of $G$-orbits
on the disjoint unions
$\dot{\cup}_{P} \Phi(P)$.
\end{prop}

If $\varphi:\sK^\prime \rightarrow \sK$ is a homomorphism of crossed modules,
the pull-back of 2-representations gives
a homomorphism of Burnside rings
$\varphi^\ast :\BA (\sK)\rightarrow\BA (\sK^\prime)$.
Consider the quotient homomorphism of crossed modules
$$
\varphi: \sK \rightarrow \bar{\sK}\coloneqq(1\rightarrow \pi_1 (\sK)), \ \ 
G\ni\vg\mapsto \vg\partial(H), \ \ 
H\ni h\mapsto 1.
$$
The Burnside ring
$\BA (\bar{\sK})$
is precisely the generalised Burnside ring of $\pi_1(\sK)$
studied by Gunnells, Rose and Rumynin \cite{GRR11},
because there are no non-trivial 2-objects in 
$\widetilde{\bar{\sK}}$.
All the groups $\Phi (P) = H^2 (P,\bK^\times )$ are finite.
Proposition~\ref{BPsK}
tells us that 
if $\bA$ 
is a field, 
containing a primitive $N$-th root of unity, 
then 
the corresponding pull-back 
algebra homomorphism
$\varphi^\ast : \BA (\bar{\sK}) = \bA^k \rightarrow \BA (\sK)$
is injective. Its image 
can be thought of as 
the Grothendieck ring of 2-representations
``trivial'' on $H$. 

\begin{center}
\section{Ganter-Kapranov 2-character}
\label{s4}
\end{center}

Let us recall the notion of a 2-categorical trace \cite{GK08}.
Let $\sC$ be a bicategory, $x\in\sC_0$ its 0-object.
{\em The 2-categorical trace} of a 1-morphism $u\in\sC_1(x,x)$
is the set $\Tr_x(u) \coloneqq \sC_2(\vi_x,u)$. It is instructive to
observe
that in the bicategory of 2-vector spaces $\tVec^\bK$
a 1-morphism $u=(U_{i,j})$ is 
an $n\times n$-matrix 
of vector spaces, while
its trace is the vector space
$$
\Tr_n(u) = \bigoplus_i \hom_\bK (\bK, U_{i,i}) \oplus
\bigoplus_{i\neq j} \hom_\bK (0, U_{i,j})
\cong
\bigoplus_i U_{i,i}
\; . $$
%where the last isomorphism is natural.
Let $\Theta$ be a 2-representation of degree $n$ of
the 2-group $\widetilde{\sK}$.
%where $\sK$ is a crossed module.
%To define the 2-character of $R$ 
%we consider 
Pick two elements $\va,\vb \in G$ such that
their images in the fundamental group commute:
$
\bar{\va} \bar{\vb} = \bar{\vb} \bar{\va} \in
\pi_1(\sK)
$ 
and $h \in H$ such that 
$\partial(h) \va\vb\va^{-1} = \vb$. This data gives a linear operator
$
\bX_\Theta(\vb,\va,h):
\Tr_n (\Theta^1(\vb)) \rightarrow
\Tr_n (\Theta^1(\vb))$
$$
\bX_\Theta(\vb,\va,h)
\Big (
\xymatrix{
n%\star
\ar@/^1pc/[rr]^{\Theta(\vb)}="1"
\ar@/_1pc/[rr]_{\vi_{n}}="2"
\ar@{}"1";"2"|(.2){\,}="7"
\ar@{}"1";"2"|(.8){\,}="8"
&& n%\star
\ar@{<=}"7" ;"8"^{v}
} 
\Big )
\; = \; 
\xymatrix{
n%\star
\ar@/^1pc/[rr]^{\Theta(\va)}="1"
\ar@/_1pc/[rr]_{\Theta(\va)}="2"
\ar@{}"1";"2"|(.2){\,}="7"
\ar@{}"1";"2"|(.8){\,}="8"
\ar@{<=}"7" ;"8"^{Id}
\ar@/^5pc/[rrrrrr]^{\Theta(\vb)}="1"
\ar@/_5pc/[rrrrrr]_{\vi_{n}}="2"
&& n%\star
\ar@/^1pc/[rr]^{\Theta(\vb)}="3"
\ar@/_1pc/[rr]_{\vi_{n}}="4"
\ar@{}"3";"4"|(.2){\,}="7"
\ar@{}"3";"4"|(.8){\,}="8"
\ar@{<=}"7" ;"8"^{v}
\ar@{}"1";"3"|(.2){\,}="7"
\ar@{}"1";"3"|(.8){\,}="8"
\ar@{<=}"7" ;"8"^{[\va,\vb,\va^{-1},h]}
\ar@{}"4";"2"|(.2){\,}="7"
\ar@{}"4";"2"|(.8){\,}="8"
%\ar@{<=}"7" ;"8"^{(\Theta^2_{1,1})^{-1}}
\ar@{<=}"7" ;"8"^{[\va,1,\va^{-1},1]^{-1}}
&& n%\star
\ar@/^1pc/[rr]^{\Theta(\va^{-1})}="1"
\ar@/_1pc/[rr]_{\Theta(\va^{-1})}="2"
\ar@{}"1";"2"|(.2){\,}="7"
\ar@{}"1";"2"|(.8){\,}="8"
\ar@{<=}"7" ;"8"^{Id}
&& n%\star
} 
$$
where $\Theta (\vb)= \Theta^1_{\star,\star}(\vb)$
and
$[\va,\vb,\va^{-1},h]$
is a composition of the natural morphism
$\Theta (\va)\diamond\Theta (\vb)\diamond\Theta (\va^{-1})
\rightarrow
\Theta (\va\vb\va^{-1})$
and the action $h\cdot : \Theta (\va\vb\va^{-1})
\rightarrow \Theta (\vb)$.
Let us write a matrix of vector spaces  
$\Theta^1_{\star,\star}(\va) = (U_{i,j} (\va))$.
Its dimension is the permutation matrix of some permutation
$\sigma$, e.g., $U_{i,j} (\va) \neq 0$ if and only if $j=\sigma (i)$.
Now the natural map 
$\vi_n \rightarrow \Theta (\va)\diamond\Theta (\va^{-1})$
is given by a collection of elements $x_i(\va)\in U_{i,\sigma (i)} (\va)$,
$y_i(\va)\in U_{\sigma(i),i} (\va^{-1})$ in a way that
$$
\bK = \vi_{n\; i,i} \ni 1_i \mapsto x_i(\va)\otimes y_i(\va) \in 
U_{i,\sigma (i)} (\va) \otimes U_{\sigma(i),i} (\va^{-1}).
$$
Now we can write the key map in an elementary way:
$$
\bX_\Theta(\vb,\va,h) (\sum_i b_i) =
h \cdot (\sum_i x_i(\va) \otimes b_{\sigma (i)}\otimes y_i(\va) )
\ \mbox{ where } \ 
\sum_i b_i \in 
\Tr_n (\Theta^1_{\star.\star}(\vb)) = \oplus_i U_{i,i} (\vb).
$$
Its trace is  
the 2-character value %$\mX_\Theta(\va,\vb,h)$ 
%is the trace
%of this linear map:
$$
\mX_\Theta (\vb,\va,h)\coloneqq
\mathrm{Tr} (\bX_\Theta(\vb,\va,h)).
$$
Let $\bG$ be the set of all triples
$(\va,\vb,h)\in G\times G\times H$
such that
$\partial(h) \va\vb = \vb\va$.
The group $G$ acts on the set $\bG$ by conjugation.
\begin{prop}
\label{tensor}
For any 2-representation $\Theta$
the function $\mX_\Theta:\bG\rightarrow\bK$
is constant on $G$-orbits.
If $\Psi$ is another 2-representation,
then
$$
\mX_{\Psi\boxtimes\Theta}(\vb,\va,h)\coloneqq
\mX_\Psi (\vb,\va,h) \cdot \mX_\Theta(\vb,\va,h).
$$
\end{prop}
\begin{proof}
Let us prove the first statement, i.e., that
$
\mX_{\Theta}(\,^\vg\vb,\,^\vg\va,\,^\vg h)=
\mX_{\Theta}(\vb,\va,h)
$
for all $\vg\in G$. We have a natural ``conjugation''
linear map
$\Gamma_\vg:
\mathbb{T}r_n (\Theta^1(\vb)) \rightarrow
\mathbb{T}r_n (\Theta^1(\,^\vg\vb))$
given by the formula
$$
\Gamma_\vg
\Big (
\xymatrix{
n%\star
\ar@/^1pc/[rr]^{\Theta^1(\vb)}="1"
\ar@/_1pc/[rr]_{\vi_{n}}="2"
\ar@{}"1";"2"|(.2){\,}="7"
\ar@{}"1";"2"|(.8){\,}="8"
&& n%\star
\ar@{<=}"7" ;"8"^{v}
} 
\Big )
\; = \; 
\xymatrix{
n%\star
\ar@/^1pc/[rr]^{\Theta^1(\vg)}="1"
\ar@/_1pc/[rr]_{\Theta^1(\vg)}="2"
\ar@{}"1";"2"|(.2){\,}="7"
\ar@{}"1";"2"|(.8){\,}="8"
\ar@{<=}"7" ;"8"^{Id}
\ar@/^5pc/[rrrrrr]^{\Theta^1(\vg\vb\vg^{-1})}="1"
\ar@/_5pc/[rrrrrr]_{\vi_{n}}="2"
&& n%\star
\ar@/^1pc/[rr]^{\Theta^1(\vb)}="3"
\ar@/_1pc/[rr]_{\vi_{n}}="4"
\ar@{}"3";"4"|(.2){\,}="7"
\ar@{}"3";"4"|(.8){\,}="8"
\ar@{<=}"7" ;"8"^{v}
\ar@{}"1";"3"|(.2){\,}="7"
\ar@{}"1";"3"|(.8){\,}="8"
\ar@{<=}"7" ;"8"^{\Theta^2_{\vg\vb\vg^{-1},1}}
\ar@{}"4";"2"|(.2){\,}="7"
\ar@{}"4";"2"|(.8){\,}="8"
\ar@{<=}"7" ;"8"^{(\Theta^2_{1,1})\,^{-1}}
&& n%\star
\ar@/^1pc/[rr]^{\Theta^1(\vg^{-1})}="1"
\ar@/_1pc/[rr]_{\Theta^1(\vg^{-1})}="2"
\ar@{}"1";"2"|(.2){\,}="7"
\ar@{}"1";"2"|(.8){\,}="8"
\ar@{<=}"7" ;"8"^{Id}
&& n%\star
} 
$$
It is straightforward to observe that
$$
\bX_{\Theta}(\,^\vg\vb,\,^\vg\va,\,^\vg h)=
\Gamma_\vg \bX_{\Theta}(\vb,\va,h) \Gamma_\vg^{-1} \, ,
$$
hence, the linear maps
$
\bX_{\Theta}(\,^\vg\vb,\,^\vg\va,\,^\vg h)
$
and
$
\bX_{\Theta}(\vb,\va,h)
$
have the same trace. Let us prove the second statement now. Writing
matrices of vector spaces 
as 
$\Theta^1_{\star,\star}(\vb) = (U_{i,j} (\vb))$
and
$\Psi^1_{\star,\star}(\vb) = (V_{i,j} (\vb))$,
we observe that
$\Tr_{nm} (\Theta\boxtimes\Psi)^1_{\star,\star}(\vb) \cong
\sum_{i,t} U_{i,i} (\vb)\otimes V_{t,t} (\vb) \cong
(\sum_{i} U_{i,i} (\vb))\otimes (\sum_t V_{t,t} (\vb)) \cong
\Tr_{n} \Theta^1_{\star,\star}(\vb)
\otimes  
\Tr_{m}\Psi^1_{\star,\star}(\vb)$.
Identifying the 2-traces under these maps, we can observe that
$\bX_{\Theta\boxtimes\Psi}(\vb,\va,h)
=
\bX_\Theta(\vb,\va,h)
\otimes 
\bX_\Psi(\vb,\va,h)$
that implies the second statement.
\end{proof}

A character table of a finite group has rows and columns. Usually
one thinks of columns as characters, yet it is often
instructive to think of rows as characters. Applying this way
of thinking  to
the 2-characters we can use Proposition~\ref{tensor}
to conclude that 
a $G$-conjugacy class of triples $(\va,\vb,h)$
with $\partial (h) \va\vb = \vb \va$
determines 
a ring homomorphism 
$$
\mX (\vb,\va,h):
\bB_\bZ (\sK)\rightarrow \bK, \ \ 
[\Theta] \mapsto 
\mX_\Theta (\vb,\va,h).
$$
It can be extended by $\bK$-linearity to a $\bK$-algebra homomorphism
$
\mX (\vb,\va,h):
\bB_\bK (\sK)\rightarrow \bK
$. 
Both versions of 
$\mX$
should be called {\em a Ganter-Kapranov 2-character}.
In the finite case (i.e., under assumptions of 
Proposition~\ref{BPsK}) the Ganter-Kapranov 2-character must be one
of the marks. Which one?
\begin{theorem}
\label{Our_Char}
In the notations above, let $P$ be the subgroup of $\pi_1(\sK)$
generated by $\bar{\va}$ and $\bar{\vb}$. \newline 
Let~$\alpha\coloneqq \mX (\vb,\va,h)$ considered as a group
homomorphism 
$\tRep^1 (\sK_P)\rightarrow \bK^\times$. If the order of $\pi_1 (\sK)$
is finite and invertible in the field $\bK$, then
$$
\mX (\vb,\va,h)
= f_P^\alpha .
$$
\end{theorem}
\begin{proof}
Let us compute both parts on an induced 2-representation.
Let $Q$ be a subgroup of $\pi_1 (\sK)$, 
$\widehat{Q}$ its inverse
image in $G$,
$\Theta=\Theta_A^\circ\in\tRep^1 (\sK_Q)$ that comes from an action of $\sK_Q$
on a full matrix algebra $A=\End_\bK (M)$.
Let us choose a left transversal $T$ to 
$\widehat{Q}$ in $G$
so that $\widetilde{\Theta}:={\Theta_A^\circ}\uparrow_{\widetilde{\sK_Q}}^{\widetilde{\sK}}
= \Theta_{\widetilde{A}}^\circ$.
where
$
\widetilde{A} 
= \oplus_{\vt\in T} (\vt, A)$.

Let $n = |T|$. The contribution to 
$\Tr_{n} (\widetilde{\Theta}(\vb))$ comes from those
$\vt$
that $\vb\vt \in \vt \widehat{Q}$ (equivalently 
$\vt^{-1}\vb\vt \in \widehat{Q}$):
$$
\Tr_{n} (\widetilde{\Theta}(\vb)) =
\bigoplus_{\vt\in T} \hom_{\widetilde{A}}
((\vt,M),(\vb\vt,M))
\cong
\bigoplus_{\vt\in T, \; \vt^{-1}\vb\vt \in \widehat{Q}}
\bK
$$
where $(\vx,M)$ is a simple module for $(\vx,A)$.
Observe that $(\vb\vt,M)=(\vt,M)^{[\vb]}$
and the condition $\vt^{-1}\vb\vt \in \widehat{Q}$
is equivalent to $\widetilde{A}$-modules
$(\vb\vt,M)$
and
$(\vt,M)$
being isomorphic. The linear map $\bX_\Theta(\vb,\va,h)$ goes via the
matrix of vector spaces
$\widetilde{\Theta}(\va)
\diamond
\widetilde{\Theta}(\vb)
\diamond
\widetilde{\Theta}(\va^{-1})
$
whose entries on the main diagonal are of the form
$$
U_{i,j} (\va) \otimes U_{j,j} (\vb) \otimes U_{j,i} (\va^{-1}) =
\hom_{\widetilde{A}}
((\vb\vt,M),(\va\vb\vt,M))
\otimes
\hom_{\widetilde{A}}
((\vt,M),(\vb\vt,M))
\otimes
\hom_{\widetilde{A}}
((\va^{-1}\vt,M),(\vt,M)).
$$
For this entry to be nonzero we need
the three elements
$(\va^{-1}\vt)^{-1}\vt = \vt^{-1}\va\vt$,
$\vt^{-1}\vb\vt$ 
and
$
(\vb\vt)^{-1}
\va\vb\vt
=
\vt^{-1}\vb^{-1}\va\vb\vt$ 
to be in $\widehat{Q}$. This is equivalent to 
$\vt^{-1}\widehat{P}\vt \subseteq \widehat{Q}$.

It remains to notice that the contribution from each such $\vt$ is
value of $\mX (\vb,\va,h)$ on the pull-back (under $\gamma_\vt$) of
the degree one  2-representation $\Theta$. The theorem follows immediately
by examining the formula for $f_P^\alpha$.
\end{proof}

\begin{center}
\section{Shapiro isomorphism}
\label{s5}
\end{center}

 Let $G$ be a group, $Q \leq G$ its subgroup, $M$ a $\bZ Q$-module.
Shapiro's lemma \cite{NM58} asserts isomorphisms in homology and
cohomology:
$$
H^{\ast}(G, \Coi^G_Q(M)) \cong H^{\ast}(Q,M), \ \ \ 
H_{\ast}(G, \Ind^G_Q(M)) \cong H_{\ast}(Q,M).
$$
The standard proof goes via a quasiisomorphism of the corresponding
complexes.
It does not give an explicit formula that we require for cohomology.
Hence, 
we supply an explicit chain homotopy 
$$
\psi : C^n(Q,M) \rightarrow C^n(G, \Coi_Q^G(M)).
$$
Choose a right transversal $T=\{\vt_1,\vt_2 \ldots\}_j$ to $Q$ in $G$ 
such that $\vt_1 = 1_G$.
The coinduced module  
$\Coi_Q^G(M)$ is the set of all $Q$-equivariant functions 
$f:G\rightarrow M$.
Such a function is uniquely determined by its values on $T$. 
The right transversal allows us to identify the coinduced module  
$\Coi_Q^G(M)$ with the set of all functions $f:T\rightarrow M$.
The cochains $C^n(Q,M)$ are also functions $\mu:Q^n\rightarrow M$.
Given elements 
$\vg_1, \ldots, \vg_n \in G$ and 
$\vt \in T$, there exist elements 
$\vh_1, \ldots, \vh_n \in Q$ and 
$\vs_0 = \vt, \vs_1, \ldots, \vs_n \in T$ uniquely determined
by the following equations:
$$
\vs_0 \vg_1 \cdots \vg_n = 
\vh_1 \vs_1 \vg_2 \cdots \vg_n = 
\ldots = 
\vh_1 \cdots \vh_{k} \vs_k \vg_{k+1} \cdots \vg_n = 
\ldots = 
\vh_1 \cdots \vh_{n-1} \vs_{n-1} \vg_n = 
\vh_1 \cdots \vh_n \vs_n.  
$$
We use these elements to define $\psi$ 
on a cochain $\mu\in C^n(Q,M)$:
$$\psi(\mu)(\vg_1, \ldots, \vg_n)(\vt) \coloneqq \mu(\vh_1, \ldots, \vh_n).
$$ 
In the opposite direction we define a map for arbitrary elements
$\vh_1, \ldots, \vh_n\in Q$:
$$
\phi : C^n(G, \Coi_Q^G(M)) \rightarrow C^n(Q,M), \ \ \ 
\phi(\theta)(\vh_1, \ldots, \vh_n) = \theta(\vh_1, \ldots, \vh_n)(1_G)
.$$
%by $\phi(\mu)(h_1, \ldots, h_n) = \mu(h_1, \ldots, h_n)(1)$. The
%following theorem will deduce the above corollary, the proof of which
%is highly computational.
We are ready for the main result of this section.

\begin{theorem} \label{prop1}
Let $Q \leq G$ be groups, $M$ a $\bZ Q$-module. The above defined maps
$\phi$ and $\psi$ are isomorphisms of the cochain complexes
$C^{\ast}(Q, M)$ and $C^{\ast}(G, \Coi_Q^G(M))$ in the homotopic category.
\end{theorem}
\begin{proof}
Observe that for $\vg_1, \ldots, \vg_n \in Q$ and $\vt = 1$ 
we get $\vh_j =\vg_j$. Hence, 
$$\phi(\psi(\mu))(\vg_1, \ldots, \vg_n)=
\psi(\mu)(\vg_1, \ldots, \vg_n)(1_G)=
\mu(\vg_1, \ldots, \vg_n)
$$
proving that 
$\phi \circ \psi$ is equal to the identity. 
In the opposite direction, 
$\psi \circ \phi$ is only homotopic to the identity:
\begin{displaymath}
\scalebox{0.95}[1]{
\xymatrix{
\cdots \ar[r] & C^{n-1}(P, \Coi_Q^G(M)) \ar[rr]^{d^{n-1}} \ar[dl] \ar@/^/[d]^{\phi \circ \psi} \ar@/_/[d]^{\mathds{1}} & &  C^{n}(P, \Coi_Q^G(M)) \ar[rr]^{d^{n}} \ar[dll]_{\varpi^{n-1}} \ar@/^/[d]^{\phi \circ \psi} \ar@/_/[d]^{\mathds{1}} & &  C^{n+1}(P, \Coi_Q^G(M)) \ar[r] \ar[dll]_{\varpi^{n}} \ar@/^/[d]^{\phi \circ \psi} \ar@/_/[d]^{\mathds{1}} & \ldots \ar[dl]\\
\cdots \ar[r] & C^{n-1}(P, \Coi_Q^G(M)) \ar[rr]^{d^{n-1}} & & C^{n}(P, \Coi_Q^G(M)) \ar[rr]^{d^{n}} & &  C^{n+1}(P, \Coi_Q^G(M)) \ar[r] & \cdots \\
}
}
\end{displaymath}
where the homotopy $\varpi$ is define by 
$$
\varpi^n(\theta)(\vg_1,\ldots,\vg_n)(\vt) = 
\sum_{j=0}^{n} (-1)^{j+1} \theta(\vh_1, \ldots, \vh_j, \vs_j, \vg_{j+1}, \ldots, \vg_n)(1).
$$
Let us verify that
$\psi \circ \phi - \mathds{1} = h^n \varpi^n + d^{n-1} \varpi^{n-1}$. 
Let us first examine the left hand side of this equality:
$$
(\psi \circ \phi - \mathds{1})(\theta)(\vg_1, \ldots, \vg_n)(\vt) = 
\psi \circ \phi(\theta)(\vg_1, \ldots, \vg_n)(\vt) - \theta(\vg_1, \ldots, \vg_n)(\vt)
= 
\theta(\vh_1, \ldots, \vh_n)(1) - \theta(\vg_1, \ldots, \vg_n)(\vt).
$$
%To scrutinise the right hand side it is convenient to 
%extend  $f\in \Coi_Q^G (M)$ 
%(we think of the coinduced module as functions $T\rightarrow M$)
%to a
%we need to remind a reader what it is.
%The function $f$ extends to a $Q$-equivariant function $\widetilde{f}$,
%that writing $\vg=\vh \vt$ for unique 
%$\vh\in P$, $\vt\in T$ 
Now we scrutinise
the first term of the right hand side.
It is useful to pay attention which $\vs_j$ appears in terms of 
the final expression
because
it tells you from which term of the second expression it originates.
We label the lines to help observe the cancellations:
\begin{align}
\varpi^n( d^n(\theta))(\vg_1, \ldots, \vg_n)(\vt) 
&
= 
\sum_{j=0}^n (-1)^{j+1} d^n(\theta)(\vh_1, \ldots, \vh_j, \vs_j,
\vg_{j+1}, \ldots, \vg_n)(1)
= \notag
\\ %1
&
- \theta(\vg_1, \ldots, \vg_n)(\vs_0)
\\ %2
& + \sum_{j=1}^{n} (-1)^{j+1} \theta(\vh_2,
  \ldots, \vh_{j}, \vs_j, \vg_{j+1}, \ldots, \vg_n)(\vh_1) \ +
\\ %3
& + \sum_{j=2}^n \sum_{k=1}^{j-1} (-1)^{j+k+1} \theta(\ldots \vh_{k-1},
\vh_k\vh_{k+1}, \vh_{k+2}, \ldots, \vh_j, \vs_j, \vg_j \ldots)(1)
\\ %4
& - \sum_{j=1}^n \theta(\vh_1, \ldots, \vh_{j-1}, \vh_j \vs_j, \vg_{j+1}, \ldots, \vg_n)(1)
\\ %5
& + \sum_{j=0}^{n-1} \theta(\vh_1, \ldots, \vh_{j-1}, \vs_j\vg_{j+1},\vg_{j+1},\ldots, \vg_n)(1)
\\ %6
& + \sum_{j=0}^{n-2} \sum_{k = j+1}^{n-1} (-1)^{j+k} 
\theta(\ldots \vh_j, \vs_j, \vg_{j+1}, \ldots, \vg_{k-1}, \vg_k\vg_{k+1}, \vg_{k+2}
\ldots ) (1)
%(-1)^k \theta(\vs_0,\vg_1, \ldots, \vg_{k-1}, \vg_k \vg_{k+1},
%\vg_{k+2}, \ldots, \vg_n)(1)
\\ %7
& + \sum_{j=0}^{n-1} (-1)^{j+n} \theta(\vh_1, \ldots, \vh_j, \vs_j,\vg_{j+1},
\ldots , \vg_{n-1})(1)
\\ & %8
+ \theta(\vh_1,  \ldots, \vh_{n})(1).
\end{align}
Lines (1) and (8) contribute to the left hand side.
Lines (4) and (5) cancel because $\vs_j\vg_{j+1}=\vh_{j+1}\vs_{j+1}$. 
The remaining
lines cancel
with the second term (line labels
correspond to 
their
cancelling counterparts):
\begin{align}
d^{n-1}(\varpi^{n-1}(\theta))(\vg_1, \ldots, \vg_n)(\vt) & = 
\notag
\\
\varpi^{n-1}(\theta)(\vg_2,\ldots, \vg_n)(\vh_1 \vs_1) 
&
+ \sum_{k=1}^{n-1} (-1)^k \varpi^{n-1}(\theta)(\vg_1,
\ldots, \vg_k\vg_{k+1} \ldots)(\vt)
+ (-1)^{n} \varpi^{n-1}(\theta)(\vg_1, \ldots,
\vg_{n-1})(\vt)
\notag
\\
& = \sum_{j=1}^n (-1)^j \theta(\vh_2, \ldots, \vh_{j}, \vs_j, 
\vg_{j+1}, \ldots,\vg_n)(\vh_1)
\tag{2}
\\
& + \sum_{k=1}^{n-1} \sum_{j=0}^{k-1}
  (-1)^{k+j+1} \theta(\ldots \vh_{j}, \vs_j, \vg_{j+1}, \ldots , \vg_{k-1},
  \vg_k\vg_{k+1}, \vg_{k+2} \ldots )(1)
\tag{6}
\\
& + \sum_{k=1}^{n-1} \sum_{j=k+1}^{n-1} (-1)^{k+j} 
\theta(\ldots \vh_{k-1},
\vh_k \vh_{k+1}, \vh_{k+2}, \ldots, \vh_j, \vs_j, \vg_{j+1}, \ldots , \vg_n)(1) 
\tag{3}
\\
& + \sum_{j=0}^{n-1} (-1)^{n+j+1} \theta(\vh_1, \ldots, \vh_j, \vs_j, \vg_{j+1}, \ldots, \vg_{n-1})(1).
\tag{7}
\end{align}
\end{proof}

\begin{center}
\section{Osorno Formula}
\label{s6}
\end{center}

In this section we investigate the special case of trivial $H$. 
Thus,  $G = \pi_1(\sK)$ is a finite group $G = \pi_1(\sK)$.
We will write $G$ for $\sK$ where appropriate, e.g.,  
$\tRep(G)=\tRep(\sK)$ etc.
The degree one  2-representations of $G$ are in bijection with
elements of the Schur multiplier over $\bK$:
\begin{prop}
  (cf. \cite[5.3]{Elg}) 
\label{Osor_linClas}
The group of degree one  2-representations $(\tRep_1 (G),\boxtimes)$ 
(see Lemma~\ref{Lin2})
is isomorphic to $H^2(G, \bK^{\times})$
where the multiplicative group $\bK^{\times}$ is a trivial $\bZ G$-module.
\end{prop}
\begin{proof}
Let $\Theta$ be a degree one  2-representation of $G$.
Lemma~\ref{morita_dec} attaches to $\Theta$
a unique $G$-Morita equivalence
class
of split simple algebras, whose representative $A$ satisfies
$\Theta\cong\Theta_A$. The cohomology class 
$\{\Theta\}\in H^2(G, \bK^{\times})$
of the corresponding projective representation 
$G\rightarrow \Aut (A)\cong \PGL_n (\bK)$
defines a bijection  \cite[Theorem 4.3]{GRR11}
$$
\langle \Theta, G \rangle \mapsto \{\Theta\}, \ \ \  \tRep_1 (G)\rightarrow H^2(G, \bK^{\times}).
$$
It is a group isomorphism because the tensor product of algebras
correspond to the addition of cocycles.
\end{proof}
Since $H$ is trivial we drop $h$ from the notation for the
Ganter-Kapranov 2-character: 
$\mX (\vb,\va)\coloneqq \mX (\vb,\va,1)$.
Let us compute its value on a degree one 2-representation:
\begin{theorem}
\label{th2}
Let $\va,\vb\in G$ be commuting elements, $\Theta$ 
a degree one 
2-representation of $G$,
$\mu\in  Z^2 (G,\bK^\times )$ a cocycle
such that
$[\mu] = \{ \Theta\}$.
Then
$$
\mX(\vb,\va)(\langle \Theta, G \rangle) = \mu(\vb,\va^{-1})
\mu(\va^{-1},\vb)^{-1}.
$$
\end{theorem}
\begin{proof}
Let $A$ be a split simple $G$-algebra such that 
$\Theta\cong\Theta_A$, $M$ a simple $A$-module.
Let $\rho$ be the projective representation of $G$
on $M$ such that $\rho (1) = \mbox{Id}_M$. 
The cocycle $\nu$ defined by the identity 
$$
%\nu (\vg,\vh) := \rho (\vg)\rho (\vh) \rho(\vg\vh)^{-1}
\rho(\vg\vh) = \nu (\vg,\vh)\rho (\vg)\rho (\vh) 
\ \mbox{ for all } \ 
\vg,\vh\in G
$$
satisfies
$[\nu] = \{ \Theta\}$. Then $\mu=\nu\cdot d\pi$ for 
some cochain $\pi\in  C^1 (G,\bK^\times )$. One can use $\nu$ rather
than
$\mu$ on the right hand side:
\begin{align*}
\mu(\vb,\va^{-1})
\mu(\va^{-1},\vb)^{-1}
= &
(\nu(\vb,\va^{-1})\pi(\vb)\pi(\vb\va^{-1})^{-1}\pi(\va^{-1}))
(\mu(\va^{-1},\vb)\pi(\va^{-1})\pi(\va^{-1}\vb)^{-1}\pi(\vb))^{-1}
\\
= &
\nu(\vb,\va^{-1})
\nu(\va^{-1},\vb)^{-1}
.&
\end{align*}
The condition $\rho (1) = \mbox{Id}_M$ makes the cocycle $\nu$
normalised and brings additional identities:
$$
\nu (\vg,1)=\nu (1,\vg)=1, \ \ 
\nu (\vg,\vg^{-1})=\nu (\vg^{-1},\vg), \ \ 
%\rho(\vg^{-1})^{-1}= \nu (\vg,\vg^{-1})^{-1}\rho(\vg).
\rho(\vg^{-1})^{-1}= \nu (\vg,\vg^{-1})\rho(\vg).
$$
The linear map $\bX_\Theta(\vb,\va)$ operates on
the one-dimensional space $\hom_A(M,M^{[\vb]})$,
a subspace of $\hom_\bK(M,M)$ spanned by $\rho (\vb)$.
More precisely,
\begin{align*}
\bX_\Theta(\vb,\va) (\rho (\vb)) = &
\rho (\va^{-1})\rho (\vb)\rho (\va^{-1})^{-1}
=
%\nu (\va^{-1},\vb)\rho (\va^{-1}\vb)\nu (\va,\va^{-1})^{-1}\rho(\va)
\nu (\va^{-1},\vb)^{-1}\rho (\va^{-1}\vb)\nu (\va,\va^{-1})\rho(\va)\\
= &
%\nu (\va^{-1},\vb)\nu (\va^{-1}\vb,\va)\nu(\va,\va^{-1})^{-1}\rho(\vb)
\nu (\va^{-1},\vb)^{-1}\nu (\va^{-1}\vb,\va)^{-1}\nu(\va,\va^{-1})\rho(\vb)
.
\end{align*}
We can
finish the proof using the cocycle condition and the fact that 
$\va^{-1}$ and $\vb$ commute:
\begin{align*}
\nu (\va^{-1},\vb)^{-1}
\nu(\va,\va^{-1})
\nu (\vb\va^{-1},\va)^{-1}
= &
\nu (\va^{-1},\vb)^{-1}
\nu(\va,\va^{-1})
\nu (\va^{-1},\va)^{-1}
\nu(\vb,1)^{-1}
\nu(\vb,\va^{-1})\\
= &
\nu(\vb,\va^{-1})
\nu(\va^{-1},\vb)^{-1}.
\end{align*}
\end{proof}
Occasionally in the literature
the opposite cocycle is associated to a projective representation: one
can use
$\nu (\vg,\vh) \rho(\vg\vh) = \rho (\vg)\rho (\vh)$ instead. Then
the formula for
$
\mX(\vb,\va)(\langle \Theta, G \rangle)
$ 
in Theorem~\ref{th2}
changes to its reciprocal. Other choices leading to the reciprocal are
using right representations instead of left ones or using
$\va^{-1}\vb\va$
in the definition of Ganter-Kapranov 2-character. We are ready to
derive
a formula for an irreducible 2-representation:
\begin{cor}
\label{cor1}
Let $\Theta$ be a degree one 2-representation of a subgroup $P \leq G$,
$\mu \in Z^2(P, \bK^{\times})$
a cocycle such that 
$\{\Theta\} = [\mu]$. Let $T$ be a right 
transversal to $P$ in $G$. 
If $\,^\vt\va:=\vt\va\vt^{-1}$ then
$$
\mX(\vb,\va)(\langle \Theta, P\rangle)
 = 
\sum_{\vt\in T, \; \,^{\vt}\va, \,^{\vt}\vb \in P} 
\frac{\mu(\,^{\vt}\vb, (\,^{\vt}\va)^{-1})}{ \mu ((\,^\vt\va)^{-1},
  \,^\vt\vb)}
=
\sum_{\vt\in T, \; \,^{\vt}\va, \,^{\vt}\vb \in P} 
\frac{\mu(\,^{\vt}\vb, (\,^{\vt}\va)^{-1})\mu(\,^{\vt}\va, (\,^{\vt}\vb)(\,^{\vt}\va)^{-1})}{\mu(\,^{\vt}\va, (\,^{\vt}\va)^{-1}) \mu(1,1)}.
$$
\end{cor}
\begin{proof}
If $\vg$ and $\vh$ commute,
then the cocycle condition implies the following identity:
$$
\mu(\vg,\vh\vg^{-1}) 
= 
\mu(\vg,\vg^{-1}\vh) 
= 
\mu(\vg^{-1},\vh)^{-1} 
\mu(\vg,\vg^{-1}) 
\mu(1,\vh) = 
\mu(\vg^{-1},\vh)^{-1} 
\mu(\vg,\vg^{-1}) 
\mu(1,1).
$$
Using 
Theorem~\ref{th2},  
Theorem~\ref{Our_Char}
and the definition of the mark homomorphism we compute the character:
\begin{align*}
\mX(\vb,\va)(\Theta, P)
 & = \frac{1}{\vert P \vert} 
\sum_{\vg\in G, \; \,^\vg \vb, \,^\vg \va \in P} 
\mu(\,^\vg\vb,(\,^\vg\va)^{-1}) \mu ((\,^\vg\va)^{-1}, \,^\vg\vb)^{-1}
\\
 & = 
\sum_{\vt\in T, \; \,^\vt \vb, \,^\vt \va \in P} 
\mu(\,^\vt\vb,(\,^\vt\va)^{-1}) \mu ((\,^\vt\va)^{-1}, \,^\vt\vb)^{-1}
\\
& = 
\sum_{\vt\in T, \; \,^{\vt}\va, \,^{\vt}\vb \in P} 
\frac{\mu(\,^{\vt}\vb, (\,^{\vt}\va)^{-1})\mu(\,^{\vt}\va,
  (\,^{\vt}\vb)(\,^{\vt}\va)^{-1})}{\mu(\,^{\vt}\va,
  (\,^{\vt}\va)^{-1}) \mu(1,1)}.
\end{align*}
\end{proof}
Corollary~\ref{cor1}
allows us to compute the value of the Ganter-Kapranov 2-character
on any 2-representation in terms of its decorated $G$-set
\cite{GRR11},
i.e. a finite $G$-set $X$, decorated with a cocycle 
$\mu_x\in Z^2 (G_x, \bK^\times)$ at every point $x\in X$.
An alternative data describing a representation is
a cocycle on a permutation module \cite[Proposition 1]{AO10}.
To describe we need a notation
$(\bK^{\times})^{X}$ for the permutation $G$-module
of all the function $X\rightarrow \bK^\times$.
Such a function $f$
is given by a collection of its values
$(f(x))=(\alpha_x)_{x\in X}$, i.e., 
non-zero field elements $\alpha_x \in\bK^\times$.
The action is left: $\vg\cdot(\alpha_x)=(\alpha_{\vg\cdot x})$.
On the level functions it is given by  
$[\vg\cdot f](x)=f(\vg^{-1}\cdot x)$.
of denotes $(\mathbb{C}^{\times})^{\vert S \vert}$ as a $A$-module through the action of $A$ on $S$. 
\begin{prop}
\label{Osor_Clas}
(cf. \cite[Prop. 1]{AO10} and \cite[5.4]{Elg}) 
There is a one-to-one correspondence between equivalence classes 
of 2-representations of $G$ over $\bK$ 
and pairs ($X$, $[\theta]$) where $X$ is a
finite 
$G$-set and $[\theta] \in H^2(G, (\bK^{\times})^X)$.
\end{prop}
\begin{proof}
Theorem~\ref{irr_rep} associates
to a 2-representation $\Theta$
a unique (up to conjugacy and an isomorphism)
a collection $(P_i,\Phi_i)$ of pairs a subgroup $P_i$
and a degree one  2-representation $\Phi_i$ of $P_i$ so that 
$$
\Theta \cong \boxplus_i \Phi_i \uparrow_{P_i}^{G}.
$$
Proposition~\ref{Osor_linClas} gives cohomology classes
$\{\Phi_i\}\in H^2(P_i, \bK^{\times})$.
The permutation module $(\bK^{\times})^{G/P_i}$
is naturally isomorphic to the coinduced module
$\Coi_{P_i}^G(\bK^{\times})$, thus, we can use
Shapiro isomorphism 
(see Theorem~\ref{prop1})  
to get unique 
cohomology classes
$\psi(\{\Phi_i\}) \in H^2(G, (\bK^{\times})^{G/P_i})$.
We have associated the set 
and the cohomology class
$$
X\coloneqq \coprod_i G/P_i, \ \ 
[\theta] \coloneqq \bigoplus_i \psi(\{\Phi_i\}) \in
\bigoplus_i H^2 (G, (\bK^{\times})^{G/P_i})
\cong 
H^2 (G, (\bK^{\times})^X)
$$
to $\Theta$. All these steps are reversible
\end{proof}
Given a finite $G$-set $X$, $x\in X$ and a cochain
$\theta\in C^2 (G, (\bK^{\times})^X)$, 
we write
$\theta^x\in C^2 (G, \bK^{\times})$ 
for the component cochains.
We have
$\theta (\vg,\vh) (x) = \theta^x (\vg,\vh)$
on the level of functions $X\rightarrow \bK^{\times}$.
We are ready to give our proof of Osorno Formula:
\begin{theorem}
\cite[Theorem 1]{AO10} 
Let $\Theta$ be a 2-representation of $G$
that corresponds to 
a $G$-set $X$ and a cohomology class $[\theta]$
for some cochain
$\theta \in Z^2(G,(\bK^{\times})^X)$. Then
$$
\mX_\Theta (\vb,\va) 
= 
\sum_{x\in X, \; x = \va\cdot x = \vb\cdot x} 
\frac{\theta^x(\vb, \va^{-1})}{\theta^x (\va^{-1},\vb)}
=
\sum_{x\in X, \; x = \va\cdot x = \vb\cdot x} 
%\mu_x(\va,\va^{-1})^{-1} \mu_x(1,1)^{-1}
\frac{\theta^x(\vb,\va^{-1})\theta^x(\va,\vb\va^{-1})}{\theta^x(\va,\va^{-1}) \theta^x(1,1)}
$$
for any commuting $\va,\vb \in G$. 
\label{Osor_Char}
\end{theorem}
\begin{proof}
The component $\theta^x$ is not a cocycle, in general.
Yet for the terms in the formula it works as a cocycle:
the restriction
$\theta^x\mid_{<\va,\vb>}$
is a cocycle on $<\va,\vb>$ since $x = \va\cdot x = \vb\cdot x$.
Thus, the second and the third expressions are equal.

Since $\mX_{\Theta\boxplus \Psi} (\vb,\va) =
\mX_\Theta (\vb,\va) + \mX_\Psi (\vb,\va)$
and the second expression is additive on $G$-orbits.
It suffices to prove the theorem under an assumption
that $\Theta$ is irreducible.
Without loss of generality 
$\Theta= \Psi \uparrow_P^G$ for a degree one  2-representation
of some subgroup $P$ and $X=G/P$.  
Let 
$\mu \in Z^2(P, \bK^{\times})$
a cocycle such that 
$\{\Psi\} = [\mu]$.
A right transversal $T$ (with $\vt_0=1$)
to $P$ in $G$
is in natural bijection with $X$ via $\vt\mapsto \vt^{-1}P$.
We use $T$ and $\mu$ to decorate $X$ with cocycles: 
$$
\mu_{\vt^{-1}P} \in Z^2(\vt^{-1}P\vt, \bK^{\times}), \ \ \
\mu_{\vt^{-1}P} (\vg,\vh) \coloneqq \mu (\,^{\vt}\vg,\,^{\vt}\vh).
$$

By Corollary~\ref{cor1}, 
$$
\mX_\Theta(\vb,\va)
= 
\sum_{\vt\in T, \; \,^{\vt}\va, \,^{\vt}\vb \in P} 
\frac{\mu(\,^{\vt}\vb, (\,^{\vt}\va)^{-1})}{ \mu ((\,^\vt\va)^{-1},
  \,^\vt\vb)}
= 
\sum_{\vt\in T, \; \va\vt^{-1}P=\vb\vt^{-1}P=\vt^{-1}P} \ 
\frac{\mu_{\vt^{-1}P}(\vb, \va^{-1})}{ \mu_{\vt^{-1}P} (\va^{-1},\vb)}
%=
%\sum_{\vt\in T, \; P\vt = P\vt\va = P\vt\vb} 
%\frac{\nu(\,^{\vt}\vb, (\,^{\vt}\va)^{-1})}{ \nu ((\,^\vt\va)^{-1},
%  \,^\vt\vb)} 
\ \ .
$$
The cohomology classes of the 
cocycles $\mu_{\vt^{-1}P}$ and $\theta$
are related via Shapiro isomorphisms with different subgroups:
$[\mu_{\vt^{-1}P}]= \phi_{\,\vt^{-1}P\vt} ([\theta])$.
Each term of the last sum
depends only on cohomology class $[\mu_{\vt^{-1}P}]$.
Hence, 
we may assume that
$\mu_{\vt^{-1}P}= \phi_{\,\vt^{-1}P\vt} (\theta)$
without loss of generality.
The condition
$\va, \vb \in \vt^{-1}P\vt$
%$P\vt = P\vt\va = P\vt\vb$ 
ensures that
$$
\mu_{\vt^{-1}P} (\va,\vb)
=\phi_{\,\vt^{-1}P\vt} (\theta) (\va,\vb) 
=\theta (\va,\vb) (\vt^{-1}) 
= \theta^{\,\vt^{-1}P} (\va,\vb)
$$
facilitating the last in the proof:
%$$
%\vt\va\vb = (\,^{\vt}\va) \vt \vb = (\,^{vt}\va) (\,^{\vt}\vb) \vt.
%$$
$$
\mX_\Theta(\vb,\va)
= 
\sum_{\vt\in T, \; \va\vt^{-1}P=\vb\vt^{-1}P=\vt^{-1}P} \ 
\frac{\theta^{\vt^{-1}P}(\vb, \va^{-1})}{ \theta^{\vt^{-1}P} (\va^{-1},\vb)}
=
\sum_{x\in X, \; \va\cdot x = \vb\cdot x = x} 
\frac{\theta^{x}(\vb, \va^{-1})}{ \theta^{x} (\va^{-1},\vb)}
\ \ .
$$
%From Proposition \ref{prop1}, $c$ is equivalent to $\mu \in H^2(P, \mathbb{C}^{\times})$ by $c(\alpha,\beta)(t_i) = \mu(\alpha',\beta')$ where $t_i \alpha \beta = \alpha' t_j \beta = \alpha' \beta' t_k$ with $\alpha',\beta' \in P$. Therefore confirming that
%\[
%\mX(a,b) = \sum_{\,^{t_i}a, \,^{t_i}b \in P} \frac{\mu(\,^{t_i}b, (\,^{t_i}a)^{-1})\mu(\,^{t_i}a, (\,^{t_i}b)(\,^{t_i}a)^{-1})}{\mu(\,^{t_i}a, (\,^{t_i}a)^{-1}) \mu(1,1)}.
%\]
%is equivalent to $\,^{t_i}a, \,^{t_i}b \in P$.
\end{proof}

To facilitate further development 
we would like to formulate
several questions (conjectures).

\begin{con}
Let $\sK= (H\xrightarrow{\partial}G)$ 
be a crossed module with finite the fundamental group $\pi_1 (\sK)$.
Then there exists an Osorno formula for the value
$\mX_\Theta (\vb,\va, h)$ of the Ganter-Kapranov 2-character.
\end{con}

In what cohomological terms would we expect the formula to play out?
Group cohomology can be read off from the classifying space.
Let $\sX$ be a classifying space for $\sK$.
A 2-representation $\Theta$ of $\widetilde{\sK}$
comes with a canonical $\pi_1(\sK)$-set $X$.
The permutation representation $(\bK^\times)^X$ of $\pi_1(\sK)$
defines a local system $\underline{(\bK^\times)^X}$ on $\sX$.
We expect it to play a crucial role.

\begin{con}
If $\pi_1(\sK)$ is finite, then
2-representations of $\widetilde{\sK}$
are classified by pair $(X,[\mu])$
where $X$ is a finite $\pi_1(\sK)$-set 
and 
$[\mu]$ is a cohomology class in $H^2(\sX, \underline{(\bK^\times)^X})$
where
$\sX$ is a classifying space of $\sK$.
\end{con}

\begin{center}
\section{Examples of 2-representations}
\label{s7}
\end{center}

2-Representations are ubiquitous, yet mathematicians do not recognise
them when they see them. We would like to give two well-known
modern mathematical themes where they play crucial role, yet the
statements are not formulated in the language of 2-representations.

The first story is Lusztig's conjectures \cite{Lus}.
Let $\fg$ be a simple finite dimensional complex Lie algebra, 
$\chi\in\fg^\ast$ its nilpotent character. The finite group $G$
attached to $\chi$ is the component group of the stabiliser of $\chi$
(or the fundamental group of the coadjoint orbit of $\chi$:
they are naturally isomorphic). 

Now there are two 2-representations of $G$ over $\bC$ that appear in
nature. The first 2-representation $\Theta_{\Lie} (\chi)$ comes from the action
of $G$ on the semisimplification $U_\chi^0/\Rad (U_\chi^0)$
of the generic block $U_\chi^0$ of the reduced enveloping
algebra. It is a 2-representation over an algebraically closed field
$\bK$
of prime characteristic $p$ that needs to be larger  than the Coxeter
number of $\fg$. For such $p$ we have a canonical isomorphism
$H^2(G,\bK^\times)\cong H^2(G,\bC^\times)$ that leads to 
the  2-representation $\Theta_{\Lie} (\chi)$ over $\bC$. It appears that this
2-representation
may depend on characteristic $p$. Independence of $p$ can be
established
by the methods developed by  Bezrukavnikov and  Mirkovi\'c
\cite{BeMi},
although there is no relevant result explicitly in the paper. 

The second 2-representation $\Theta_{\Cox} (\chi)$ comes from geometry
of the double cell of the Langlands dual affine Weyl group
that corresponds to $\chi$ 
\cite{LusCell}.
Bezrukavnikov and Ostrik extract a decorated $G$-set from this geometry
\cite{Bezruk01}, thus, a 2-representation of $G$ in our terminology.  
The following conjecture by
Gunnells, Rose and Rumynin \cite{GRR11} is itself
a reformulation of a conjecture by Lusztig \cite{Lus}:

\begin{con} (2-Lusztig Conjecture)
The 2-representations $\Theta_{\Lie} (\chi)$ and $\Theta_{\Cox} (\chi)$ constructed above
are equivalent. 
Moreover, their cohomology class $[\mu]\in H^2 (G, (\bC^\times)^X)$ is trivial.
\end{con}

The second story is McKay's conjecture.
One issue with proving it is that even if you know it holds for
composition factors, it is not known how to deduce that the group
itself
satisfies it.
Isaacs, Malle and Navarro suggest a stronger condition
of McKay-goodness that is inherited by a group from its composition
factors \cite{IMN}. 
It appears that McKay goodness is a 2-representation-theoretic condition.

Let $H$ be a finite group, $p$ a fixed prime.
Let $Z$ be the centre of $H$,
$P$ a Sylow $p$-subgroup of $H$,
$N$ the normaliser of $P$.
The group $G$ of interest for us is the group of such automorphisms
$\varphi : G \rightarrow G$ 
that
$\varphi (P)=P$ and 
$\varphi|_Z = \mbox{Id}_Z$.
Let $\bC H_{p^\prime}$ be the direct summand of $\bC H$ consisting of
those
matrix algebras whose size is not divisible $p$. It is a 
$\bC Z$-algebra under the natural map 
$\bC Z\rightarrow \bC H \rightarrow \bC H_{p^\prime}$. It also has a
$\bC Z$-linear action of $G$. 
Likewise, there is a
$\bC Z$-linear action of $G$ on $\bC N_{p^\prime}$.
For each character $\chi: Z \rightarrow
\bC^\times$
we get two 2-representations 
$$
\Theta_H (\chi) \coloneqq 
\Theta_{\bC H_{p^\prime} \otimes_{\bC Z} \bC (\chi)}, \ \ \
\Theta_N (\chi) \coloneqq 
\Theta_{\bC N_{p^\prime} \otimes_{\bC Z} \bC  (\chi)}.
$$
\begin{con} (2-McKay Conjecture)
For each character $\chi$
the 2-representations 
$\Theta_H (\chi)$ and
$\Theta_N (\chi)$ 
of $G$
are equivalent.
\end{con}

Both conjectures would be solved if we could compute the values of
marks on 2-representations. Indeed, two 2-representations are
equivalent 
if and only if
the values of all marks are equal
as follows from Proposition~\ref{BPsK}.
Thus, the most crucial question is 
to find efficient innovative methods for computing marks.

\begin{center}
\section*{References}
\end{center}

\end{document}